\documentclass[a4paper,leqno]{amsart}
        \title{Coarse flow spaces for relatively hyperbolic groups}
       \author{Bartels, A.}
        \email{a.bartels@wwu.de}
       \address{Westf\"alische Wilhelms-Universit\"at M\"unster\\
               Mathematisches Institut\\
               Einsteinstr.~62,
               D-48149 M\"unster, Germany} 
         \date{July 2016}

  \usepackage{amsmath}
  \usepackage{hyperref}
  \usepackage{calc}
  \usepackage{enumerate,amssymb}
  \usepackage[arrow,curve,matrix,tips,2cell]{xy}
    \SelectTips{eu}{10} \UseTips
    \UseAllTwocells
  \usepackage{tikz}
  \usepackage{pdfcolmk}
  \usepackage{wasysym}

  \DeclareMathAlphabet{\matheurm}{U}{eur}{m}{n}

  \newcommand{\IN}{\mathbb{N}}

  \newcommand{\IR}{\mathbb{R}}

  \newcommand{\IZ}{\mathbb{Z}}

  \newcommand{\cala}{\mathcal{A}}
  
  \newcommand{\calc}{\mathcal{C}}

  \newcommand{\calf}{\mathcal{F}}

  \newcommand{\caln}{\mathcal{N}}
  
  \newcommand{\calp}{\mathcal{P}}

  \newcommand{\calu}{\mathcal{U}}
  \newcommand{\calv}{\mathcal{V}}
  \newcommand{\calw}{\mathcal{W}}

  \newcommand{\bfK}{{\mathbf K}}
  \newcommand{\bfL}{{\mathbf L}}

  \newcommand{\ignore}[1]{}

  \theoremstyle{plain}
  \newtheorem{theorem}{Theorem}[section]
  \newtheorem{lemma}[theorem]{Lemma}
  \newtheorem{corollary}[theorem]{Corollary}
  \newtheorem{proposition}[theorem]{Proposition}
  \newtheorem{addendum}[theorem]{Addendum}
  
  \newtheorem*{theorem*}{Theorem}
  \newtheorem*{corollary*}{Corollary}

  \theoremstyle{definition}
  \newtheorem{definition}[theorem]{Definition}

  \newtheorem*{definition*}{Definition}

  \theoremstyle{remark}
  \newtheorem{remark}[theorem]{Remark}  
  \newtheorem*{remark*}{Remark}  
    
  \newtheorem{example}[theorem]{Example}

  \makeatletter\let\c@equation=\c@theorem\makeatother

  \newenvironment{numberlist}
   {\begin{list}{}%
     {%
       \setlength{\leftmargin}{\labelwidth+\labelsep}%
     }%
   }%
   {\end{list}}

  \DeclareMathOperator{\id}{id}
  \DeclareMathOperator{\ind}{ind}
  \DeclareMathOperator{\mesh}{mesh}
  \DeclareMathOperator{\VCyc}{VCyc}

  \newcommand{\ANR}{{\mathit{ANR}}}

  \newcommand{\CF}{{\mathit{CF}}}
  
  \newcommand{\CW}{{\mathit{CW}}}
  \newcommand{\dd}{{\partial}}
  \newcommand{\e}{{\varepsilon}}

  \newcommand{\FS}{{\mathit{FS}}}
  
  \newcommand{\intgf}[2]{\mbox{$\int_{#1} #2$}}

  \newcommand{\thick}{{\mathit{th}}}

  \newcommand{\x}{{\times}}
  
\begin{document}

  \begin{abstract}
    We introduce coarse flow spaces for relatively hyperbolic groups
    and use them  to verify a regularity condition for the action
    of relatively hyperbolic groups on their boundaries.   
    As an application the Farrell-Jones Conjecture for
    relatively hyperbolic groups can be reduced to
    the peripheral subgroups 
    (up to index $2$ overgroups in the $L$-theory case).
  \end{abstract}
  
   \maketitle

  \section*{Introduction}

  Farrell and Jones~\cite{Farrell-Jones(1986a)} used the geodesic 
  flow on closed Riemannian manifolds of negative sectional curvature
  to prove that the Whitehead group of the fundamental group of such
  manifolds vanishes.
  This method has been extremely fruitful and has been generalized 
  in many ways.

  Among the developments following~\cite{Farrell-Jones(1986a)} was the 
  formulation of what is now known as the Farrell-Jones 
  Conjecture~\cite{Farrell-Jones(1993a)}.
  This conjecture predicts that the $K$- and $L$-theory of group rings
  $R[G]$ is determined by group homology and the $K$- and $L$-theory
  of group rings of virtually cyclic subgroups.
  If the conjecture holds for a group $G$, then this often yields vanishing
  results  or computational results for Whitehead groups and the 
  manifolds structure set appearing in surgery theory.
  In particular, the Farrell-Jones Conjecture has implications for  the
  classification of higher dimensional non-simply connected  manifolds. 
  We will review the precise formulation of the conjecture in Section~\ref{sec:FJ-rel-hyp}.
  More information about the Farrell-Jones Conjecture and its 
  applications can be found for example
  in~\cite{Bartels-Lueck-Reich(2008appl),Lueck(ICM2010),Lueck-Reich(2005)}.

  In many cases it is fruitful to replace the family of virtually cyclic
  subgroups $\VCyc$ with a bigger family of subgroups $\calf$.
  There is then a formulation of the Farrell-Jones Conjecture relative
  to $\calf$.
  This version of the conjecture is particularly useful whenever the 
  groups in the family $\calf$ are already known to satisfy the
  original Farrell-Jones Conjecture.

  In work with L\"uck and  Reich 
  the geodesic flow method from~\cite{Farrell-Jones(1986a)}
  has been successfully implemented
  in~\cite{Bartels-Lueck(2012annals), Bartels-Lueck-Reich(2008cover),
  Bartels-Lueck-Reich(2008hyper)}  
  to prove the Farrell-Jones Conjecture for hyperbolic groups.
  More generally, the results 
  from~\cite{Bartels-Lueck(2012annals),Bartels-Lueck-Reich(2008hyper)}    
  state that the Farrell-Jones Conjecture for a group $G$ holds
  relative to a family $\calf$ whenever
  there exists an action of $G$ on a finite dimensional
  contractible $\ANR$ satisfying a regularity condition
  relative to the family $\calf$.
  We will review this regularity condition shortly and refer to
  actions satisfying it relative to $\calf$ as
  finitely $\calf$-amenable actions. 
  In this language the main result of~\cite{Bartels-Lueck-Reich(2008cover)}
  implies that for hyperbolic groups the action on the 
  boundary is finitely $\VCyc$-amenable.
  In this paper we prove a similar result for
  relatively hyperbolic groups.
  For definitions of relatively hyperbolic groups 
  see~\cite{Farb(1998),Gromov(1987),Gromov(1993),Szczepanski(Rel-hyp-groups)}.
  We will use Bowditch's point of view~\cite{Bowditch(Rel-hyperbolic-groups)}, 
  recalled in Section~\ref{sec:relative-hyperbolic-groups}.
  
  \begin{theorem*} \label{thm:rel-hyp-finitely-amenable}
    Let $G$ be a countable group that is relatively hyperbolic to
    subgroups $P_1,\dots,P_n$.
    Let $\calp$ be the family of subgroups of $G$ that are either
    virtually cyclic or subconjugated to one of the $P_i$.
    Then the action of $G$ on its boundary $\Delta$ is
    finitely $\calp$-amenable.
  \end{theorem*}

  \noindent
  This result appears as Theorem~\ref{thm:long-wide-GxDelta} in
  Section~\ref{sec:coarse-flow-space}.

  The boundary $\Delta$ is usually neither contractible nor an $\ANR$.
  For hyperbolic groups Bestvina-Mess proved that 
  the union of the Rips complex with the boundary
  is a contractible compact $\ANR$~\cite{Bestvina-Mess(1991)}.
  Similarly, for relatively hyperbolic groups there is a relative 
  Rips complex~\cite{Dahmani(2003class-spaces+boundaries-rel-hyp),
    Mineyev-Yaman(rel-hyperbolic-bounded-cohom)} 
  and in the appendix we extend the Bestvina-Mess result to the relatively 
  hyperbolic case. 
  This is closely related to results of 
  Dahmani~\cite{Dahmani(2003class-spaces+boundaries-rel-hyp)}.  
  Using~\cite{Bartels-Lueck(2012annals),Bartels-Lueck-Reich(2008hyper)}
  we will obtain the following application to the Farrell-Jones Conjecture.
  
  \begin{corollary*}
    Let $G$ be a countable group that is relatively hyperbolic to
    subgroups $P_1,\dots,P_n$.
    If $P_1,\dots,P_n$ satisfy the Farrell-Jones Conjecture then
    $G$ satisfies the Farrell-Jones Conjecture.
  \end{corollary*}

  This corollary is more carefully formulated as
  Corollary~\ref{cor:FJ-rel-hyp}. 
  Such results are typical for relatively hyperbolic groups.
  For example if a group $G$ is relatively hyperbolic to groups of 
  finite asymptotic dimension, then $G$ is of finite 
  asymptotic dimension~\cite{Osin(Asym-dim-rel-hyp-groups)}.
  
  For groups that are relatively hyperbolic to groups that satisfy the Farrell-Jones 
  Conjecture and are in addition residually finite 
  Antol\'in, Coulon and Gandini~\cite{Antolin-Coulon-Gandini(FJ-via-Dehn)}  
  give an alternative proof of the above corollary using Dehn fillings.

  \subsection*{A regularity condition}

  Let $X$ be a $G$-space and $\calf$ be a family of subgroups of $G$.
  An open subset $U \subseteq X$ is said to be an $\calf$-subset if
  there is $F \in \calf$ such that $g U = U$ for $g \in F$ and
  $g U \cap U = \emptyset$ if $g \not\in F$.
  A cover $\calu$ of open subsets of $X$ is said to be $G$-invariant if
  $gU \in \calu$ for all $g \in G$, $U \in \calu$.
  A $G$-invariant cover $\calu$ of $X$ is said to be an $\calf$-cover if
  the members of $\calu$ are all $\calf$-subsets. 
  The order of a collection $\calu$ of subsets of $X$ is $\leq N$ if
  each $x \in X$ is contained in at most $N+1$-members of $\calu$. 
  If $\calu$ is a cover, then we will often call the order of $\calu$ 
  the dimension of $\calu$. 

  \begin{definition} \label{def:N-calf-amenable}
    Let $G$ be a group and $\calf$ be a family of subgroups.
    An action of $G$ on a space $X$ is said to be $N$-$\calf$-amenable
    if for any finite subset $S$ of $G$ there exists an open $\calf$-cover 
    $\calu$ of
    $G \x X$ (equipped with the diagonal $G$-action) with the following
    properties:
    \begin{enumerate}
    \item \label{def:N-calf-amenable:dim}
       the dimension of $\calu$ is at most $N$;
    \item \label{def:N-calf-amenable:wide}
       for all $x \in X$ there is $U \in \calu$ with
       $S \x \{ x \} \subseteq U$.
    \end{enumerate}
    An action that is $N$-$\calf$-amenable for some $N$ 
    is said to be finitely $\calf$-amenable.
  \end{definition}
 
  \noindent
  In~\cite{Bartels-Lueck-Reich(2008hyper)} such cover were called \emph{wide}.

  \begin{remark}
    \label{rem:amenable-action-isotropy}
    Suppose that the action $G \curvearrowright X$
    is $N$-$\calf$-amenable. 
    Then all finitely generated subgroups $H$ of $G$ that
    fix a point $x \in X$ belong to $\calf$.
    Indeed, if $S$ is symmetric, contains $e$, generates $H$
    and satisfies $S \x \{ x \} \subseteq U$ 
    then $(e,x) \in sU$ for all $s \in S$ and therefore
    $U \cap sU \neq \emptyset$ for all $s \in S$.
  \end{remark}  

  \begin{remark} \label{rem:equivariant-cover} 
     Suppose that the action $G \curvearrowright X$ is such  
     that there exists an $\calf$-cover $\calv$ for
     $X$ of dimension $N$. 
     Then the action  is $N$-$\calf$-amenable.
     Indeed, if we set $\calu := \{ G \x V \mid V \in \calv \}$ 
     then $\calu$ is an $\calf$-cover for $G \x X$ of dimension 
     $N$ and $G \x \{x\} \subseteq U = G \x V \in \calv$ 
     whenever $x \in V \in \calv$. 
     Such a cover $\calv$ exists for example for any
     cellular action on an $N$-dimensional $\CW$-complex
     whenever all isotropy groups of the action belong to $\calf$
  \end{remark}

  \begin{remark}
    \label{rem:almost-equivariant-maps}
    Suppose that $X$ is compact and metrizable.  
    Then given an $N$-$\calf$-amenable action of a countable group $G$ on $X$
    there is a sequence of $G$-equivariant maps $f_n \colon G \x X \to K_n$
    with the following properties.
    The space $K_n$ is an $N$-dimensional simplicial complex with a
    simplicial $G$-action. 
    The isotropy groups for this action belong to $\calf$.
    The maps $f_n$ are contracting in the $G$-direction:
    for any $g \in G$ we have
    $\sup_{x \in X} \| f_n(e,x) - f_n(g,x) \|_1 \to 0$ as $n \to \infty$.
    The maps $f_n$ can be contructed using a partition of unity
    subordinated to the covers $\calu$ appearing in the definition of
    $N$-$\calf$-amenability, 
    compare~\cite[Sec.~4 and 5]{Bartels-Lueck-Reich(2008hyper)}.
    The maps $f_n(e,-) \colon X \to K_n$ are then almost $G$-equivariant 
    in the following sense: for any $g \in G$ we have 
    $\sup_{x \in X} \| g f_n(e,x) - f_n(e,gx) \|_1 \to 0$ as $n \to \infty$.    
    The action of $G$ on $X$ is amenable~\cite{Ozawa(2006amenact)} 
    if there is such an almost equivariant sequence of maps to
    the space of probability measures on $G$.
    The isotropy groups for the action on the space of probability measures
    are the finite subgroups of $G$.
    Thus $N$-$\calf$-amenability can be thought of 
    as both stronger and weaker than amenabilty:
    stronger since a finite dimensional target is required; weaker since
    the target may have larger isotropy groups. 
    
    Ozawa~\cite{Ozawa(2006boundary)} investigated amenable actions for 
    relatively hyperbolic groups.
    In particular, his results imply that if $G$ is 
    relatively hyperbolic to amenable groups, then the
    action of $G$ on the boundary is amenable. 
  \end{remark}

  \subsection*{Flow spaces}

  If $G$ is the fundamental group of a negatively curved manifold $M$, then
  the geodesic flow is a flow on the unit sphere bundle $SM$ of the tangent 
  bundle of $M$.
  We will say that $SM$ together with the geodesic flow is the 
  geodesic flow space for $G$.
  Let $G$ now be a hyperbolic group.
  Mineyev~\cite{Mineyev(2005)} constructed an 
  analog of the geodesic flow space and this flow space $\FS$ and its dynamic 
  properties are key ingredients to the proof of
  finite $\VCyc$-amenability for the action of $G$ on its boundary 
  in~\cite{Bartels-Lueck-Reich(2008cover)}.
  The proof has naturally two parts.
  In the first part the so called \emph{long and thin covers} of $\FS$ 
  are constructed.
  In the second part the dynamic of the flow is used to construct
  maps $G \x \dd G \to \FS$ under which the long and thin covers 
  of $\FS$ pull back to the necessary covers of $G \x \dd G$. 

  The key property of the long and thin covers 
  $\calu_\alpha$ is that they are  long in the direction 
  of the flow: for each $x \in \FS$ there is $U \in \calu_\alpha$ such 
  that $x$ stays in $U$ for time $t \in [-\alpha, \alpha]$.
  Typically the members of $\calu$ are very thin transverse to the flow 
  -- thus the name long and thin covers.
  These covers are a variation of the long and thin cell structures 
  appearing in~\cite{Farrell-Jones(1986a)}.
  The construction of these long and thin 
  covers in~\cite{Bartels-Lueck-Reich(2008cover)} 
  is quite involved but works for very general flow spaces.
  Moreover, assumptions on the order of finite subgroups of $G$
  and the structure of periodic orbits were later shown to be 
  not necessary by Mole-R\"uping~\cite{Mole-Rueping(EquivRefine)}
  and by Kasprowski-R\"uping~\cite{Kasprowski-Rueping(long-and-thin)}.
  Sauer~\cite{Sauer(AmenableCovers)} used packing methods to 
  prove in a measure theoretic context results that are similar
  to long and thin covers.
  Later he pointed out that such packing methods should also be applicable to
  the construction of long and thin covers of flow spaces.
  This led to a much simpler construction for the long and thin covers
  from~\cite{Bartels-Lueck-Reich(2008cover)} by 
  Kasprowski-R\"uping~\cite{Kasprowski-Rueping(long-and-thin)}.
  
  \subsection*{Coarse flow spaces}

  An observation of the present paper is that the construction
  of long and thin covers using the packing method works in
  a more general context than flow spaces.
  This can be used to give an alternative argument for the
  finitely $\VCyc$-amenability of the actions of hyperbolic groups 
  on their boundaries that avoids Mineyev's flow space. 
  Moreover, this alternative argument generalizes to relatively 
  hyperbolic groups\footnote{It is plausible that the argument 
   from~\cite{Bartels-Lueck-Reich(2008cover)} can also be extended to
   relatively hyperbolic groups. A step in this direction
   is~\cite{Mole(2013)}.}.
  In this and the next subsection  
  we outline this argument for hyperbolic groups.
  The case of relatively hyperbolic groups is treated in detail in the main 
  text of this paper.

  Let $G$ be a hyperbolic group.
  We will replace Mineyev's flow space with
  a more easily defined \emph{coarse flow space}. 
  Let $\Gamma$ be a Cayley graph for $G$.
  Assume that $\Gamma$ is $\delta$-hyperbolic.
  Let $\overline{G} := G \cup \dd G$
  and $Z :=  \overline{G}^2$. 

  \begin{definition} \label{def:coarse-flow-space-hyp}
    The coarse flow space $\CF$ for $G$ is the 
    subspace of $G \x Z$ consisting of triples
    $(g,\xi_-,\xi_+)$ such that there exists a geodesic in $\Gamma$ from
    $\xi_-$ to $\xi_+$ that passes $g$ within distance $\delta$. 
  \end{definition}

  There is no actual flow on this coarse flow space, but there are 
  natural analoga of the orbits of the flow on $\FS$.
  These analoga are the subsets
  $G_{\xi_-,\xi_+} := \{ g \in G \mid (g,\xi_-,\xi_+) \in \CF \} \subseteq G$. 
  Since $\Gamma$ is hyperbolic each $G_{\xi_-,\xi_+}$ is contained in
  a uniformly bounded neighborhood of a geodesic.   
  Consequently, the $G_{\xi_-,\xi_+}$ satisfy a uniform 
  doubling property: 
  there is $D$ such that for any $R$ and any 
  subset $S$ of a $2R$-ball  in $G_{\xi_-,\xi_+}$ the following holds:
  if $S$ is $R$-separated, i.e., $d_G(s,s') \geq R$ for all $s \neq s' \in S$,
  then the cardinality of $S$ is at most $D$. 
  This observation is the main ingredient for the following version of
  long and thin covers for $\CF$.
  Let $d_G$ be a left-invariant word metric on $G$.

  \begin{proposition} \label{prop:long-thin-hyp}
    There is $N$ such that for any $\alpha > 0$ there 
    exists an $\VCyc$-cover $\calw$ of $\CF$ such that
    the following holds:
    \begin{enumerate}
    \item \label{prop:long-thin-hyp:long} 
      for any $(g,\xi_-,\xi_+) \in \CF$ there is $W \in \calw$
      such that $B_\alpha(g)  
                   \x \{ (\xi_-,\xi_+) \} \cap \CF \subseteq W$;   
    \item \label{prop:long-thin-hyp:dim} 
      the dimension of $\calw$ is at most $N$.
    \end{enumerate}
  \end{proposition}
  
  In Theorem~\ref{thm:long-thin-cover-X-V-version} we prove a version 
  of this result in a more general situation that will also be applicable to
  the coarse flow spaces for relatively hyperbolic groups introduced 
  in Definition~\ref{def:coarse-FS-Theta}.
  An important assumption is again a uniform doubling property.
  As an application of Theorem~\ref{thm:long-thin-cover-X-V-version}
  we obtain a version of Proposition~\ref{prop:long-thin-hyp} for 
  relatively hyperbolic groups in Proposition~\ref{prop:cover-CF(Theta)}.
  In a different direction a corollary to 
  Theorem~\ref{thm:long-thin-cover-X-V-version} 
  is that all actions of finitely generated 
  virtually nilpotent groups on finite dimensional normal 
  separable spaces with isotropy in $\calf$ are finitely $\calf$-amenable,
  see Corollary~\ref{cor:nilpotent}.
  This generalizes a result for free actions of nilpotent groups by Szab{\'o}-Wu-Zacharis~\cite{Szabo-Wu-Zacharias-Rohklin-res-finite}.
    
  \begin{proof}[Sketch of proof of Proposition~\ref{prop:long-thin-hyp} 
           if $\dim \dd G = 0$]
    Since $\dim \overline{G}^{2} = 0$ there is a basis of the topology
    of $\overline{G}^{2}$ consisting of sets that are open and closed.
    Choose a covering of $\CF$ by sets of the form
    $\{ g_i \} \x V_i \cap \CF$, $i \in \IN$ where $V_i$ is open and closed
    in $\overline{G}^2$.
    Then define inductively $U_i$ by $U_0 :=  V_0$ and
    \begin{equation*}
      U_i := V_i \setminus \bigcup_{j} U_j
    \end{equation*}
    where the union is over all $j$ with $j < i$ and 
    $d_G(g_i,g_j) \leq \alpha$.
    Since the $V_i$ are closed and open, the $U_i$ are still open.
    The open sets
    $W_i := B_{2\alpha}(g_i) \x U_i \cap \CF$ form then the desired cover
    $\calw$ of $\CF$. 
    This cover is $\alpha$-long in the direction of $G$, 
    more or less by construction.
    To compute the dimension of $\calw$ one checks that if 
    $W_{i_1} \cap \dots \cap W_{i_N} \neq \emptyset$, then
    the $g_i$ form an $\alpha$-separated set in a ball of radius
    $2\alpha$ in one of the $G_{\xi_-,\xi_+}$ and therefore
    $\dim \calw \leq D-1$.
    
    In this sketch we ignored the action of $G$ on $\CF$.
    To amend this one has to choose the $V_i$ sufficiently small in 
    order to avoid intersections $g W_i \cap W_i$ for to many $g$.
    Moreover, in the definition of $U_i$ and $W_i$ the group action has 
    to be taken into account.
    In order to extend the argument to the case
    $\dim \dd G > 0$  an induction over subspaces of $\CF$ 
    of lower dimension can be used.
  \end{proof}

  \subsection*{Pulling back long and thin covers from $\CF$ to $G \x \dd G$.}  

  Let $\CF$ be the coarse flow space for the hyperbolic group $G$.
  For $W \subseteq \CF$ and $\tau > 0$ we define 
  $\iota^{-\tau}W \subseteq G \x \dd G$ to consist of all pairs $(g,\xi)$
  with the following property.
  If $v \in G$ belongs to a geodesic from $g$ to $\xi$ and satisfies
  $d_G(g,v) = \tau$, then $(v,g,\xi) \in W$.
  
  One way to think about $\iota^{-\tau} W$ is as follows:
  First define $\iota \colon G \x \dd G \to \CF$ by 
  $\iota(g,\xi) = (g,g,\xi)$.
  Next apply a partially defined multi-valued geodesic flow $\phi_\tau$ on $\CF$:
  this flow takes $(g,g,\xi)$ to the set of all $(g,v,\xi)$ where $v \in G$
  belongs to a geodesic between $g$ and $\xi$ and is of distance $\tau$ 
  from $g$. 
  Then $\iota^{-\tau} W$ is the pull-back of $W$ under the composition 
  $\phi_\tau \circ \iota$.

  This construction allows us to use the long thin covers of $\CF$ from
  Proposition~\ref{prop:long-thin-hyp} to prove that the action of
  $G$ on $\dd G$ is finitely $\VCyc$-amenable:
  If $\calw$ is a $\VCyc$-cover of $\CF$ then the same holds
  for $\iota^{-\tau} \calw := \{ \iota^{-\tau} W \mid W \in \calw \}$. 
  By construction
  $\dim \iota^{-\tau} \calw \leq \dim \calw$. 
  Finally, if $\calw$ is a 
  long cover of $\CF$ (as in  
  Proposition~\ref{prop:long-thin-hyp}~\ref{prop:long-thin-hyp:long})
  then for sufficiently large $\tau$ the cover  $\iota^{-\tau} \calw$ 
  of $G \x \dd G$ is wide in the $G$-direction (as
  in Definition~\ref{def:N-calf-amenable}~\ref{def:N-calf-amenable:wide}). 
  This last fact can be thought of as a consequence of dynamic 
  properties of $\phi_\tau$ and uses the hyperbolicity of $\Gamma$.

  For relatively hyperbolic groups this step is carried out in detail in 
  Section~\ref{sec:coarse-flow-space}.
  The main additional difficulty appearing is discussed in the next subsection. 
 
  \subsection*{Relatively hyperbolic groups}
  The precise definitions for relatively hyperbolic groups that we use will 
  be given
  in Section~\ref{sec:relative-hyperbolic-groups} and mostly follows
  Bowditch~\cite{Bowditch(Rel-hyperbolic-groups)}.
  Let $G$ be relatively hyperbolic to the peripheral subgroups $P_1,\dots,P_n$.
  By definition $G$ acts on a hyperbolic graph $\Gamma$.
  Unlike the Cayley graph for hyperbolic groups $\Gamma$ will contain vertices
  of infinite valency and the isotropy groups of these vertices will 
  be conjugated to the $P_i$. 
  We write $V$ for the set of vertices of $\Gamma$ and 
  $V_\infty$ for the set of vertices of $\Gamma$ of infinite valency.
  The boundary of $G$ is defined by Bowditch as the union of
  $\Delta := \dd \Gamma \cup V_\infty$; this is a compact space.
  
  The key additional property of the graph $\Gamma$ used here is fineness, 
  as introduced 
  by Bowditch~\cite{Bowditch(Rel-hyperbolic-groups)}.
  Under a mild additional assumption
  this property can be used to define a proper metric on the set of
  edges of $\Gamma$, see~\cite{Mineyev-Yaman(rel-hyperbolic-bounded-cohom)}. 
  In particular, it is possible to measure angles in $\Gamma$.
  Here an angle is a pair of edges that share a vertex.
  In order to allow for peripheral subgroups that are not necessarily 
  finitely generated it is better to avoid the additional assumption.
  To this end we take a slightly different point of view and consider
  $G$-invariant $G$-cofinite subsets $\Theta$ of the set of all angles.
  Such a subset will be called a \emph{size for angles}.
  Each size for angles $\Theta$ is then locally finite in the following sense: 
  for each edge $e$ there are only finitely many edges $e'$ such that 
  $(e,e') \in \Theta$.
  Fineness of $\Gamma$ implies that the set of all angles appearing in 
  any non-degenerate
  geodesic triangle is such a size for angles.

  For any size for angles $\Theta$ we define a coarse flow space $\CF(\Theta)$.
  Its definition is similar to the hyperbolic case in 
  Definition~\ref{def:coarse-flow-space-hyp}, where we replace $G$ with
  vertices of finite valency and only use geodesics along which all
  angles are $\Theta$-small.
  The argument outlined in the hyperbolic case above can then be used to
  produce wide covers of a certain subspace of $G \x \Delta$, see 
  Proposition~\ref{prop:cover-G-x-Theta-dd}.
  In order to prove that the action of $G$ on $\Delta$ is finitely 
  $\calp$-amenable, we need to
  extend these wide covers to all of $G \x \Delta$.
  This is done by an explicit construction in
  Proposition~\ref{prop:cover-G-x-V_infty}.
  
  \subsection*{Acknowledgements}
 
  I thank Adam Mole for many discussions about relatively hyperbolic groups,
  Roman Sauer and Daniel Kasprowski for discussions about 
  long and thin covers and packing methods,
  Brian Bowditch and Fran\c{c}ois Dahmani for helpful emails about
  relatively hyperbolic groups and their boundaries.
  Comments by  
  Svenja Knopf, Timm von Puttkamer and a referee 
  greatly improved this paper. 
  This work has been supported by the SFB 878 in M\"unster.
  
  \section{Long thin covers of subspaces of $V \x Z$}  
     \label{sec:covering-V-x-Z}

  \noindent
  Throughout this section we fix
  \begin{itemize}
  \item a group $G$;
  \item a family $\calf$ of subgroups of $G$;
  \item a discrete countable proper 
     metric space $V$ with a proper isometric $G$-action,
     the metric of $V$ will be denoted by $d_V$;
     we allow distances for $d_V$ to be $\infty$;
  \item a separable metrizable space $Z$ with an action of $G$ by
    homeomorphisms;
  \item a closed $G$-invariant subspace $X$ of $V \x Z$;
    we will always use the diagonal action of $G$ on $V \x Z$. 
  \end{itemize}  

  An example for $X$ is the coarse flow space for hyperbolic groups 
  from Definition~\ref{def:coarse-flow-space-hyp}.
  Let $\alpha > 0$. 
  We write $B_\alpha(v) := \{ w \in V \mid d_V(v,w) \leq \alpha \}$
  for the closed $\alpha$-ball around $v$.
  A subset $S \subseteq V$ is said to be \emph{$\alpha$-separated} if
  $d_V(s,s') > \alpha$ for any two distinct elements of $S$.
  We will say that a subset $V_0$ of $V$ has the 
  \emph{$(D,R)$-doubling property} 
  if  for any $\alpha \geq R$ 
  the following holds: if $S \subseteq V_0$ is 
  $\alpha$-separated and contained in a ball of radius $2\alpha$, then the
  cardinality of $S$ is at most $D$.  

  For $v \in V$ we set $Z_v := \{ z \in Z \mid (v,z) \in X \}$.
  This is a closed subset of $Z$.
  For $z \in Z$ we set $V_z := \{ v \in V \mid (v,z) \in X \} \subseteq V$.
  Then 
  \begin{equation*}
    X = \bigcup_{v \in V} \{ v \} \x Z_v = \bigcup_{z \in Z} V_z \x \{ z \}.
  \end{equation*}
 
  \noindent 
  The following is the most abstract result about long thin covers 
  in this paper.

  \begin{theorem}
    \label{thm:long-thin-cover-X-V-version}
    Assume that $X$ satisfies the following assumptions.
    \begin{enumerate}[\;\;\; A)] 
     \item \label{thm:long-thin-V:assm-dim} 
        $X$ is finite dimensional;
     \item \label{thm:long-thin-V:assm-packing} 
        there are $D > 0$, $R \geq 0$ such that for all $z \in Z$
        the subspace $V_z$ of $V$ 
        has the $(D,R)$-doubling property;
     \item \label{thm:long-thin-V:isotropy}
        for each $(v,z) \in X$ the isotropy group
        $G_z := \{ g \in G \mid gz=z \}$ belongs to $\calf$.
     \end{enumerate}    
     Then $X$ admits long thin covers as follows:
     there is a number $N$ depending only on the
     dimension of $X$ and the doubling constant $D$,
     such that for any $\alpha > 0$ 
     there is an $\calf$-cover $\calw$ of $X$ such that
     \begin{enumerate}
     \item \label{thm:long-thin-V:dim}
         $\dim \calw \leq N$;
     \item \label{thm:long-thin-X:long}
         for any $(v,z) \in X$ there is $W \in \calw$ such that
         $B_\alpha(v) \x \{z\} \cap X \subseteq W$.
     \end{enumerate}    
  \end{theorem}

  The proof of Theorem~\ref{thm:long-thin-cover-X-V-version}
  will proceed by induction on the dimension of subspaces of $X$.
  The following proposition is the induction step.
  Since we can take $Y = X$ to start the induction
  it implies Theorem~\ref{thm:long-thin-cover-X-V-version}. 

  \begin{proposition} \label{prop:induction-V}
    Retain the assumptions of 
    Theorem~\ref{thm:long-thin-cover-X-V-version}.
    Let $Y \subseteq X$ be a non-empty $G$-invariant closed subspace.
    Assume that $\dim Y = n$.
    For any $\alpha > 0$ there is 
    a $G$-invariant collection $\calw$ of $\calf$-subsets of $X$ 
    and a $G$-invariant closed subspace $Y' \subseteq Y$ such that
    \begin{enumerate}
     \item \label{prop:ind-V:Y'}
         $\dim Y' < \dim Y$;
     \item \label{prop:ind-V:dim}
         the order of $\calw$ is at most $D-1$;
     \item \label{prop:ind-V:long}
         for any $(v,z) \in Y \setminus Y'$ there is $W \in \calw$ such that
         $B_\alpha(v) \x \{z\} \cap X \subseteq W$.
    \end{enumerate}     
  \end{proposition}

  \begin{lemma}
    \label{lem:dim-U}
    Retain the assumptions of Proposition~\ref{prop:induction-V}.
    Let $(v,z) \in Y$, $\alpha > 0$ and $U_0$ be a neighborhood of $z$ in $Z$.
    Then $U_0$ contains a smaller open neighborhood $U$ of $z$ in $Z$
    such that
    \begin{enumerate}
    \item \label{lem:dim-U:F} 
      $\{ g \in G \mid d_V(gv,v) \leq \alpha,
               U \cap gU \neq \emptyset \} \subseteq 
       G_z = \{ g \in G \mid gz=z \}$;
    \item \label{lem:dim-U:dim-dd}
      for all $w \in V$ we have $\dim Y_w \cap \dd U < \dim Y$,
      where $Y_w := \{ z \in Z_w \mid (w,z) \in Y \}$ 
      and $\dd U$ denotes the boundary of $U$ in $Z$. 
    \end{enumerate}
  \end{lemma}

  \begin{proof}    
    We have $Y = \bigcup_{w \in V} \{w\} \x Y_w$.
    Since $\dim Y = n$ and $V$ is discrete each $Y_w$ satisfies 
    $\dim Y_w \leq n$.
    Since $Y$ is closed, $Y_w \subseteq Z$ is closed for each $w \in V$.   
    The countable sum theorem in dimension 
    theory~\cite[Thm.~2.5, p.125]{Pears(Dimension-theory)} 
    asserts that in a normal space the 
    dimension of the countable union of closed subspaces is the
    supremum of the dimension of the subspace. 
    Since $V$ is countable this implies that the dimension of 
    $Y_V := \bigcup_{w \in V} Y_w$ is at most $n$.
      
    Since every subspace $Z'$ of $Z$ is separable and metrizable,
    its (covering) dimension equals its (small) inductive dimension: 
    $\dim Z' = \ind Z'$~\cite[Cor.~5.10, p.~184]{Pears(Dimension-theory)}.
    This means that for any $z \in Z'$ there are arbitrary small
    neighborhoods whose boundary in $Z'$ has dimension less than
    $\dim Z'$. 

    Since $d_V$ is proper and the action of $G$ on $V$ is proper, 
    $\{ g \mid d_V(gv,v) \leq \alpha \}$ is finite.
    Since $Z$ is Hausdorff there is an open neighborhood $U'$ of $z$ in $Z$
    that is contained in $U_0$ and satisfies
    $g U' \cap U' = \emptyset$ for all $g \in G \setminus G_z$ with
    $d_V(gv,v) \leq \alpha$. 
    Applying the facts from the preceding paragraph to $Y_V$, we obtain 
    an open neighborhood $U''$ of $z$ in $Y_V$ that is contained
    in $U'$ and satisfies $\dim \dd^{Y_V} U'' < \dim Y$, 
    where $\dd^{Y_V}$ denotes the boundary in $Y_V$.
    Using Lemma~\ref{lem:dd-U_+} we can extend $U''$
    to an open subset $U$ of $Z$ 
    that is contained in $U'$ and satisfies
    $U \cap Y_V = U''$, and $\dd U \cap Y_V = \dd^{Y_V}U''$, 
    where $\dd$ denotes the boundary in $Z$.

    Let $w \in V$. 
    Since $Y_w$ is closed in $Y_V$ we have
    $\dim Y_w \cap \dd U \leq \dim Y_V \cap \dd U 
        = \dim \dd^{Y_V}U'' < \dim Y$.
    Thus $U$ satisfies~\ref{lem:dim-U:dim-dd}.
    Since $U \subseteq U'$, $U$ also satisfies~\ref{lem:dim-U:F}.     
  \end{proof}

  \begin{lemma} \label{lem:thick-X-still-D}
    Retain the assumptions of Theorem~\ref{thm:long-thin-cover-X-V-version}.
    Let $\alpha \geq R$ be given.    
    Then there 
    exists an open $G$-invariant neighborhood $X^{\thick}$ of $X$ in 
    $V \x Z$ such that for all $z \in Z$, the subspace
    $V^{\thick}_{z} := 
     \{ v \in V \mid (v,z) \in X^{\thick} \}$ of $V$
    has the following property: if $S \subseteq V^{\thick}_z$ is $\alpha$-separated and 
    contained in a ball of radius $2\alpha$, then
    the cardinality of $S$ is at most $D$.  
  \end{lemma}

  \begin{proof}
    Since the action of $G$ on $V$ is proper, we can 
    pick metrics $d_v$, $v \in V$ on $Z$ that are 
    compatible with the $G$-action, i.e., such that we have
    $d_v(z,z') = d_{gv}(gz,gz')$ 
    for all $v \in V$, $g \in G$, $z,z' \in Z$.      
  
    Let $(v,z) \in X$  (i.e., $z \in Z_v$).
    For $w \in V$ with $z \not\in Z_w$, let
    $\e(v,w,z) := d_v(z,Z_w)/{3}$.
    The $Z_w$ are closed since $X$ is closed.
    Thus $\e(v,w,z) > 0$. 
    Let $\e(v,z) := \min \e(v,w,z)$, where the minimum is
    taken over all $w \in V$ with $d_V(v,w) \leq 4\alpha$ and $z \not\in Z_w$.
    Since $d_V$ is proper the minimum exists and is positive.
    Let $B(v,w,z)$ be the open ball of radius $\e(v,z)$ around $z$
    with respect to the metric $d_w$.
    Set $B(v,z) := \cap_{d_V(v,w) \leq 4\alpha} B(v,w,z)$.
    Since $d_V$ is proper $B(v,z)$ is open.
    Moreover, by construction,
    \begin{numberlist}
    \item[\label{nl:B(v,z),Z_w}] 
        $d_v(z,Z_w) > 2 \e(v,z)$ whenever $z \in Z_v \setminus Z_w$ and
      $d_V(v,w) \leq 4\alpha$;
    \item[\label{nl:z-z'}] 
        $d_{v}(z',z) < \e(w,z)$ whenever $z \in Z_w$, $z' \in B(w,z)$ and
        $d_V(v,w) \leq 4\alpha$.
    \end{numberlist}
    
    Define $Z_v^{\thick} := \bigcup_{z \in Z_v} B(v,z)$ and
    $X^{\thick} := \bigcup_{v \in V} \{ v \} \x Z_v^{\thick}$.
    By construction $X^{\thick}$ is an open neighborhood of $X$ in $V \x Z$.
    The compatibility of the metrics $d_v$ with the $G$-action 
    guarantees that $X^{\thick}$ is $G$-invariant.

    Let $S \subseteq V$ be an $\alpha$-separated subset of
    some $2\alpha$-ball in $V$.
    Assume that the cardinality of $S$ exceeds $D$.
    Since all $V_z$ have the $D$-doubling property we have
    $\bigcap_{w \in S} Z_w = \emptyset$.
    We need to show that $\bigcap_{w \in S} Z^{\thick}_w = \emptyset$ as well.
    Assume by contradiction that $z' \in \bigcap_{w \in S} Z^{\thick}_w$.
    Then we find for each $w \in S$, a point $z_w \in Z_w$ with
    $z' \in B(w,z_w)$. 
    Choose $v \in S$ such that $\e(w,z_w) \leq \e(v,z_{v})$ for all
    $w \in S$. 
    Note that $d_V(w,v) \leq 4\alpha$ for all $w \in S$.
    Now~\eqref{nl:z-z'} implies that for all $w \in S$
     \begin{align*}
      d_{v}(z_{v},Z_w) \leq 
           d_{v}(z_{v},z_w)   \leq d_{v}(z_{v},z') & + d_{v}(z',z_w) \\
          & \leq \e(v,z_{v}) + \e(w,z_w) \leq 2\e(v,z_{v}).
     \end{align*}
    On the other hand~\eqref{nl:B(v,z),Z_w} implies 
    $d_{v}(z_{v},Z_w) > 2 \e(v,z_{v})$ if $z_{v} \not\in Z_w$.
    Therefore $z_{v} \in Z_w$ for all $w \in S$.
    This contradicts  $\bigcap_{w \in S} Z_w = \emptyset$.
  \end{proof}

  \begin{proof}[Proof of Proposition~\ref{prop:induction-V}]
    Throughout this proof closure and boundary will always be taken with
    respect to $Z$.

    Let $\alpha \geq R$.
    Let $X^{\thick}$ be as in Lemma~\ref{lem:thick-X-still-D}.
    Since $Y$ is separable we can use Lemma~\ref{lem:dim-U}
    and find sequences $v_i \in V$, $F_i \in \calf$, 
    $U_i \subseteq Z$ open, such that
    \begin{numberlist}
    \item [\label{nl:F}] 
      $\{ g \in G \mid d_V(gv_i,v_i) \leq 4\alpha,
               U_i \cap gU_i \neq \emptyset \} \subseteq F_i$;
    \item [\label{nl:dim-dd}]
      for all $w \in V$ we have $\dim Y_w \cap \dd U_i < \dim Y$;
    \item [\label{nl:cover-Y}]
      $Y \subseteq \bigcup_{i} \{v_i\} \x U_i \subseteq X^{\thick}$. 
    \end{numberlist} 
    Now we define inductively $U'_i$ by $U'_0 := U_0$ and
    \begin{equation*}
      U'_i := U_i \setminus 
         \bigcup_{h,j} h \overline{U'_j}
    \end{equation*}
    where the union is over all pairs $(h,j)$ with $h \in G$, $j < i$, 
    and $d_V(v_i,hv_j) \leq \alpha$. 
    Since $d_V$ is proper and the action of $G$ on $V$ is proper,
    this is a finite set of pairs and $U'_i$ is open.
    An easy induction shows that 
    $\dd U'_i \subseteq \bigcup_{j \leq i, g \in G} g \dd U_j$.
    Thus with $Y'' :=  Y \cap V \x  \bigcup_{i,g} g\dd U_i$ we have
    \begin{equation*}
      Y \cap \{ v \} \x  g \dd U'_i \subseteq Y''
    \end{equation*}
    for all $v \in V$, $i \in \IN$, $g \in G$. 
    For $i \in \IN$ let
    \begin{equation*}
      W_i := B_{2\alpha}(v_i) \x U'_i \cap X
    \end{equation*}
    and set $\calw := \{ g F_i W_i \mid g \in G, i \in \IN \}$.
    
    Clearly, $\calw$ consists of open subsets of $X$ and is $G$-invariant.
    Consider $g F_i W_i \in \calw$.
    If $\gamma \in g F_i g^{-1}$, then $\gamma g F_i W_i = g F_i W_i$.
    If $\gamma \notin g F_i g^{-1}$, then  
    $\gamma g F_i W_i \cap g F_i W_i = \emptyset$.
    Indeed, if $(v,z) \in \gamma g F_i W_i \cap g F_i W_i$, 
    then there are $a,b \in F_i$ with
    $d_V(v,\gamma g a v_i), d_V(v,g b v_i) \leq 2\alpha$
    and $z \in \gamma g a U_i \cap g b U_i$.
    Then~\eqref{nl:F} implies $b^{-1}g^{-1}\gamma ga \in F_i$ and therefore
    $\gamma \in g F_i g^{-1}$.   
    Thus each $g F_i W_i$ is an $\calf$-subset.

    We now prove $\dim Y'' < \dim Y$.
    For $v \in V$, $i \in \IN$ and $g \in G$ we have
    $Y \cap \{v \} \x g \dd U_i = g( Y \cap \{ g^{-1}v \} \x \dd U_i) \cong
      Y_{g^{-1}v} \cap \dd U_i$. 
    By~\eqref{nl:dim-dd} all these spaces are of dimension $< \dim Y$.
    Since they are all closed in $Y$
    the countable sum theorem~\cite[Thm.~2.5, p.125]{Pears(Dimension-theory)}
    implies that their union, $Y''$, is also of dimension $< \dim Y$.
    
    \vspace{1ex}
    Next we prove that $\calw$ is of order $\leq D-1$, i.e., that
    it satisfies~\ref{prop:ind-V:dim}. 
    Let $(v,z) \in X$.
    Suppose that $(v,z) \in g F_i  W_i$.
    Then there is $a \in F_i$ such that
    $d_V(v, g a v_i) \leq 2 \alpha$ and $z \in  g a U'_i$. 
    As $X^{\thick}$ is $G$-invariant,~\eqref{nl:cover-Y} implies 
    $(gav_i,z) \in \{ gav_i \} \x \{ gaU'_i \} \subseteq X^{\thick}$.
    Thus $gav_i \in V^{\thick}_z$. 
    If also $(v,z) \in h F_j  W_j$ then there is $b \in F_j$ with
    $d_V(v,hbv_j) \leq 2 \alpha$, $z \in hbU'_j$ and $hbv_j \in V^{\thick}_z$.

    If $i=j$ then $g F_i  W_i \cap h F_i W_i \neq \emptyset$
    and since $g F_i W_i$ is a $\calf$-subset
    we then have $g F_i  W_i = h F_j W_j$. 
    If $i \neq j$,
    then $\gamma U'_i \cap U'_j = \emptyset$ for all 
    $\gamma \in G$ with $d_V(\gamma v_i,  v_j) \leq \alpha$.
    Since $b^{-1}h^{-1}gaU'_i \cap U'_j \neq \emptyset$, this
    implies 
    $d_V (b^{-1}h^{-1}ga v_i, v_j) = d_V( gav_i,hb v_j ) > \alpha$.
    
    All together, we have shown that if $(v,z)$ is contained in $N$ distinct 
    members in $\calw$, then there is an $\alpha$-separated set in 
    $B_{2\alpha}(v) \cap V^{\thick}_z$ of cardinality $N$.
    Lemma~\ref{lem:thick-X-still-D} implies now that $N$ is 
    bounded by the doubling constant $D$ appearing
    in the assumptions of Theorem~\ref{thm:long-thin-cover-X-V-version}.
    This implies that the order of $\calw$ is at most $D-1$.   
   
    \vspace{1ex} 
    Next we claim that for any $(v,z) \in Y \setminus Y''$ there
    is  $W \in \calw$ with
    $B_\alpha(v) \x \{ z \} \cap X \subseteq W$.  
    Let $(v,z) \in Y \setminus Y''$.
    By~\eqref{nl:cover-Y} there is $i \in \IN$ 
    with $(v,z) \in \{v_i\} \x  U_i$. 
    So $v = v_i$.
    If $z \in  U'_i$, then 
    \begin{equation*}
      B_\alpha(v) \x \{z\} \cap X \subseteq B_{2\alpha} ( v_i) \x  U'_i \cap X
        \subseteq  F_i W_i.
    \end{equation*}
    If $z \not\in  U'_i$, then $z \in h \overline U'_j$
    for some $j < i$ and $h \in G$ with $d_V(v_i,hv_j) \leq \alpha$.
    We observed earlier $Y \cap \{ v \} \x h \dd U'_j \subseteq Y''$.
    Thus we have $z \in h U'_j$. 
    Since  $d_V(v,h v_j) = d_V( v_i, h v_j)  \leq \alpha$ we have then
    \begin{equation*}
      B_\alpha(v) \x \{ z \} \cap X \subseteq B_{2\alpha} (h v_j) \x h U'_j \cap X
        \subseteq h F_j W_j.
    \end{equation*}
    This proves our claim.
     
    \vspace{1ex}
    Now observe that if $B_\alpha(v) \x \{ z \} \cap X \subseteq W$,
    then, since $W$ is open, $B_\alpha(v) \x \{ z' \} \cap X \subseteq W$
    for all $z'$ in a neighborhood of $z$. 
    Since $\calw$ is $G$-invariant, this implies that there is an 
    open $G$-invariant neighborhood $N$ of $Y \setminus Y''$
    such that all $(v,z) \in N$ satisfy~\ref{prop:ind-V:long}. 
    
    \vspace{1ex}
    Now $Y' := Y \setminus N$ is $G$-invariant and closed.     
    Moreover, since $Y' \subseteq Y''$, we have $\dim Y' \leq \dim Y'' < \dim Y$
    and~\ref{prop:ind-V:Y'} holds.
  \end{proof}

  \begin{corollary}
    \label{cor:nilpotent}
    Let $G$ be a finitely generated virtually nilpotent group.
    Then any action of $G$ on a finite dimensional separable metrizable space $Z$
    is finitely $\calf$-amenable where $\calf$ consists of all
    subgroups of $G$ that fix a point in $Z$.
  \end{corollary}

  \begin{proof}
    This follows from Theorem~\ref{thm:long-thin-cover-X-V-version} 
    where we take
    $V = G$ with a word metric and $X = G \x Z$:
    Assumption A) is satisfied since $Z$ is finitely dimensional.
    Finitely generated virtually nilpotent groups are of 
    polynomial growth~\cite{Bass(Degree-of-poly-growth)}.
    This implies that $G$ has a doubling property.
    Thus assumption B) is satisfied.    
    Assumption C) holds by choice of $\calf$.
  \end{proof}

  \section{Relative hyperbolic groups}
    \label{sec:relative-hyperbolic-groups}    

  Throughout this section $\Gamma$ will be a fine and hyperbolic 
  graph in the sense of Bowditch~\cite{Bowditch(Rel-hyperbolic-groups)}
  and $G$ will be a countable group equipped with a 
  cocompact simplicial action on $\Gamma$.
  In particular, $\Gamma$ is uniformly fine: 
  for any $\alpha$ there is $N_\alpha$
  such that for any edge $e$ there are at most $N_\alpha$ circuits of length
  $\leq \alpha$ containing $e$.
  (A circuit is an embedded loop in $\Gamma$.)
  The isotropy groups of edges for the action of $G$ on $\Gamma$ 
  are assumed to be finite.  
  Let $P_1,\dots,P_n$ be representatives of the conjugacy classes of 
  the isotropy groups of the vertices of $\Gamma$ of infinite valency.  
  We will then say that $G$ is 
   \emph{relatively hyperbolic to $P_1,\dots,P_n$}\footnote{
      Bowditch~\cite[Def.~2]{Bowditch(Rel-hyperbolic-groups)} 
      assumes in addition
      that $P_1,\dots,P_n$ are finitely generated, but this restriction will not
      be necessary here.}.

  We denote by $\Delta$ the union of the Gromov boundary $\dd \Gamma$ of
  $\Gamma$ with the set $V_\infty$ of vertices of $\Gamma$ with infinite
  valency.
  By $V$ we denote the set of all vertices of $\Gamma$ and by 
  $E$ the set of all edges of $\Gamma$.
  We write $\Delta_+ := \dd \Gamma \cup V$. 
  In other words, $\Delta_+$ is the union of $\Delta$ with the set of all vertices of $\Gamma$ of finite valency. 
  We think of edges as subsets of $V$ with two elements;
  in particular, edges are not oriented. 

  We will use the term geodesic for finite geodesics, geodesic rays and 
  bi-infinite geodesics. 
  If $c$ is a geodesic and $\xi,\xi' \in \Delta_+$ are both contained in $c$ or  
  endpoints of $c$ then we write $c|_{[\xi,\xi']}$ for the restriction of $c$ to
  a geodesic between $\xi$ and $\xi'$. 
 
  Hyperbolicity of a graph is usually formulated in terms of geodesic triangles
  whose sides are finite geodesics.
  But it then follows that all geodesic triangles, including geodesic 
  triangles with one or more corners at the boundary are uniformly 
  slim~\cite[Lem.~2.11]{Groves-Manning(Dehn-filling)}\footnote{In this reference the proof is left as an exercise; for completeness a solution to this exercise is included in Appendix~\ref{app:ideal-slim}.}.
  We can therefore fix a constant $\delta > 0$ such that in all 
  geodesic triangles each side is contained in the union of the 
  $\delta$-neighborhood of the other two.
  It will be convenient to assume that $\delta$ is an integer.
  We will refer to $\delta$ as a hyperbolicity constant for $\Gamma$.

  \subsection*{Angles.}

   An unordered pair $(e,e')$ 
  of edges in $\Gamma$ that have a vertex $v$ in 
  common is called an \emph{angle} at $v$.
  If $e=e'$, then we say that the angle $(e,e')$ is trivial.
  The group $G$ acts on the set of angles; this action will usually be 
  not cofinite.
  A $G$-invariant, $G$-cofinite subset $\Theta$ of the set of all angles,
  that contains all trivial angles will be called a
  \emph{size for angles}.
  Members of $\Theta$ will be called \emph{$\Theta$-small}.
  Angles that are not contained in $\Theta$ will 
  be called \emph{$\Theta$-large}.
  If $c$ is a geodesic in $\Gamma$ then $c$ determines a non-trivial angle at 
  every internal vertex $v$ of $c$; this angle will be called
  the angle of $c$ at $v$ and sometimes be denoted by $\varangle_v c$.
  If all these angles are $\Theta$-small, then we will say that 
  $c$ is $\Theta$-small.
  In particular, geodesics of length $1$ are always $\Theta$-small, as they
  contain no internal vertices. 
  If $\Theta$ and $\Theta'$ are two sizes for angles then we define
  $\Theta + \Theta'$ to consist of all angles $(e,e'')$ for which there
  is an edge $e'$ with $(e,e') \in \Theta$ and $(e',e'') \in \Theta'$. 
  As a consequence of the following Lemma~\ref{lem:angles-locally-finite}, 
  $\Theta + \Theta'$ is again a size for angles. 

  \begin{lemma} \label{lem:angles-locally-finite}
    If $\Theta$ is a size for angles then
    each edge $e$ is contained in only finitely many angles
    of $\Theta$, i.e., $\Theta_e := \{ e' \in E \mid (e,e') \in \Theta \}$
    is finite. 
    Conversely, a $G$-invariant set of angles $\Theta$ with the property
    that $\Theta_e$ is finite for any edge $e$ is a size for angles. 
  \end{lemma}

  \begin{proof}
    If the action of $G$ on $\Theta$ is cofinite, the same holds for the
    action of the isotropy group $G_e$ of $e$ on $\Theta_e$.
    Since $G_e$ is finite, so is then $\Theta_e$.

    Conversely, since the action of $G$ on $E$ is cofinite, there are finitely
    many edges $e_1,\dots,e_n$ such that each orbit for the action of $G$
    on the set of all angles contains an angle of the form $(e,e_i)$.
    Thus, if $\Theta_{e_i}$ is finite for $i=1,\dots,n$, then the action of $G$ on
    $\Theta$ is cofinite.
  \end{proof}

  \begin{corollary}
    \label{cor:finite-theta-balls}
    Let $\Theta$ be a size for angles and $\alpha > 0$.
    Let $v$ be a vertex and $e$ be an edge incident to $v$.
    Consider the set $V'$ of all vertices $v'$
    for which there exists a $\Theta$-small geodesic
    from $v$ to $v'$ of length at most $\alpha$ whose initial edge $e'$ 
    satisfies $(e,e') \in \Theta$.
    Then the cardinality of $V'$ is bounded by a number depending only
    on $\Theta$ and $\alpha$.
  \end{corollary}

  \begin{proof}
    This is an immediate consequence of the first statement in 
    Lemma~\ref{lem:angles-locally-finite} and the cofiniteness of
    the action of $G$ on the set of edges.
  \end{proof}

  \begin{definition}
    We define $\Theta^{(3)}$ as the set of all angles $(e,e')$ such that there 
    exists a geodesic triangle (possibly with some corners in $\dd \Gamma$)
    with sides $c$, $c'$ and $c''$ 
    such that $c$ and $c'$ determine the angle $(e,e')$ at the corner $v \in V$,
    and such that  $c''$ does not meet $v$.
  \end{definition}

  \begin{lemma} \label{lem:theta_3-cofinite}
    The set $\Theta^{(3)}$ is a size for angles.
  \end{lemma}

  \begin{proof}
    Clearly, $G$ acts on $\Theta^{(3)}$ and $\Theta^{(3)}$ contains all trivial 
    angles.
    To show that $\Theta^{(3)}$ is cofinite it suffices to show by 
    Lemma~\ref{lem:angles-locally-finite}, that for
    all $e$ there are only finitely many $e'$ with $(e,e') \in \Theta^{(3)}$.
    We will use hyperbolicity of $\Gamma$ to show
    for each $e'$ with $(e,e') \in \Theta^{(3)}$ 
    there is a circuit of uniformly bounded length 
    in $\Gamma$ that contains both $e$ and $e'$.
    Since $\Gamma$ is fine there are only finitely many such circuits.

    To construct the circuit let $c$, $c'$ and $c''$ be sides of a geodesic
    triangle such that $c$ and $c'$ determine the angle $(e,e')$ 
    at the corner $v \in V$, and such that  $c''$ does not meet $v$.
    Now pick $w \in c$ as follows.
    Recall that we picked the hyperbolicity constant $\delta$ to be an integer.
    If the length of $c$ is  $\geq 3\delta$,  then let $w \in c$ with 
    $d_\Gamma(v,w) = 3\delta$. 
    Otherwise, let $w$ be the endpoint of $c$ (not $v$).
    Similarly, pick $w' \in c'$.
    Hyperbolicity implies that there is a path of length $\leq 10 \delta$ between
    $w$ and $w'$ that misses $v$.
    Indeed, if $d_\Gamma(w,c') \leq \delta$ then we can connect $w$ first to $c'$ and
    then to $w'$ along $c'$.
    If $d_\Gamma(w',c) \leq \delta$ then we can connect $w'$ first to $c$ and then
    to $w$ along $c$.
    Otherwise, by hyperbolicity, we have $d_\Gamma(w,c'') \leq \delta$ and 
    $d_\Gamma(w',c'') \leq \delta$ and we can first connect $w$ and $w'$ to $c''$, and
    then connect along $c''$.
    Altogether we have constructed a loop of length at most $16\delta$ that meets
    $v$ exactly once in the angle $(e,e')$. 
    The loop may not be embedded, but we can shorten it to produce a circuit 
    containing the angle $(e,e')$.
  \end{proof}

  \begin{remark} \label{rem:geodesic-n-gones}
    By cutting geodesic $n$-gons in $\Gamma$ into geodesic 
    triangles one obtains a version of Lemma~\ref{lem:theta_3-cofinite}
    for geodesic $n$-gons.

    For $n \geq 3$ let $(n-2) \Theta^{(3)}$ 
    be the $n-2$-fold sum of $\Theta^{(3)}$.
    Suppose that the geodesics $c_1,\dots,c_n$ are the sides of an $n$-gon 
    in $\Gamma$ (possibly with some corners in $\dd \Gamma$).
    Let $(e,e')$ be the angle in $\Gamma$ determined by the corner $v \in V$
    of the $n$-gon between $c_1$ and $c_2$.
    If $v \notin c_3 \cup \dots \cup c_n$ then $(e,e') \in (n-2)\Theta^{(3)}$.

    Similarly, if $v$ is an internal vertex of $c_i$ that does not belong to
    any of the $c_j$, $j\neq i$, then $\varangle_v c_i \in (n-1)\Theta^{(3)}$.
  \end{remark}

  \begin{lemma} \label{lem:geodesic-2-gons}
    Let $c$ and $c'$ be two geodesics from $v \in V$ to 
    $\xi \in \Delta_+$, $\xi \neq v$.
    Let $e$ and $e'$ be the initial edges of $c$ and $c'$.
    Then $(e,e')$ is $\Theta^{(3)}$-small.
  \end{lemma}
   
  \begin{proof}
    If $e \neq e'$ then we can subdivide $c$ or $c'$ and obtain a 
    geodesic triangle for which $(e,e')$ will be the angle at $v$
    and for which $v$ does not belong to all three sides.
    Thus $(e,e') \in \Theta^{(3)}$.
  \end{proof}

  \begin{lemma} \label{lem:geodescs-between-geodesics}
    Let $c$ and $c'$ be two geodesics between $\xi_-$ and $\xi_+ \in \Delta_+$.
    Let $v \in c$ and $v' \in c'$ be such that no geodesic between $v$ and $v'$
    is contained in $c \cup c'$.
    Then there exists a $2\Theta^{(3)}$-small geodesic $\hat c$ between $v$ and $v'$.  
  \end{lemma}

  \begin{proof}
     Choose a geodesic $\hat{c}$ from $v$ to $v'$, that first travels 
     along $c$, then along a  
     geodesic $c_1$ that meets $c$ and $c'$ only in its
     endpoints and finally along $c'$.
     We set $c_0 := c \cap \hat c$ and $c_2 := c' \cap \hat c$. 
     Then $c_0$ is disjoint from $c'$, and $c_2$ is disjoint from $c$, since otherwise there is a geodesic from $v$ to $v'$ that is contained in $c \cup c'$.
     Now $c_1$ splits the bi-gone formed by $c$ and $c'$ in two
     geodesic triangles.
     For any internal vertex $w$ of $\hat c$ we can use one of these two triangles   
     and apply Remark~\ref{rem:geodesic-n-gones}
     to conclude that $\varangle_{w} \hat c$ is $2\Theta^{(3)}$-small.
  \end{proof}

  \subsection*{Existence of geodesics.}

  The following (presumably well-known) 
  fact is implicitly used in~\cite{Bowditch(Rel-hyperbolic-groups)}, 
  but not explicitly stated.
  
  \begin{lemma}
    \label{lem:existenc-geodesics}
    For $\xi \neq \xi' \in \Delta_+$, 
    there exists a geodesic between $\xi$ and $\xi'$.
  \end{lemma}

  \begin{proof}
    We will use angles to extend the proofs in the locally finite case
    from~\cite[Lem.~3.1,3.2, p.428]{Bridson-Haefliger(1999)}.

    First, we consider the case $\xi' = v \in V$.
    If also $\xi \in V$, then the existence of a geodesic between $\xi$ and $v$ 
    is obvious.
    Let $c$ be a geodesic ray from some vertex $w$ to $\xi$.
    For $n \in \IN$ pick finite geodesics $c_n$ from $v$ to $c(n)$.
    It will be convenient to assume that $c_n \cap c$ is a finite geodesic,
    i.e.,  $c_n$ does not leave $c$ after meeting $c$.
    We claim, that, as in the locally finite case, a subsequence of $c_n$ converges
    pointwise to a geodesic ray $c'$.
    This ray is then a ray from $v$ to $\xi$.
    Let $e_n$ be the initial edge of $c_n$.
    Then we can use the geodesic triangle whose sides are
    $c_n$, $c_m$ and $c_{[c(n),c(m)]}$ to deduce that
    $(e_n,e_m)$ is $\Theta^{(3)}$-small.
    It follows that the $e_n$ range only over a finite set and we can pick
    a subsequence $I \subseteq \IN$ with $e=e_n$, $n \in I$ constant.
    Inductively we can now produce a subsequence of the $c_n$ that converges
    pointwise as claimed.

    Next, we consider the case $\xi, \xi' \in \dd \Gamma$.
    Let $c$ and $c'$ be geodesic rays to $\xi$ and $\xi'$.
    By the first case we can assume $c(0) = c'(0)$.
    We can also assume $c \cap c' = c(0)$.
    Let $c_n$ be a finite geodesic from $c(n)$ to $c'(n)$.
    We claim that, as in the locally finite case, there is 
    a vertex $p$ that belongs to infinitely many of the $c_n$.
    Since $\xi \neq \xi'$ there is $n_0 \in \IN$ with
    $d_\Gamma(c(n_0), c') > \delta$.
    Hyperbolicity implies then $d_\Gamma(c({n_0}),c_n) \leq \delta$ for $n \geq n_0$.
    We can therefore for $n \geq n_0$ 
    pick geodesics $\tilde c_n$ starting in $c(n_0)$ and ending
    in $w_n \in c_n$  of length $\leq \delta$.
    We can assume $\tilde c_n \cap c_n = w_n$.
    Let $w'_n$ be the last vertex on $\tilde c_n$ that is also on $c$.
    Since $d_\Gamma(w'_n,c(n_0)) \leq \delta$ and $w'_n \in c$ there
    is a subsequence $I \subseteq \IN$ with $w' = w'_n$, $n \in I$ constant.
    Consider for $n \in I$ now the geodesic triangles whose sides are
    $c|_{[w',c(n)]}$, $c_n|_{[c(n),w_n]}$ and $\tilde c_n|_{[w',w_n]}$.
    Using Remark~\ref{rem:geodesic-n-gones} it follows that 
    $\tilde c_n|_{[w',w_n]}$ is $2\Theta^{(3)}$-small.
    It also follows that the initial edge $e_n$ of $\tilde c_n|_{[w',w_n]}$
    forms a $\Theta^{(3)}$-small angle with an edge $e$ of $c$ incident to $w'$.
    Now Corollary~\ref{cor:finite-theta-balls} implies that
    the $w_n$ range only over a finite set.
    Thus after passing to a further subsequence we can assume 
    that there is a vertex $p$ contained in all $c_n$, $n \in I$.
    Now the argument used in the first case allows us to
    pass to a further subsequence and to assume that
    $c_n|_{[p,c(n)]}$, $n \in I$ converges pointwise to a geodesic ray
    from $p$ to $\xi$ and that $c_n|_{[p,c'(n)]}$, $n \in I$ converges
    pointwise to a geodesic ray from $p$ to $\xi'$.
    These two rays now combine to a bi-infinite geodesic between $\xi$ and $\xi'$.
  \end{proof}

  \subsection*{The observer topology.}

  Bowditch~\cite[Sec.~8]{Bowditch(Rel-hyperbolic-groups)} defined a 
  topology on $\Delta$; this topology is sometimes called the
  observer topology.
  This topology naturally is also defined on $\Delta_+ = \dd \Gamma \cup V$.
  We recall a basis for the observer topology. 
  For $\xi \in \Delta_+$ and a finite subset $V_0 \subseteq V$ let
  $M(\xi,V_0)$ consist of all $\xi' \in \Delta_+$ for which 
  all geodesics between $\xi$ and $\xi'$  miss $V_0 \setminus \{ \xi \}$.
  The $M(\xi,V_0)$ form an open basis for the 
  observer topology. 
  An important fact~\cite[p.~51]{Bowditch(Rel-hyperbolic-groups)} is
  that the $\forall$ in the definition of $M(\xi,V_0)$ 
  can be replaced with $\exists$ without changing the topology: 
  the sets $M'(\xi,V_0)$, defined to consist of all $\xi'$ for which 
  there \emph{exists} a geodesic between $\xi$ and $\xi'$ missing 
  $V_0 \setminus \{ \xi \}$,
  also form an open neighborhood basis for the observer topology. 
  The $M(v,V_0)$ with $v \in V$ and $V_0 \subset V$ finite are a countable basis for the observer topology, see the discussion preceding Lemma~8.4 in~\cite{Bowditch(Rel-hyperbolic-groups)}.
  The observer topology is compact and in particular 
  Hausdorff~\cite[Lem.~8.4, 8.6]{Bowditch(Rel-hyperbolic-groups)}.
  A convenient procedure to produce convergent subsequences in $\Delta_+$ is
  reviewed in Lemma~\ref{lem:compactness}.
  The $M(v,V_0)$ with $v \in V$ observer topology has a countable basis, 
  As a compact space with a countable basis for the topology $\Delta_+$ is 
  metrizable.  

   \begin{lemma}
    \label{lem:compactness}
    Let $(\xi_n)_{n \in \IN}$ be a sequence in $\Delta_+$.
    Let $v \in V$.
    For each $n$, let $c_n$ be a geodesic from $v$ to $\xi_n$.
    Then there exists a subsequence $I \subseteq \IN$ such
    that $\xi_n \to \xi \in \Delta_+$ and moreover
    the $c_n$, $n \in I$ converge as follows
    \begin{enumerate}
    \item \label{lem:compactness:boundary}
       if $\xi \in \dd \Gamma$, then the $c_n$, $n \in I$ converge
       pointwise to a geodesic ray from $v$ to $\xi$;
    \item \label{lem:compactness:finite}
       if $\xi \in V$, then $\xi \in c_n$ for all $n \in I$ 
       and each edge $e$ incident to $\xi$ is the initial edge
       for the restriction $c_n|_{[\xi,\xi_n]}$ for at most finitely
       many $n \in I$.
    \end{enumerate}
  \end{lemma}

  \begin{proof}
    For each $k$ let $c_n(k)$ be the $k$-th vertex along $c_n$ 
    (starting from $v = c_n(0)$). If $d_\Gamma(v,\xi_n) < k$, then
    we set $c_n(k) := \xi_n$.

    If $V_k := \{ c_n(k) \mid n \in \IN \}$ is finite for all
    $k$ then (using a diagonal subsequence) we can pick
    a subsequence $I \subseteq \IN$ such that $c(k) := c_n(k)$, $n \in I$ is
    eventually constant.
    If $d_\Gamma(v,c(k)) \to \infty$ with $k \to \infty$, $k \in \IN$,
    then the $c_n$, $n \in I$ converge pointwise to a geodesic ray $c$
    whose endpoint $\xi$ is also a limit for the $\xi_n$.
    In particular, \ref{lem:compactness:boundary} holds.
    If $d_\Gamma(v,c(k)) \not\to \infty$, then for all sufficiently large
    $n \in I$, $k \in \IN$, $\xi_n = c_n(k) = c(k) = \xi \in V$ must
    be constant. 
    In this case both $\xi_n$ and $c_n$, $n \in I$ are eventually constant
    and~\ref{lem:compactness:finite} holds.    
   
    If not all $V_k$ are finite, then
    there is $k_0$ such that $V_{k}$ is finite for all $k < k_0$ 
    and $V_{k_0}$ is infinite.
    Then we find a subsequence $I \subseteq \IN$ such that
    $c_i(k)$ is constant in $i$ for $k < k_0$ and
    $c_i(k_0) \neq c_j(k_0)$ for $i \neq j \in I$.
    Then $\xi_i \to \xi := c_i(k_0-1) \in V$ for $i \in I$, $i \to \infty$
    and again~\ref{lem:compactness:finite} holds.   
  \end{proof}

  \begin{addendum}
    \label{add:compactness}
    Suppose that in Lemma~\ref{lem:compactness} in addition $v$ is a vertex
    of finite valency and that the $c_n$ are all $\Theta$-small for
    some size for angles $\Theta$.
    Assume also the $\xi_n$ eventually leave every ball of finite radius in $\Gamma$.
    Then $\xi \in \dd \Gamma$ and, for a subsequence, the $c_n$ converge
    pointwise to a geodesic ray from $v$ to $\xi$.
  \end{addendum}

  \begin{proof}
    Corollary~\ref{cor:finite-theta-balls} implies that the $V_k$ appearing in the
    proof of Lemma~\ref{lem:compactness} are all finite.
    Since the $\xi_n$ leave eventually every ball of finite radius in $\Gamma$
    we have $d_\Gamma(v,c(k)) \to \infty$ with $k \to \infty$ in the proof
    of Lemma~\ref{lem:compactness} and the result follows.
  \end{proof}

  \subsection*{Large angles.}

  If the angle of a geodesic at some vertex $v$ is large,
  then this often forces further geodesics to pass through $v$.
  In the next few lemmas we collect some results of this kind.
  Results of this form are very common in 
  connection with relatively hyperbolic groups,
  see for example~\cite[Lem.~3]{Mineyev-Yaman(rel-hyperbolic-bounded-cohom)}. 

  \begin{lemma} \label{lem:large-angles}
    Let $c$ be a geodesic between $\xi_-$ and $\xi_+$.
    Assume that $v$ is an internal vertex of $c$ and that 
    $\varangle_v c$ is $\Theta^{(3)}$-large.
    Then any other geodesic $c'$ between $\xi_-$ and $\xi_+$ will
    also pass through $v$.
  \end{lemma}

  \begin{proof}
    Consider the geodesic triangle whose sides are 
    $c'$, $c|_{[v,\xi_-]}$ and $c|_{[v,\xi_+]}$. 
    By definition of $\Theta^{(3)}$ we have $v \in c'$.
  \end{proof}

  \begin{lemma} \label{lem:large-angles-in-triangles-no-c}
    Let $\xi$, $\xi_1$ and $\xi_2 \in \Delta_+$.
    Let $c_1$ be a geodesic between $\xi$ and $\xi_1$ and
    $c$ be a geodesic between $\xi_1$ and $\xi_2$.
    Let $\Theta_0$  be a size for angles.
    Let $v \notin c$ be an internal vertex of $c_1$
    such that the angle $\varangle_v c_1$ is $\Theta_0 + 2\Theta^{(3)}$-large.
    Then any geodesic $c_2$ between $\xi_2$ and $\xi$ contains
    $v$ as an internal vertex and  $\varangle_v c_2$ is
    $\Theta_0$-large.
  \end{lemma}

  \begin{proof}
    For the first claim, we subdivide $c_1$ at $v$ and obtain a $4$-gon 
    whose corners are $v$, $\xi$, $\xi_1$ and $\xi_2$.
    Since $v \notin c$, we must have $v \in c_2$ for otherwise
    the angle of $c_1$ at $v$ would need to be $2\Theta^{(3)}$-small,
    see Remark~\ref{rem:geodesic-n-gones}.
    Let $e_1$, $e'_1$, $e_2$ and $e_2'$ be edges incident to $v$
    such that $e_1$ points towards $\xi_1$ along $c_1$,
    $e'_1$ points towards $\xi$ along $c_1$,
    $e_2$ points towards $\xi_2$ along $c_2$, and
    $e'_2$ points towards $\xi$ along $c_2$.
    Thus $(e_1,e'_1) = \varangle_v c_1$ and
    $(e_2,e'_2) = \varangle_v c_2$.
    Lemma~\ref{lem:geodesic-2-gons} implies that $(e'_1,e'_2)$ is 
    $\Theta^{(3)}$-small.
    Since $v \notin c$, we can use the geodesic triangle whose sides
    are $c$, $c_1|_{[v,\xi_1]}$ and $c_2|_{[v,\xi_2]}$ 
    to see that $(e_1,e_2)$ is also $\Theta^{(3)}$-small.
    Since $(e_1,e'_1)$ is $\Theta_0 + 2\Theta^{(3)}$-large,
    this implies that $(e_2,e'_2)$ is $\Theta_0$-large.
  \end{proof}

  \begin{lemma}
    \label{lem:tripod}
    Let $\Theta_0$ be a size for angles.
    Let $\xi$, $\xi_1$ and $\xi_2 \in \Delta_+$.
    Let $c$ be a geodesic between $\xi_1$ and $\xi_2$,
    $c_1$ be a geodesic between $\xi_1$ and $\xi$,
    $c_2$ be a geodesic between $\xi_2$ and $\xi$.
    Suppose that $v$ is an internal vertex of
    $c$, $c_1$ and $c_2$ and that $\varangle_v c$ is 
    $2 \Theta_0 + 3 \Theta^{(3)}$-large.
    Then $\varangle_v c_1$ or $\varangle_v c_2$  is $\Theta_0$-large. 
  \end{lemma}

  \begin{proof}
    Let $e,e',e_1,e_1',e_2,e_2'$ be edges incident to $v$ such that
    $e$ points towards $\xi_1$ along $c$,
    $e'$ points towards $\xi_2$ along $c$,
    $e_1$ points towards $\xi_1$ along $c_1$,
    $e'_1$ points towards $\xi$ along $c_1$,
    $e_2$ points towards $\xi_2$ along $c_2$, and
    $e'_2$ points towards $\xi$ along $c_2$.
    Thus $(e,e') = \varangle_v c$,
    $(e_1,e_1') = \varangle_v c_1$, and
    $(e_2,e_2') = \varangle_v c_2$. 
    Lemma~\ref{lem:geodesic-2-gons} implies that $(e'_1,e'_2)$, $(e,e_1)$ 
    and $(e',e_2)$ are  $\Theta^{(3)}$-small.
    By assumption $(e,e')$ is $2 \Theta_0 + 3 \Theta^{(3)}$-large.
    Thus not both $(e_1,e_1')$ and $(e_2,e_2')$ can be $\Theta_0$-small.
  \end{proof}

  \begin{lemma}
    \label{lem:large-angles-in-triangles}
    Let $\Theta_0$ be a size for angles. 
    Then there exists a size $X$ for angles such that
    the following holds.
    Assume there is a  $\Theta_0$-small geodesic $c$ 
    between $\xi_1$ and $\xi_2$.
    Let $\xi \in \Delta_+ \setminus \{ \xi_1, \xi_2 \}$,  
    and $c_1$ be a geodesic from $\xi_1$ to $\xi$.
    Let $v \in V$ be an internal vertex of $c_1$ for which
    $\varangle_v c_1$ is $X$-large and $v \neq \xi_2$.
    Then any geodesic $c_2$ between $\xi_2$ and $\xi$ will also pass
    through $v$.
    Moreover, for any size for angles $\Theta$ we have
    \begin{enumerate}
    \item \label{lem:angles-in-triangles:small}
      if $\varangle_v c_1$ is $\Theta$-small, 
      then $\varangle_v c_2$ is $\Theta + X$-small; 
    \item \label{lem:angles-in-triangles:large}
      if $\varangle_v c_1$ is  $\Theta + X$-large,  
      then $\varangle_v c_2$ is $\Theta$-large.  
    \end{enumerate}
  \end{lemma}

  \begin{proof}
    For the first statement we can take $X := \Theta_0 + 2\Theta^{(3)}$.
    Lemma~\ref{lem:large-angles-in-triangles-no-c} implies
    $v \in c \cup c_2$.
    If $v \not\in c_2$ then Lemma~\ref{lem:large-angles-in-triangles-no-c}
    implies that $v$ is an internal vertex of $c$ and that
    $\varangle_v c$ is $\Theta_0$ large.
    But $c$ is assumed to be $\Theta_0$-small and thus $v \in c_2$. 
    
    For the second statement we can take $X := \Theta_0 + 3\Theta^{(3)}$.
    Let $e_1,e_1',e_2,e_2'$ be the edges incident to $v$, such that
    $e_1$ points towards $\xi_1$ along $c_1$,  
    $e_1'$ points towards $\xi$ along $c_1$,
    $e_2$ points towards $\xi_2$ along $c_2$, and 
    $e_2'$ points towards $\xi$ along $c_2$.
    In particular $\varangle_v c_1 = (e_1,e_1')$ and 
    $\varangle_v c_2 = (e_2,e_2')$.
    In order to prove~\ref{lem:angles-in-triangles:small}
    and~\ref{lem:angles-in-triangles:large} it 
    will suffice to show that
    $(e'_1,e'_2)$ is $\Theta^{(3)}$-small and that
    $(e_1,e_2)$ is $\Theta_0 + 2 \Theta^{(3)}$-small.
    The first statement follows from Lemma~\ref{lem:geodesic-2-gons}.
    For the second statement we need to distinguish the cases 
    $v \notin c$ and $v \in c$, i.e., to distinguish whether only $c_1$ and $c_2$ intersect in $v$, or whether all three geodesics intersect in $v$.
    If $v \notin c$ we can use the geodesic triangle whose sides
    are $c$, $c_1|_{[v,\xi_1]}$ and $c_2|_{[v,\xi_2]}$ to deduce that $(e_1,e_2)$ 
    is $\Theta^{(3)}$-small.
    If $v \in c$, then let $e$ and $e'$ 
    be the edges on $c$ incident to $v$,
    such that $e$ points towards $\xi_1$ along $c$,
    and $e'$ points towards $\xi_2$ along $c$. 
    Lemma~\ref{lem:geodesic-2-gons} implies that
    the angles $(e,e_1)$ and $(e',e_2)$ are $\Theta^{(3)}$-small.
    Since $c$ is $\Theta_0$-small the angle $(e,e')$ is $\Theta_0$-small.
    Thus $(e_1,e_2)$ is indeed $\Theta_0 + 2 \Theta^{(3)}$-small.
  \end{proof}

  \begin{lemma}
    \label{lem:pass-through-v} 
    Let $\xi \in \Delta_+$ and $v \in V$ where $\xi \neq v$.
    Then there exist neighborhoods $U$ of $v$ and $U'$ of $\xi$ 
    in $\Delta_+$  such that 
    any geodesic starting in $U$ and ending in $U'$ will meet $v$.   
  \end{lemma}

  \begin{proof}
    Assume this fails.
    Then there are $\xi_n, x_n \in \Delta_+$ with $\xi_n \to \xi$
    and $x_n \to v$ for which there are geodesics $b_n$ from $\xi_n$ to $x_n$ 
    in $\Gamma$ that do not meet $v$.
    Let $c$ be a geodesic from $\xi$ to $v$, $a_n$ be geodesics
    from $v$ to $x_n$ and $c_n$ be geodesics from $\xi_n$ to $\xi$. 
    Since $x_n \to v$, the initial edges of the
    $a_n$ form an infinite set.
    Lemma~\ref{lem:angles-locally-finite} and 
    Remark~\ref{rem:geodesic-n-gones} imply that
    not all angles at $v$ 
    determined by $c$ and $a_n$ can be $2\Theta^{(3)}$-small.
    On the other hand we can consider the geodesic $4$-gon with 
    sides $c$, $a_n$, $b_n$ and $c_n$.
    Since the angle at $v$ determined by $c$ and $a_n$ is eventually
    not $2\Theta^{(3)}$-small, eventually $v \in b_n \cup c_n$.
    Since $\xi_n \to \xi$, we can arrange that the $c_n$ eventually miss $v$.
    This implies that eventually $v \in b_n$, contradicting our assumption.
  \end{proof}

  \subsection*{Isotropy groups.}

  \begin{lemma}
    \label{lem:finitely-many-edges}
    Let $\xi \in \Delta_+$, $v \in V$ where $v \neq \xi$.
    There is a neighborhood $U$ of $\xi \in \Delta_+$ such that
    only finitely many edges in $\Gamma$ appear as the initial edge
    of a geodesic from $v$ to some $\xi' \in U$.
  \end{lemma}

  \begin{proof}
    Assume this fails.
    Then there are $\xi_n \in \Delta_+$, geodesics $c_n$ from $v$ to $\xi_n$ such
    that $\xi_n \to \xi$  and the initial edges of the $c_n$ are 
    pairwise different.
    Since the initial edges of the $c_n$ are all different it follows that 
    $\xi_n \to v$. 
    But $v \neq \xi$.
  \end{proof}

  \begin{lemma} \label{lem:isotropy-2-points}
    Let $\xi_- \neq \xi_+ \in \Delta_+$.
    Then the intersection of the isotropy groups 
    $G_{\xi_-} \cap G_{\xi_+}$ is virtually cyclic.
  \end{lemma}
 
  \begin{proof}
    If $\xi_- \in V$, then, by Lemma~\ref{lem:finitely-many-edges} 
    there are only finitely many edges that are incident to $\xi_-$
    and are part of a geodesic from $\xi_-$ to $\xi_+$. 
    The group $G_{\xi_-} \cap G_{\xi_+}$ acts on this finite set.
    Since the action of $G$ on the set of all edges is proper, 
    this implies that $G_{\xi_-} \cap G_{\xi_+}$ is finite.
    The same argument applies if $\xi_+ \in V$.

    It remains to treat the case $\xi_-, \xi_+ \in \dd \Gamma$.
    Let $L$ be the subgraph of $\Gamma$ spanned by all geodesics
    from $\xi_-$ to $\xi_+$.
    By hyperbolicity $L$ is contained in a bounded neighborhood of
    a fixed geodesic $c$ between $\xi_-$ and $\xi_+$. 
    By~\cite[Lem.~8.2]{Bowditch(Rel-hyperbolic-groups)}
    the intersection of $L$ with any bounded ball
    in $\Gamma$ around a vertex $v$ from $L$ contains only finitely many vertices,
    in particular, $L$ is locally finite.
    
    Since the action of $G$ on edges is proper, 
    the action of $G_{\xi_-} \cap G_{\xi_+}$ on $L$ is proper as well.
    Fix a vertex $v_0 \in c$.
    Now, if the group $G_{\xi_-} \cap G_{\xi_+}$ is infinite, then
    it contains elements $g$, with
    $d_\Gamma(v_0,gv_0)$ arbitrarily large.
    If $d_\Gamma(v_0,gv_0) >> 0$, then, since $L$ is contained in a 
    $\delta$-neighborhood of the geodesic $c$, it follows that
    either $d_\Gamma(v_0,g^2v_0) \sim 0$
    or $d_\Gamma(v_0,g^2v_0) \sim 2d_\Gamma(v_0,gv_0)$.
    In both cases, again since $L$ is contained in a bounded neighborhood of $c$,
    we get a good understanding of the action of $g$ on $L$.
    In the first case $g$ will act, up to bounded error, 
    as reflection with respect to a point of $c$.
    In particular, $g$ will exchange  $\xi_-$ and $\xi_+$.
    This excludes the first case.
    In the second cases $g$ acts, up to bounded error, like a translation along $c$.
    It follows  that $g^{\pm n}v_0 \to \xi_\pm$, and, in particular,
    that the action on $L$ of the cyclic group $C$ generated by $g$ is cocompact.
    Moreover, since $L$ is contained in a bounded neighborhood of $L$ and 
    since bounded balls in $L$ are finite and since the action of 
    $G$ on the edges of $L$ is proper, it follows that $C$ has finite
    index in $G_{\xi_-} \cap G_{\xi_+}$.  
  \end{proof}

  \begin{remark} \label{rem:type-I}
    Suppose that the group $G_{\xi_-} \cap G_{\xi_+}$ appearing in 
    Lemma~\ref{lem:isotropy-2-points} is infinite virtually cyclic.
    Then, since $G_{\xi_-} \cap G_{\xi_+}$ fixes the two ends of $L$,
    this group is virtually cyclic of type I, i.e., admits
    a surjection to an infinite cyclic group.
  \end{remark}

  \subsection*{Sequences of bounded distance} 

  \begin{lemma}
    \label{lem:small-at-infinity}
    Let $(v_n)_{n \in \IN}$ be a sequence in $V$
    that converges in $\Delta_+$ to $\xi \in \dd \Gamma$.
    If $(v'_n)_{n \in \IN}$ is another sequence in $V$ such that
    $d_\Gamma(v_n,v'_n)$ is bounded, then $v'_n$ also converges to $\xi$.
  \end{lemma}

  \begin{proof}
    Assume this fails.
    As $\Delta_+$ is compact there is then a convergent subsequence 
    $(v'_n)_{n \in I}$ with $\lim_{n \in I} v'_n =: \xi' \neq \xi$.
    If $\xi' \in \dd \Gamma$, then we can use Lemma~\ref{lem:compactness}
    to produce geodesic rays $c$ to $\xi$ and $c'$ to $\xi'$
    such that, after passing to a further subsequence, we can 
    assume $v_i \in c$ and $v'_i \in c'$ for all $i \in I$.
    Then, by hyperbolicity, $c$ and $c'$ are asymptotic and $\xi = \xi'$.

    It remains to contradict $\xi' \in V$.
    If $\xi' \in V$, then Lemma~\ref{lem:pass-through-v} allows us to assume, 
    after passing to a further subsequence, that for all $i \in I$ any geodesic
    between $v_i$ and $v'_i$ passes through $\xi'$.
    But then $d_\Gamma(\xi',v_i)$ is bounded and this contradicts $\xi \in \dd \Gamma$. 
  \end{proof}

  \begin{addendum}
    \label{add:small-at-infinty-V_infty}
    Retain the assumptions of Lemma~\ref{lem:small-at-infinity},
    but assume  $\xi = v \in V$.
    Assume in addition that there is $d > 0$ and a size for angles $\Theta$ 
    such that there are $\Theta$-small geodesics $c_n$ of length $\leq d$
    between $v_n$ and $v'_n$.
    Assume also that the $v'_n$ do \emph{not} converge to $v$.
    Then there is a subsequence $I \subseteq \IN$ such 
    that both $v_n$ and $v'_n$ are constant for $n \in I$.
  \end{addendum}

  \begin{proof} 
    We can argue as in the proof of Lemma~\ref{lem:small-at-infinity}
    and assume that $\xi' = \lim_n v'_n$ exists in $\Delta_+$ with $\xi' \neq v$.
    Using Lemma~\ref{lem:pass-through-v} as in the last 
    paragraph of proof of Lemma~\ref{lem:small-at-infinity},
    we can assume that the $c_n$, $n \in I$ all pass through $v$.
    Lemma~\ref{lem:finitely-many-edges} allows us, after passing to a further
    subsequence, to assume that  the initial edge
    of the restriction $c_n|_{[v,v'_n]}$ does not depend on $n \in I$.
    Since the $c_n$ are $\Theta$-small, it follows from 
    Lemma~\ref{lem:angles-locally-finite} that the initial
    edges of the $c_n|_{[v,v_n]}$, $n \in I$, can only vary over a finite set of edges.
    But now, as $v_n \to v$, we must have $v_n = v$ for all sufficiently large
    $n \in I$.  
    Since the $c_n$ are $\Theta$-small and of bounded length,   
    Corollary~\ref{cor:finite-theta-balls} implies that the $v'_n$, $n \in I$  
    also vary only over a finite set. 
  \end{proof}

  \subsection*{Coarse convexity.} 
  
  \begin{lemma}
    \label{lem:U-V}
    Let $\xi \in \Delta_+$ and $U$ be a neighborhood
    of $\xi \in \Delta_+$.
    Then there exists a neighborhood $U'$ of $\xi$ with
    the following property.
    If $v_-, v_+ \in U'$ and $v$ belongs to a geodesic
    between $v_-$ and $v_+$ then $v \in U$.
  \end{lemma}

  \begin{proof}
    Assume this fails.
    Then we find sequence of vertices $(v_-)_n, (v_+)_n, v_n$
    such that $(v_-)_n \to \xi$, $(v_+)_n \to \xi$,
    $v_n \notin U$ and $v_n$ belongs to a geodesic $c_n$ between
    $(v_-)_n$ and $(v_+)_n$.
    If $\xi \in V$, then we  apply Lemma~\ref{lem:pass-through-v}
    to  $c_n|_{[v_n,(v_-)_n]}$  and to $c_n|_{[v_n,(v_+)_n]}$.
    Lemma~\ref{lem:pass-through-v} implies that both restrictions
    eventually contain $\xi$, which can only happen if
    $v_n = \xi$.
    Since this contradicts $v_n \notin U$ we have
    $\xi \in \dd \Gamma$.
    Hyperbolicity implies that there are $v'_n$ such that
    $v'_n$ belongs to a geodesic from $(v_-)_n$ to $\xi$
    or to a geodesic from $(v_+)_n$ to $\xi$ and such 
    that $d_\Gamma(v_n,v_n')$ is uniformly bounded.
    The first property ensures $v'_n \to \xi$ since
    $(v_-)_n \to \xi$ and $(v_+)_n \to \xi$.
    The second property allows us to apply
    Lemma~\ref{lem:small-at-infinity} 
    and deduce that also $v_n \to \xi$.
    This contradicts again $v_n \notin U$.
  \end{proof}

  \section{The coarse $\Theta$-flow space for 
                                      relatively hyperbolic groups}
    \label{sec:coarse-flow-space}

  Throughout this section we use the notations and assumptions from
  Section~\ref{sec:relative-hyperbolic-groups}.
  In particular, $G$ is a group with a simplicial cocompact action on 
  a fine and hyperbolic graph $\Gamma$. 
  The stabilizers of edges under this action are finite. 
  We fix again a hyperbolicity constant $\delta \geq 0$ for $\Gamma$. 
  We will use the family $\calp$ consisting of 
  all subgroups $H \leq G$ that are virtually cyclic or 
  fix a vertex $v \in V$. 
  The space $\Delta$ is the union $\dd \Gamma \cup V_\infty$ and 
  equipped with the observer topology.  
  We also fix a proper left invariant metric $d_G$ on $G$.

  \begin{theorem}
    \label{thm:long-wide-GxDelta}
    The action of $G$ on $\Delta$ is finitely $\calp$-amenable.
  \end{theorem}     

  \begin{proof}
    This will be a consequence of Proposition~\ref{prop:cover-G-x-Theta-dd}
    and~\ref{prop:cover-G-x-V_infty}.
  \end{proof}

  \subsection*{Covering $G \x_{\Theta_0} \dd \Gamma$}
  It will be convenient to replace $\Gamma$ with its first
  barycentric subdivision $\Gamma'$.
  The set of vertices in the barycentric subdivision corresponding to
  the edges of the original graph will be denoted $V_E$.
  Thus the set of vertices of $\Gamma'$ is the disjoint union $V \cup V_E$,
  where $V$ is the set of vertices of $\Gamma$.
  The vertices in $V_E$ are all of valence $2$.
  We will give edges in $\Gamma'$ the length $1/2$.
  Then the path length metric of $\Gamma$ and $\Gamma'$ coincide.
  We will in this subsection use $\delta' := \delta + 1$.
  This has the effect that we can use vertices from $V_E$ when we apply 
  hyperbolicity:
  for any geodesic triangle and any vertex $v$ from $V_E$ on one
  side of the triangle there is a vertex $w$ from $V_E$ on one of the two
  other sides with $d_\Gamma(v,w) \leq \delta'$.
  Moreover, $\Gamma'$ is still fine and the 
  considerations from Section~\ref{sec:relative-hyperbolic-groups}
  and the appendix
  all apply to $\Gamma'$ as well.
  In particular, we can define $\Delta'$ and $\Delta'_+$.
  Since all vertices in $V_E$ are of valence $2$ we have
  $\Delta' = \Delta$ and $\Delta'_+ = \Delta_+ \cup V_E$.
  Of course, $\dd \Gamma = \dd \Gamma'$.
  For $v \in V$ there is a canonical bijection
  between angles at $v$ with respect to $\Gamma$ and
  angles with respect to $\Gamma'$. 
  In this subsection sizes for angles will only be considered at vertices
  from $V$.
  (Since there is only one non-trivial angle at any vertex $v \in V_E$,
  angles at $v \in V_E$ are not very interesting.)  

  We fix a base point $v_0 \in V_E$.
  For a size for angles $\Theta$ we write $G \x_\Theta \dd \Gamma$ for the
  subset of $G \x \dd \Gamma$ consisting of all pairs 
  $(g,\xi) \in G \x \dd \Gamma$ for which there exists a $\Theta$-small
  geodesic from $gv_0$ to $\xi$. 
  The main result of this subsection produces partial covers of
  $G \x \Delta$. 
  
  \begin{proposition}
    \label{prop:cover-G-x-Theta-dd}
    There is a number $N$ such that for any $\alpha > 0$ and any
    size for angles $\Theta_0$ there exists
    a $G$-invariant collection $\calu$ of  $\VCyc$-subsets of
    $G \x \Delta$ such that
    \begin{enumerate}
    \item the order of $\calu$ is at most $N$;
    \item for every $(g,\xi) \in G \x_{\Theta_0} \dd \Gamma$ there is
      $U \in \calu$ with $B_\alpha(g) \x \{ \xi \} \subseteq U$.
    \end{enumerate}
  \end{proposition}

  To prove Proposition~\ref{prop:cover-G-x-Theta-dd} we will
  construct a coarse flow space and use the long and thin covers from
  Theorem~\ref{thm:long-thin-cover-X-V-version}.
  For a size for angles $\Theta$,    
  we define the metric $d_\Theta$ on $V_E$ as follows.
  For $v,v' \in V_E$ we set
  \begin{equation*}
    d_\Theta(v,v') := \min \sum_{i=1}^n d_\Gamma(w_{i-1},w_i)
  \end{equation*}
  where the minimum is taken over all finite sequences 
  $v = w_0, w_1,\dots,w_n =v'$ of vertices from $V_E$
  such that there are $\Theta$-small geodesics 
  between $w_{i-1}$ and $w_i$ for all $i$.
  If there is no such sequence, then $d_\Theta(v,v') = \infty$.
  If there exists a $\Theta$-small geodesic 
  between $v$ and $v'$, then $d_\Theta(v,v') = d_\Gamma(v,v')$, but 
  in general $d_\Theta(v,v') \geq d_\Gamma(v,v')$. 

  \begin{lemma}
    \label{lem:d_Theta-proper}
    The metric $d_\Theta$ on $V_E$ is proper.
  \end{lemma}

  \begin{proof}
    If $d_\Theta(v,v') \leq n$, then there is a finite sequence 
    $v = w_0, w_1 \dots,w_n = v'$ in $V_E$ with
    $d_\Theta(w_{i-1}, w_i) \leq 1$ for all $i$.
    Therefore, it suffices to check that balls of radius $1$ are 
    finite.
    This is a consequence of Lemma~\ref{lem:angles-locally-finite}. 
  \end{proof}

  \begin{definition}
    \label{def:coarse-FS-Theta}
    Set $Z := (\Delta'_+)^{2}$.
    Let $\Theta$ be a size for angles. 
    The coarse $\Theta$-flow space $\CF(\Theta)$ for $\Gamma$ is the
    subset of $V_E \x Z$ consisting of all
    triples $(v,\xi_-,\xi_+)$ for which there exist $v' \in V_E$ and
    a $\Theta$-small geodesic $c$ between $\xi_-$ and $\xi_+$ such that
    $v' \in c$ and  $d_{\Theta}(v,v') \leq \delta'$.
    Moreover, we require $\xi_-, \xi_+ \in V_E \cup \dd \Gamma$.
  \end{definition}

  \begin{example}
    \label{ex:coarse-flow-space-tree}
    Suppose that $\Gamma$ is a locally finite tree. 
    The flow space $\FS$ from~\cite{Bartels-Lueck(2012CAT(0)flow)} for $\Gamma$
    consists of all generalized (parametrized) geodesics 
    $c \colon \IR \to \Gamma$.
    If we use $\delta' = 0$ and all angles for $\Theta$, then there is a natural 
    embedding $\CF(\Theta) \to \FS$ that sends
    $(v,\xi_-,\xi_+)$ to the generalized geodesic $c \colon \IR \to \Gamma$ with
    $c(-\infty) = \xi_-$, $c(0) = v$ and $c(+\infty) = \xi_+$.
  \end{example}

  \begin{lemma}
    \label{lem:properties-CF(Theta)}
    Let $\Theta$ be a size for angles with $2\Theta^{(3)} \subseteq \Theta$.
    Then
    \begin{enumerate}
    \item \label{lem:prop-CF(Theta):closed}
      $\CF(\Theta) \subseteq V_E \x Z$ is closed;
    \item \label{lem:prop-CF(Theta):dim} 
      $\dim \CF(\Theta) \leq \dim Z < \infty$;
    \item \label{lem:prop-CF(Theta):D} 
      for $(\xi_-,\xi_+) \in Z$, 
      $V_{\xi_-,\xi_+} := \{ v \in V_E \mid (v,\xi_-,\xi_+) \in \CF(\Theta) \}$
      has the $(D,R)$-doubling property with respect to $d_\Theta$,
      where $D$ and $R$ are independent of $(\xi_-,\xi_+)$ and $\Theta$;
    \item \label{lem:prop-CF(Theta):isotropy}
      for $(v,\xi_-,\xi_+)$ the isotropy group 
      $G_{\xi_-,\xi_+} = \{ g \in G \mid g\xi_-=\xi_-, g\xi_+=\xi_+ \}$
      is virtually cyclic. 
    \end{enumerate}
  \end{lemma}

  \begin{proof}
    \ref{lem:prop-CF(Theta):closed}
     It suffices to show that 
     $Z_v := \{ (\xi_-,\xi_+) \mid (v,\xi_-,\xi_+) \in \CF(\Theta) \}$
     is closed for each $v \in V_E$.
     Let $((\xi_-)_n,(\xi_+)_n) \in Z_v$ with 
     $((\xi_-)_n,(\xi_+)_n) \to (\xi_-,\xi_+)$ in $Z$.
     We need to show  $(\xi_-,\xi_+) \in Z_v$.
     Since $((\xi_-)_n,(\xi_+)_n) \in Z_v$ there are $\Theta$-small geodesics
     $c_n$ from $(\xi_-)_n$ to $(\xi_+)_n$ and vertices $v_n \in c_n \cap V_E$
     with $d_\Theta(v_n,v) \leq \delta'$.
     As $d_\Theta$ is proper by Lemma~\ref{lem:d_Theta-proper}
     we can pass to a subsequence and assume that $v_n = w$ is constant.
     Moreover, we can apply Lemma~\ref{lem:compactness} to
     the $c_n|_{[w,(\xi_-)_n]}$ and assume that
     either the $c_n|_{[w,(\xi_-)_n]}$ converge pointwise to a geodesic 
     ray from $w$ to $\xi_- \in \dd \Gamma$, 
     or that $\xi_- \in c_n|_{[w,(\xi_-)_n]}$ for all $n$.
     In the second case, it also follows
     that eventually $(\xi_-)_n = \xi_-$, as otherwise $(\xi_-)_n \to \xi_-$ would imply
     that the $\varangle_{\xi_-} c_n$ are eventually $\Theta$-large by
     Lemma~\ref{lem:angles-locally-finite}.
     Therefore, $\xi_- \in V_E \cup \dd \Gamma$ and
     the $c_n|_{[w,(\xi_-)_n]}$ converge pointwise to a geodesic $c_-$ from
     $w$ to $\xi_-$.
     Similar, $\xi_+ \in V_E \cup \dd \Gamma$ and we can assume that the
     $c_n|_{[w,(\xi_+)_n]}$ converge pointwise to a geodesic $c_+$ from
     $w$ to $\xi_+$.
     Now $c_-$ and $c_+$ combine to a geodesic between $\xi_-$ and $\xi_+$
     that passes through $w$.
     Thus $(\xi_-,\xi_+) \in Z_v$.    
     \\[1ex]
    \ref{lem:prop-CF(Theta):dim} 
     Theorem~\ref{thm:dim-delta-finite} asserts that $\Delta'_+$
     is finite dimensional.
     Since $V_E$ is discrete and $\CF(\Theta)$ is closed in
     $V_E \x Z$ it follows that 
     $\dim \CF(\Theta) \leq \dim Z = 2 \dim \Delta'_+ < \infty$.
     \\[1ex]
     \ref{lem:prop-CF(Theta):D} 
     If $V_{\xi_-,\xi_+}$ is non-empty, then there is a $\Theta$-small
     geodesic $c$ between $\xi_-$ and $\xi_+$.
     By hyperbolicity, any vertex $v' \in V_E$ on 
     any other $\Theta$-small geodesic $c'$ from $\xi_-$ to $\xi_+$
     will be within distance $\delta'$ of some vertex $v \in V_E \cap c$.
     If $c \cup c'$ does not contain a geodesic between $v$ and $v'$, then
     Lemma~\ref{lem:geodescs-between-geodesics} provides us with a $\Theta$-small
     geodesic between $v$ and $v'$.
     If $c \cup c'$ contains a geodesic $c''$ between $v$ and $v'$, then
     we can assume that $c''$ changes only once from $c$ to $c'$, at some 
     vertex $w \in V$.
     Then $c''$ may fail to be $\Theta$-small only at this vertex $w$.
     In this case we may assume that $v$ is incident to $w$.
     Let now $\tilde v \in V_E$ be the unique vertex incident to
     $w$ on $c' \setminus c''$.
     Then $c'|_{[v',\tilde v]}$ is a $\Theta$-small geodesic of length
     at most $\delta'$ and thus $d_\Theta(v',\tilde v) \leq \delta'$.
     We can apply Lemma~\ref{lem:geodesic-2-gons} to the bi-gone 
     spanned by $c|_{[w,\xi_\pm]}$ and $c'|_{[w,\xi_\pm]}$ to
     deduce that the angle at $w$ spanned by the edges $(w,\tilde v)$
     and $(w,v)$ is $\Theta^{(3)}$-small.
     Thus $d_\Theta(v,\tilde v) = 1$ and $d_\Theta(v,v') \leq \delta'+1$.
          
     Thus any $v \in V_{\xi_-,\xi_+}$ will be within distance $2\delta'+1$ of
     a vertex from $V_E \cap c$ with respect to $d_\Theta$.
     Therefore, any $\alpha$-separated subset $S$ in a $2\alpha$-ball 
     in $V_{\xi_-,\xi_+}$ can be mapped injectively to an 
     $\alpha-4\delta'-2$-separated subset $S'$
     of an $2\alpha+2\delta'+1$-ball in $V_E \cap c$.   
     For sufficiently large $\alpha$ (for example, $\alpha > R := 24 \delta' + 12$)
     any such set $S'$ contains at most $5$ elements, since $V_E \cap c$ is isometric
     to a subset of $\IZ$.
     \\[1ex]
     \ref{lem:prop-CF(Theta):isotropy}
     If $\xi_- \in V_E$ or $\xi_+ \in V_E$, then $G_{\xi_-,\xi_+}$ is finite
     since the action of $G$ on $V_E$ is proper.
     If $\xi_-,\xi_+ \in \dd \Gamma$, then, by definition of $\CF(\Theta)$,
     $\xi_- \neq \xi_+$.
     Lemma~\ref{lem:isotropy-2-points} implies that then $G_{\xi_-,\xi_+}$ is
     virtually cyclic. 
  \end{proof}

  \begin{proposition}
    \label{prop:cover-CF(Theta)}
    There is a number $N'$ such that for any $\alpha' > 0$ and 
    any size for angles $\Theta$ containing $2\Theta^{(3)}$ there exists a $\VCyc$-cover
    $\calw$ of $\CF(\Theta)$ such that
    \begin{enumerate}
    \item \label{prop:cover-CF(Theta):dim}
     $\dim \calw \leq N'$;
    \item \label{prop:cover-CF(Theta):long}
     for any $(v,\xi_-,\xi_+) \in \CF(\Theta)$ there is
     $W \in \calw$ such that 
     $B^\Theta_{\alpha'}(v) \x \{ (\xi_-,\xi_+) \} \cap \CF(\Theta) \subseteq W$.
     Here $B^\Theta_{\alpha'}(v)$ is the $\alpha'$-ball in $V_E$ with respect to
     $d_\Theta$.
    \end{enumerate}
  \end{proposition}
  
  \begin{proof}
    Let $\Theta$ be a size for angles.
    Lemma~\ref{lem:d_Theta-proper} and~\ref{lem:properties-CF(Theta)}
    allow us to apply Theorem~\ref{thm:long-thin-cover-X-V-version}.
    We obtain a number $N'$ such that for any $\alpha' > 0$ there
    exists a $\VCyc$-cover $\calw$ of $\CF(\Theta)$ 
    satisfying \ref{prop:cover-CF(Theta):dim} and 
    \ref{prop:cover-CF(Theta):long}. 

    The number $N'$ only depends on $\dim \CF(\Theta)$ and 
    the doubling constant. For these numbers
    Lemma~\ref{lem:properties-CF(Theta)} provides bounds  
    that do not depend on $\Theta$.
    Therefore $N'$ does not depend on $\Theta$.  
  \end{proof}
  
  \begin{lemma}
    \label{lem:alpha+Theta_0-->Theta}
    Let $\alpha > 0$ and $\Theta_0$ be a size for angles.
    Then there is a size for angles $\Theta$ containing $\Theta_0 + 2\Theta^{(3)}$ with the following property:
    If $c$ is a $\Theta_0$-small geodesic from $gv_0$ to $\xi \in \dd \Gamma$ and $g' \in G$ with $d_G(g,g') \leq \alpha$, then 
    \begin{enumerate}
    \item \label{lem:alpha-Theta_0:c'} 
      any geodesic $c'$ from $g'v_0$ to $\xi$
      is $\Theta$-small;
    \item \label{lem:alpha-Theta_0:c''}
      any geodesic $c''$ starting on $c$ and ending on a geodesic
      $c'$ from $g'v_0$ to $\xi$ is $\Theta$-small.
    \end{enumerate}
  \end{lemma}

  \begin{proof}
    Pick $\Theta_1$ such that for any $h \in B_\alpha(e)$ there exists
    a $\Theta_1$-small geodesic from $hv_0$ to $v_0$.
    Thus, if $g' \in G$ with $d_G(g,g') \leq \alpha$ 
    then there is a $\Theta_1$-small
    geodesic between $gv_0$ and $g'v_0$.
    Let $X$ be the size for angles from 
    Lemma~\ref{lem:large-angles-in-triangles}.
    In particular, 
    for any size for angles $Y$, and any geodesic triangle in $\Gamma$
    where the first side is $\Theta_0$-small and the second side is $Y$-small,
    the third side will be $Y + X$-small.
    Then~\ref{lem:alpha-Theta_0:c'} holds whenever 
    $\Theta_1 + X  \subseteq \Theta$ 
    and~\ref{lem:alpha-Theta_0:c''} holds whenever 
    $\Theta_1 + 2X \subseteq \Theta$.
  \end{proof}

  \begin{definition}
    \label{def:from-CF(Theta)-to-Gx_ThetaDelta}
    Let $W \subseteq \CF(\Theta)$ and $\tau \in \IN$.
    We define
    \begin{equation*}
      \iota^{-\tau} W \subseteq G \x_\Theta \dd \Gamma 
    \end{equation*}
    as the subspace consisting of pairs $(g,\xi) \in G \x_\Theta \dd \Gamma$
    with the following property.
    For every $\Theta$-small geodesic $c$ from $gv_0$ to $\xi$ 
    we have $(v_c,gv_0,\xi) \in W$, where $v_c$ is the unique vertex
    in $V_E \cap c$ with $d_\Gamma(gv_0,v_c) = \tau$. 
  \end{definition}

  \begin{lemma}
    \label{lem:iota^-tau-open}
    If $W$ is open in $\CF(\Theta)$, then $\iota^{-\tau} W$
    is open in $G \x_\Theta \dd \Gamma$.
  \end{lemma}

  \begin{proof}
    Let $(g,\xi) \in \iota^{-\tau} W$. 
    Assume that $(g,\xi)$ does not belong to the interior of $\iota^{-\tau} W$
    in $G \x_\Theta \dd \Gamma$.
    Then there are $\xi_n \in \dd \Gamma$ with $\xi_n \to \xi$ and
    $\Theta$-small geodesics $c_n$ from $gv_0$ to $\xi_n$
    such that for the vertices $v_n \in V_E \cap c_n$ with 
    $d_\Gamma(gv_0,v_n) = \tau$ we have $(v_n,gv_0,\xi_n) \not\in W$.
    Using Lemma~\ref{lem:compactness}~\ref{lem:compactness:boundary} 
    and passing to a subsequence
    we may assume that the $c_n$ converge pointwise to a $\Theta$-small
    geodesic $c$ from $gv_0$ to $\xi$.
    Then eventually $v_n = v_c$ is constant and belongs to $c$.
    Since $(g,\xi) \in W$, it follows that $(v_c,gv_0,\xi) \in W$.
    Since $W$ is open, eventually $(v_n,gv_0,\xi_n) \in W$,
    contradicting our assumption. 
  \end{proof}

  \begin{proof}[Proof of Proposition~\ref{prop:cover-G-x-Theta-dd}]
    Let $\Theta_0$ and $\alpha > 0$ be given.
    Let $\Theta$ be the size for angles from 
    Lemma~\ref{lem:alpha+Theta_0-->Theta}.
    Since $B_\alpha(e) \subseteq G$ is finite, we can find a number 
    $\alpha' > 0$
    such that $d_\Gamma(g v_0, gh v_0) + 2 \delta' \leq \alpha'$ for all
    $g \in G$ and $h \in B_\alpha(e)$.
    Let $\calw$ be the cover of $\CF(\Theta)$ 
    from Proposition~\ref{prop:cover-CF(Theta)}.
    For $\tau \in \IN$ let 
    $\iota^{-\tau}\calw := \{ \iota^{-\tau} W \mid W \in \calw \}$. 
    By Lemma~\ref{lem:iota^-tau-open}, the members of $\iota^{-\tau}\calw$
    are open subsets of $G \x_\Theta \dd \Gamma$.
    Since $\iota^{-\tau}$ is a $G$-equivariant operation and 
    $\calw$ is $G$-invariant and consist of 
    $\VCyc$-subset, the same is true for $\iota^{-\tau} \calw$.
    The order of $\calw$ is bounded by 
    Proposition~\ref{prop:cover-CF(Theta)}~\ref{prop:cover-CF(Theta):dim}.
    Since $\iota^{-\tau}$ commutes with intersection, the order of $\iota^{-\tau} \calw$ is bounded by the order of $\calw$.     
    
    We claim that there exists $\tau \in \IN$ such that 
    $\iota^{-\tau} \calw$ is $\alpha$-wide in the $G$-direction, i.e.,
    we claim that there is $\tau$ such that 
    \begin{numberlist}
      \item [\label{nl:long}]
        for any $(g,\xi) \in G \x_{\Theta_0} \dd \Gamma$ 
        there is $W \in \calw$ with 
        $B_\alpha(g) \x \{ \xi \} \subseteq \iota^{-\tau} W$.
    \end{numberlist}
    Suppose there is no such $\tau$. 
    Then there is a sequence of pairs $(g_\tau, \xi_\tau)_{\tau \in \IN}$ 
    in $G \x_{\Theta_0} \dd \Gamma$  such that 
    \begin{numberlist}
      \item [\label{nl:no-W}]
         for all $W \in \calw$ and all $\tau \in \IN$ we have 
         $B_\alpha(g_\tau) \x \{ \xi_\tau \} \not\subseteq \iota^{-\tau} W$.
    \end{numberlist}
    By definition of $G \x_{\Theta_0} \dd \Gamma$
    there are $\Theta_0$-small geodesics $c_\tau$ from $g_\tau v_0$ to 
    $\xi_\tau$  in $\Gamma$.
    Let $v_\tau$ be the unique vertex on $c_\tau$ with 
    $d_\Gamma(g_\tau v_0, v_\tau) = \tau$.
    Since the action of $G$ on $V$ is cofinite, and since the 
    $\iota^{-\tau}\calw$ are $G$-invariant  we may, 
    after passing to a subsequence 
    $(g_\tau, \xi_\tau)_{\tau \in T}$, assume that $v_\tau = v$ is constant.
    Using the compactness of $\Delta'_+$, we may, after passing to a further
    subsequence, assume that
    $\xi_- := \lim g_\tau v_0$ and $\xi_+ := \lim \xi_\tau$
    exist in $\Delta'_+$.
    Addendum~\ref{add:compactness} implies $\xi_-,\xi_+ \in \dd \Gamma$
    and allows us to assume that the restrictions  $c_\tau|_{[v,g_\tau v_0]}$
    and $c_\tau|_{[v,\xi_\tau]}$ 
    converge pointwise to geodesic rays $c_-$ from $v$ to $\xi_-$
    and $c_+$ from $v$ to $\xi_+$.  
    In particular,  $c_-$ can be combined with $c_+$ to obtain
    a geodesic $c$ between $\xi_-$ and $\xi_+$.
    This geodesic $c$ is $\Theta_0$-small, since the $c_\tau$ are 
    $\Theta_0$-small.

    By Proposition~\ref{prop:cover-CF(Theta)}~\ref{prop:cover-CF(Theta):long}
    there is $W \in \calw$ such that
    $B^\Theta_{\alpha'}(v) \x \{ (\xi_-,\xi_+) \} \cap \CF(\Theta) \subseteq W$.
    Let $V'$ be the set of $v' \in B^\Theta_{\alpha'}(v)$ with 
    $d_\Theta(v',w) \leq \delta'$ for some $w \in c$.
    Then $(v',\xi_-,\xi_+) \in \CF(\Theta)$ for all $v' \in V'$.
    Since $d_\Theta$ is proper by Lemma~\ref{lem:d_Theta-proper},
    the set $V'$ is finite.
    Therefore we find neighborhoods $U_-$ of $\xi_-$ and $U_+$ of $\xi_+$,
    such that $V' \x U_- \x U_+ \cap \CF(\Theta) \subseteq W$.
    Elements of $B_\alpha(g_\tau)$ can be written as $g_\tau h$ with
    $h \in B_\alpha(e)$.
    For each $h$, $\lim_\tau g_\tau h v_0 = \lim_\tau g_\tau v_0 = \xi_-$
    by Lemma~\ref{lem:small-at-infinity}.
    Since $B_\alpha(e)$ is finite we find $\tau_0$ such that
    $g_{\tau_0} hv_0 \in U_-$ for all $h \in B_\alpha(e)$.
    Moreover, we can arrange for $\tau_0$ to satisfy
    $\tau_0 > \alpha'$, $\xi_{\tau_0} \in U_+$,
    and
    $c_{\tau_0} \cap B_{\alpha'}(v) = c \cap B_{\alpha'}(v)$.
    We claim that
    \begin{numberlist}
      \item[\label{nl:tau-in-iota-tau-W}]
      $B_\alpha(g_{\tau_0}) \x \{ \xi_{\tau_0} \} \subseteq \iota^{-\tau_0} W$. 
    \end{numberlist}
    Let $h \in B_\alpha(e)$.
    By Lemma~\ref{lem:alpha+Theta_0-->Theta}~\ref{lem:alpha-Theta_0:c'}
    there is a $\Theta$-small geodesic $c'$
    from $g_{\tau_0} hv_0$ to $\xi_{\tau_0}$.
    Let $v' \in V_E \cap c'$ be the vertex with 
    $d_\Gamma(g_{\tau_0} hv_0,v') = \tau_0$.
    We need to show that $(v',g_{\tau_0}hv_0,\xi_{\tau_0})$ belongs to $W$.    
    To this end it suffices to show $v' \in V'$, 
    where $V'$ is as defined earlier.
    Now we use hyperbolicity to find $w \in V_E \cap c_{\tau_0}$ with
    $d_\Gamma(w,v') \leq \delta'$.
    Lemma~\ref{lem:alpha+Theta_0-->Theta}~\ref{lem:alpha-Theta_0:c''}
    implies that there is a $\Theta$-small geodesic between $w$ and $v'$.
    Therefore $d_\Theta(w,v') \leq \delta'$.
    Since $d_\Gamma(g_{\tau_0}v_0,g_{\tau_0}hv_0) \leq \alpha' - 2\delta'$ and 
    $d_\Gamma(w,v') \leq \delta'$ we have
    $d_\Gamma(g_{\tau_0}hv_0,v') - (\alpha' - \delta') 
     \leq d_\Gamma(g_{\tau_0}v_0,w) 
     \leq d_\Gamma(g_{\tau_0}hv_0,v') + (\alpha' - \delta')$.
    Since $d_\Gamma(g_{\tau_0}v_0,v) = d_\Gamma(g_{\tau_0}hv_0,v') = \tau_0$
    we have $d_\Gamma(v,w) \leq \alpha' - \delta'$.
    Now $c_{\tau_0} \cap B_{\alpha'}(v) = c \cap B_{\alpha'}(v)$
    implies $w \in c$.
    Since $c$ is $\Theta_0$-small and therefore $\Theta$-small,
    and since $v,w \in c$,
    we have $d_\Theta(w,v) \leq \alpha' - \delta'$.
    Therefore $d_\Theta(v',v) \leq \alpha'$ and $v' \in B_{\alpha'}^\Theta(v)$.
    Altogether we have shown $v' \in V'$ and 
    therefore~\eqref{nl:tau-in-iota-tau-W}.
    This contradicts~\eqref{nl:no-W} and finishes the proof of~\eqref{nl:long}.

    It remains to extend $\iota^{-\tau} \calw$ to a $G$-invariant collection
    of $\VCyc$-subset of $G \x \Delta$ of the same order.
    Since $\Delta$ is metrizable, there is a $G$-invariant metric on 
    $G \x \Delta$.
    Thus the needed extension exists by Lemma~\ref{lem:extend-cover}.
  \end{proof}
   
  \subsection*{Covering $G \x \Delta \setminus G \x_\Theta \dd \Gamma$}

  \begin{definition} \label{def:N(v)}
    Let $v \in V$ and let $\Theta$ be a size for angles in $\Gamma$.
    Define 
    \begin{equation*}
      V_+(v,\Theta) \subseteq G \x \Delta
    \end{equation*}
    as the set of all $(g,\xi)$ with the following properties
    \begin{enumerate}
    \item all geodesics from $gv_0$ to $v$ 
      are $\Theta$-small;
    \item if $\xi \neq v$, then there is a geodesic $c$ in 
      $\Gamma$ from $gv_0$ to $\xi$ such that
      $\varangle_v c$ is $\Theta$-large.
      (In particular, we require $v \in c$.)
    \end{enumerate}
    Let $V(v,\Theta)$ be the interior of $V_+(v,\Theta)$.
  \end{definition}

  \begin{lemma} \label{lem:V(beta,v)} 
    Assume that $\Theta$ is a size for angles
    containing $\Theta^{(3)}$.
    \begin{enumerate}
    \item \label{lem:V(beta,v):interior}
       For $v \in V$ we have $V_+(v,\Theta) \cap G \x V_\infty \subseteq V(v,\Theta)$;
    \item \label{lem:V(beta,v):dim}
       for $v \neq v' \in V$  we have 
       $V(v,\Theta) \cap V(v',\Theta) = \emptyset$;
    \item \label{lem:V(beta,v):G}
       if $g V(v,\Theta) \cap V(v,\Theta) \neq \emptyset$ with 
       $g \in G$, $v \in V$,
       then $gv = v$ and $g V(v,\Theta) = V(v,\Theta)$. 
    \end{enumerate}
  \end{lemma}  

  \begin{proof} 
    \ref{lem:V(beta,v):interior}
    Let $(g,w) \in V_+(v,\Theta)$ with $w \in V_\infty$.
    Let $c$ be a geodesic from $gv_0$ to $w$. 
    We proceed by contradiction and assume that $(g,w)$ does not
    belong to the interior of $V_+(v,\Theta)$.
    Then there is a sequence $\xi_n$ in $\Delta$ with $\xi_n \to w$
    and $(g,\xi_n) \notin V_+(v,\Theta)$.
    Let $c_n$ be a geodesic from $gv_0$ to $\xi_n$.
    By Lemma~\ref{lem:pass-through-v} the $c_n$ will eventually pass 
    through $w$.
    In this case we can change $c_n$ and arrange for $c_n|_{[gv_0,w]} = c$. 
    If $w \neq v$ then these $c_n$ prove that eventually 
    $(g,\xi_n) \in V_+(v,\Theta)$ which 
    contradicts our assumption.
    If $w = v$, then it remains to show that $\varangle_v c_n$  
    will eventually be $\Theta$-large.
    In this case the initial edges of the restriction  $c_n|_{[v,\xi_n]}$
    will form an infinite set since $\xi_n \to v$. 
    This implies, by Lemma~\ref{lem:angles-locally-finite}, that 
    $\varangle_v c_n$ at $v$ will eventually be $\Theta$-large. 
    \\[1ex]
    \ref{lem:V(beta,v):dim}
    Assume there is $(g,\xi) \in V(v,\Theta) \cap V(v',\Theta)$.
    Let $c$ and $c'$ be  geodesics from
    $gv_0$ to $\xi$ such that $\varangle_v c$ is $\Theta$-large (if $v \neq \xi$)
    and  $\varangle_{v'} c'$  is $\Theta$-large (if $v' \neq \xi$).
    Since $\Theta^{(3)} \subseteq \Theta$, Lemma~\ref{lem:large-angles} 
    implies $v \in c'$ and $v' \in c$.
    Assume $v \neq \xi$.
    Without loss of generality, we may assume that $v$ is closer to $gv_0$
    than $v'$.
    Then $v$ is an internal vertex of $c|_{[gv_0,v']}$.
    Since $\varangle_v c$ is $\Theta$-large, $c|_{[gv_0,v']}$ 
    is not $\Theta$-small.
    This contradicts $(gv_0,\xi) \in V(v',\Theta)$.   
    \\[1ex]
    \ref{lem:V(beta,v):G}
    Clearly, $g V(v,\Theta) = V(gv, \Theta)$. 
    Thus~\ref{lem:V(beta,v):dim} shows that 
    $g V(v,\Theta) \cap V(v,\Theta) \neq \emptyset$
    implies $gv = v$. 
  \end{proof}

  \begin{lemma}
    \label{lem:interior-V}
    Let $\Theta$ be a size for angles.
    Let $v \in V$, $g \in G$, $\xi \in \Delta$.
    Assume that any geodesic from $gv_0$ to $v$ is
    $\Theta$-small.
    Assume that there is a geodesic $c$ from 
    $gv_0$ to $\xi$ that passes through $v$ and satisfies
    one of the following two conditions
    \begin{enumerate}
    \item \label{lem:interior-V:3} 
       $\varangle_v c$ is $\Theta + 2\Theta^{(3)}$-large; 
    \item \label{lem:interior-V:4}
       $\varangle_v c$ is $\Theta$-large and 
       there is an interior vertex $w$ of $c$ with
       $d_\Gamma(gv_0,v) < d_\Gamma(gv_0,w)$ such that  $\varangle_w c$  
       is $2\Theta^{(3)}$-large.
    \end{enumerate}
    Then $(g,\xi) \in V(v,\Theta)$.
  \end{lemma}

  \begin{proof}
    If \ref{lem:interior-V:3} holds, then set $A := \{ v \}$;
    if \ref{lem:interior-V:4} holds, then set $A := \{ w \}$.
    Recall that $M'(\xi,A)$ consists of all 
    $\xi'$ for which there exists a geodesic between $\xi$ and $\xi'$ 
    that misses $A \setminus \{ \xi \}$.
    Since $M'(\xi,A)$ is part of a neighborhood basis for the 
    observer topology, it suffices to show that
    $(g,\xi') \in  V_+(v,\Theta)$ for all $\xi' \in M'(\xi,A)$.
    Let $\xi' \in M'(\xi,A)$. 
    Let $c''$ be a geodesic between $\xi$ and 
    $\xi'$ that misses $A$.
    Let $c'$ be a geodesic between $gv_0$ and $\xi'$.
    We apply Lemma~\ref{lem:large-angles-in-triangles-no-c} to
    the geodesic triangle with sides $c$, $c'$ and $c''$.
    
    If \ref{lem:interior-V:3} holds,  then
    Lemma~\ref{lem:large-angles-in-triangles-no-c} implies that 
    $c'$ passes through $v$ and that $\varangle_v c'$ is $\Theta$-large.
    Thus in this case $(g,\xi') \in  V_+(v,\Theta)$.     

    If \ref{lem:interior-V:4} holds, then 
    Lemma~\ref{lem:large-angles-in-triangles-no-c} implies $w \in c'$.
    Then we can replace $c'$ and assume $c'|_{[gv_0,w]} = c|_{[gv_0,w]}$.
    In particular $\varangle_v c' = \varangle_v c$ is $\Theta$-large.
    Thus in this case $(g,\xi') \in  V_+(v,\Theta)$ as well.
  \end{proof}

  \begin{proposition}
    \label{prop:cover-G-x-V_infty}
    Let $\alpha > 0$.
    There is a size for angles $\Theta$ and
    a $G$-invariant collection $\calv$ of $\calp$-subsets of $G \x \Delta$ 
    such that
    \begin{enumerate}
    \item \label{prop:cover-V_infty:dim}
      the order of $\calv$ is at most $2$; 
    \item \label{prop:cover-V_infty:long}
      for every $(g,\xi) \in G \x \Delta$ at least one of the following
      two statements holds
      \begin{itemize}
      \item there is $V \in \calv$ such that
        $B_\alpha(g) \x \{ \xi \} \subseteq V$;
      \item $\xi \in \dd \Gamma$ and 
        there is a $\Theta$-small geodesic from $gv_0$ to $\xi$.
      \end{itemize}
    \end{enumerate}
  \end{proposition}

  \begin{proof} 
    Since $\Gamma$ is fine there are only finitely many geodesics 
    between any two vertices.
    In particular, we can pick a
    size for angles $\Theta_0$ such that for all $h,h',h'' \in B_\alpha(e)$
    any geodesic starting at $hv_0$ and ending in some vertex
    on another geodesic between vertices $h'v_0$ 
    and $h''v_0$ is $\Theta_0$-small.
    Let $X$ be the size for angles appearing in 
    Lemma~\ref{lem:large-angles-in-triangles}.
    After increasing $X$, if necessary, we may assume 
    $\Theta_0 + 2\Theta^{(3)} \subseteq X$.  
    Let $\calv_1 := \{ V(v,2X) \mid v \in V \}$,
    $\calv_2 := \{ V(v,5X) \mid v \in V \}$
    and $\calv_3 := \{ V(v,6X) \mid v \in V \}$.
    All three collections are $G$-invariant.
    Lemma~\ref{lem:V(beta,v)} implies that 
    all three collections are of order $0$,
    and consist of $\calp$-sets. 
    Thus $\calv := \calv_1 \cup \calv_2 \cup \calv_3$ 
    is a $G$-invariant collection of $\calp$-sets
    and satisfies~\ref{prop:cover-V_infty:dim}.
 
    It remains to check that $\calv$ also 
    satisfies~\ref{prop:cover-V_infty:long}
    where we use $\Theta := 6X$. 
    Since $\calv$ is $G$-invariant it suffices to consider 
    $(e,\xi) \in G \x \Delta$.

    Let $\calc$ be the set of all geodesics from some 
    $hv_0$, $h \in B_\alpha(e)$ to $\xi$.
    If all $c \in \calc$ are $6X$-small, 
    then~\ref{prop:cover-V_infty:long} holds.
    Indeed, if $\xi \in V$, then, using 
    Lemma~\ref{lem:V(beta,v)}~\ref{lem:V(beta,v):interior},
    we obtain 
    $B_\alpha(e) \x \{ \xi \} \subseteq V(\xi,6X)$ 
    and~\ref{prop:cover-V_infty:long} holds. 
    If $\xi \in \dd \Gamma$, then~\ref{prop:cover-V_infty:long} 
    holds as well, simply since $\calc$ contains a  $\Theta$-small
    geodesic from $v_0$ to $\xi$. 
    Therefore we may assume that not all $c \in \calc$ are $6X$-small.

    Let $c_0$ be a geodesic from $v_0$ to $\xi$.
    Let $W$ be the set of all internal vertices $w$ of $c_0$  
    for which $\varangle_w c_0$ is $X$-large.
    Lemma~\ref{lem:large-angles-in-triangles}  implies for any $c \in \calc$
    and any size for angles $\Theta'$ the following holds:
    \begin{numberlist}
    \item[\label{nl:w-on-c}] $c$ will pass through any $w \in W$;
    \item[\label{nl:small-at-v}] 
       if $v$ is an internal vertex of $c$, where $v \notin W$,
       then the angle of $c$ at  $v$ is 
       $2X$-small;
    \item[\label{nl:large-at-w}] if, for $w \in W$, 
       the angle of $c_0$ at $w$ is
       $\Theta' + X$-large, then the angle of $c$ at $w$ is
       $\Theta'$-large;
       if, for $w \in W$, the angle of $c_0$ at $w$ is
       $4X$-small, then the angle of $c$ at $w$ is
       $5X$-small.
    \end{numberlist}
    Since not all $c \in \calc$ are $6X$-small,
    we conclude from~\eqref{nl:large-at-w} that $W \neq \emptyset$.
    Among all $w \in W$ we let $w_0$ be the one closest to $v_0$.
    Then~\eqref{nl:w-on-c}  implies that $w_0$ 
    is also closest to all $h v_0$, $h \in B_\alpha(e)$.
    In particular, any $c$ will, on its way to $\xi$, meet $w_0$ before meeting
    any of the other $w$.
    Using~\eqref{nl:small-at-v} this implies that any geodesic starting in 
    some $hv_0$, $h \in B_\alpha(e)$ and ending in $w_0$ will be $2X$-small.
    If $\varangle_{w_0} c_0$ is $4X$-large,
    then by~\eqref{nl:large-at-w} for any $c \in \calc$ 
    the angle of $c$ at $w_0$ is $3X$-large.
    Using Lemma~\ref{lem:interior-V}~\ref{lem:interior-V:3} we see that in 
    this case, $B_\alpha(e) \x \{ \xi \} \subseteq V(w_0,2X)$
    and~\ref{prop:cover-V_infty:long} holds. 

    Therefore we may assume that $\varangle_{w_0} c_0$ is $4X$-small.
    Then~\eqref{nl:large-at-w} implies that for all $c \in \calc$ the 
    angle at $w_0$ is $5X$-small.
    Since not all $c \in \calc$ are $\Theta$-small, there
    are $w \in W$ and $c \in \calc$ such that
    the angle of $c$ at $w$ is $5X$-large.
    Among all such pairs we pick $(w_1,c_1)$ such that 
    $w_1$ is closest to $v_0$.
    As before~\eqref{nl:w-on-c}  implies that $w_1$ 
    is also closest to all $h v_0$, $h \in B_\alpha(e)$.
    Using~\eqref{nl:small-at-v} 
    this implies that all geodesics
    starting in some $hv_0$, $h \in B_\alpha(e)$ 
    and ending at $w_1$ are $5X$-small.
    Since $w_0 \neq w_1$ we can use~\eqref{nl:w-on-c} again, to see that
    for any $h \in B_\alpha(e)$ there is a geodesic from
    $hv_0$ to $\xi$ that agrees with $c_1$ between $w_0$ and $\xi$.
    Therefore, for any $h \in B_{\alpha}(e)$ there is a geodesic
    from $h v_0$ to $\xi$ for which the angle at $w_1$ is
    $5X$-large.
    If $\xi \in V$, then using 
    Lemma~\ref{lem:V(beta,v)}~\ref{lem:V(beta,v):interior}, we deduce  
    $B_\alpha(e) \x \{ \xi \} \subseteq V(w_1,5X)$ 
    and~\ref{prop:cover-V_infty:long} holds.  
    
    If $\xi \in \dd \Gamma$ we need to distinguish two
    further cases.
    If the angle of $c_1$ at $w_1$ is even $\Theta = 6X$-large,
    then Lemma~\ref{lem:interior-V}~\ref{lem:interior-V:3} 
    implies $B_\alpha(e) \x \{ \xi \} \subseteq V(w_1,5X)$ 
    and~\ref{prop:cover-V_infty:long} holds. 
    Otherwise, we use again that not all $c \in \calc$ are $6X$-small.
    Therefore, there is $w_2 \in W$ and $c_2 \in \calc$ such that the
    angle of $c_2$ at $w_2$ is $6X$-large.
    Since $w_1$ was chosen closest to $v_0$, $w_2$ will be further
    from $v_0$ than $w_1$.
    Now using~\eqref{nl:large-at-w} twice we see that
    for any geodesic $c \in \calc$ the angle at $w_2$ is $4X$-large. 
    Since we already found geodesics for any $h \in B_\alpha(e)$
    from $hv_0$ to $\xi$ whose angle at $w_1$ is $5X$-large, 
    we can now use
    Lemma~\ref{lem:interior-V}~\ref{lem:interior-V:4} to conclude
    $B_\alpha(e) \x \{ \xi \} \subseteq V(w_1,5X)$. 
    Therefore~\ref{prop:cover-V_infty:long} holds.  
  \end{proof}

  \subsection*{Covering $G \x \overline{P_{d,\Theta}}$} 

  Associated to $\Gamma$ there is, for given $d > 0$ and a given size
  for angles $\Theta$,  a finite dimensional simplicial complex $P_{d,\Theta}$,
  the relative Rips complex of $\Gamma$.
  The construction of $P_{d,\Theta}$ is reviewed 
  in Definition~\ref{def:relative-rips}.
  The union $\overline{P_{d,\Theta}} := P_{d,\Theta} \cup \dd \Gamma$ 
  carries a compact topology,
  see Lemma~\ref{lem:properties-of-overlineP_d,Theta}
    ~\ref{lem:overlineP:compact}.
  Theorem~\ref{thm:long-wide-GxDelta} has the following 
  straight forward extension  from $\Delta$ to $\overline{P_{d,\Theta}}$.

  \begin{corollary}
    \label{cor:wide-cover-relative-Rips}
    For any $d > 0$ and any size for angles $\Theta$ the action of $G$ on
    $\overline{P_{d,\Theta}}$ is finitely $\calp$-amenable.
  \end{corollary}

  \begin{proof}
    Given $\alpha > 0$ we need to construct a $\calp$-cover of
    $G \x \overline{P_{d,\Theta}}$ as in Definition~\ref{def:N-calf-amenable}.

    The action of $G$ on $P_{d,\Theta}$ is simplicial and 
    the dimension of $P_{d,\Theta}$ is finite by
    Lemma~\ref{lem:P_d,Theta-finite-dim+loc-finite}. 
    In particular, Remark~\ref{rem:equivariant-cover} applies and we 
    find a $\calp$-cover $\calw$ of $G \x P_{d,\Theta}$ of   
    dimension $N'$ such that 
    \begin{numberlist}
      \item[\label{nl:calw-dim}]
         $\dim \calw$ is bounded by the dimension of $P_{d,\Theta}$;
      \item[\label{nl:calw-wide}]
         for any $(g,x) \in G \x P_{d,\Theta}$ there
         is $W \in \calw$ such that $G \x \{ x \} \subseteq W$.
    \end{numberlist}
    The compact topology on $\overline{P_{d,\theta}}$ is such that 
    $P_{d,\Theta}$ is not necessarily open in $\overline{P_{d,\theta}}$,
    but $P_{d,\Theta} \setminus V_\infty$ is open in $\overline{P_{d,\theta}}$.
    Therefore, the collection 
    $\{ W \setminus G \x  V_\infty \mid W \in \calw \}$
    consists of open $\calp$-subsets, satisfies~\eqref{nl:calw-dim}  and 
    satisfies~\eqref{nl:calw-wide} for all 
    $(g,x) \in G \x (P_{d,\Theta} \setminus V_\infty)$.

    Given $\alpha > 0$, Theorem~\ref{thm:long-wide-GxDelta} provides a 
    $\calp$-cover $\calv$ for $G \x \Delta$  such that
    \begin{numberlist}
      \item[\label{nl:calv-dim}]
         $\dim \calv$ is bounded by a number independent of $\alpha$;
      \item[\label{nl:calv-wide}]
         for any $(g,\xi) \in G \x \Delta$ there
         is $V \in \calv$ such that $B_\alpha(g) \{ \xi \} \subseteq V$.  
    \end{numberlist}
    Since $\overline{P_{d,\Theta}}$ is compact with a countable
    basis for the topology  it is also metrizable.
    Therefore we find a $G$-invariant metric $d_{G \x \overline{P_{d,\Theta}}}$ on 
    $G \x \overline{P_{d,\Theta}}$.
    Using Lemma~\ref{lem:extend-cover} we can extend $\calv$
    to $G$-invariant collection $\calv^+$ of $\calp$-subset of 
    $G \x \overline{P_{d,\Theta}}$
    that still satisfies~\eqref{nl:calv-dim} and~\eqref{nl:calv-wide}.     
    Altogether $\{ W \setminus G \x  V_\infty \mid W \in \calw \}
        \cup  \calv^+$
    is the desired cover of $G \x \overline{P_{d,\Theta}}$.
  \end{proof}

  \section{The Farrell-Jones Conjecture for relatively hyperbolic groups}
    \label{sec:FJ-rel-hyp}

  Let $G$ be a group and $\cala$ be an additive category with a 
  strict $G$-action and a strict direct sum. 
  In~\cite[Sec.~4.1]{Bartels-Lueck(2012annals)} such categories are 
  called additive $G$-categories.
  Similar to (twisted) group rings there is an 
  additive category $\int_G \cala$ whose morphisms are formal 
  linear combinations of
  group elements with morphisms from $\cala$ as coefficients\footnote{
  In~\cite[Sec.~2]{Bartels-Reich(2007coeff)} this category 
  is denoted $\cala*_G G/G$.}.
  Given a family $\calf$ of subgroups of $G$ there is the assembly map
  \begin{equation}
    \label{eq:K-theory-assembly-map}
    H_*^G(E_\calf G; \bfK_\cala) \to K_*(\intgf{G}{\cala}).
  \end{equation}
  The $K$-theoretic Farrell-Jones Conjecture (with  coefficients) 
  asserts that this map is an isomorphism for the family 
  $\calf := \VCyc$ of virtually
  cyclic subgroups of $G$~\cite[Conj.~3.2]{Bartels-Reich(2007coeff)}. 
  The original formulation of the Conjecture in~\cite{Farrell-Jones(1993a)} 
  can be recovered as a special case of~\eqref{eq:K-theory-assembly-map} 
  by using for $\cala$
  the category of finitely generated free $\IZ$-modules; then 
  $\int_G \cala$ is equivalent to the category of finitely generated free 
  $\IZ[G]$-modules.
  We will say that 
  \emph{$G$ satisfies the $K$-theoretic Farrell-Jones Conjecture 
  relative to $\calf$} 
  if~\eqref{eq:K-theory-assembly-map} is an isomorphism for all 
  additive  $G$-categories $\cala$.
  A consequence of the transitivity principle 
  (see~\cite[Thm.~A.10]{Farrell-Jones(1993a)} 
      and~\cite[Thm.~2.10]{Bartels-Farrell-Lueck(cocompact-lattices)}) 
  is the following: Suppose that $G$ satisfies the $K$-theoretic 
  Farrell-Jones Conjecture 
  relative to $\calf$ and that every group $F \in \calf$ satisfies the
  $K$-theoretic Farrell-Jones Conjecture. 
  Then $G$ satisfies the $K$-theoretic Farrell-Jones Conjecture.    

  If $\cala$ is in addition equipped with a strict involution, 
  see for example~\cite[Sec.~4.1]{Bartels-Lueck(2012annals)}, 
  then $\int_G \cala$ inherits 
  an involution and there is the $L$-theoretic assembly map
  \begin{equation}
    \label{eq:L-theory-assembly-map}
    H_*^G(E_\calf G; \bfL^{-\infty}_\cala) \to 
                 L_*^{\langle -\infty \rangle}(\intgf{G}{\cala}).
  \end{equation}
  The $L$-theoretic Farrell-Jones Conjecture (with coefficients) asserts 
  that this map
  is an isomorphism for the family of virtually cyclic subgroups $\VCyc$.
  Everything said for $K$-theory above has an $L$-theory counterpart.
  In particular, we will say that 
  \emph{$G$ satisfies the $L$-theoretic Farrell-Jones Conjecture 
    relative to $\calf$} 
  if~\eqref{eq:L-theory-assembly-map} is an isomorphism for all 
  additive $G$-categories $\cala$ with involution.
  
  Next we give a minor reformulation of conditions 
  from~\cite{Bartels-Lueck(2012annals),Bartels-Lueck-Reich(2008hyper)} 
  that implies the Farrell-Jones Conjecture relative to $\calf$.  
  For a family $\calf$ of groups we write $\calf_2$ for the family of groups 
  that contain a group from $\calf$ as a subgroup of index $\leq 2$.

  \begin{theorem}
    \label{thm:conditions-for-FJ}
    Let $G$ be a group that admits a finitely $\calf$-amenable action
    on a compact contractible finite-dimensional $\ANR$ $\overline{X}$.
    Then
    \begin{enumerate}
      \item  \label{thm:cond-FJ:K}
         $G$ satisfies the $K$-theoretic Farrell-Jones Conjecture  
         relative to $\calf$;
      \item \label{thm:cond-FJ:L}
         $G$ satisfies the $L$-theoretic Farrell-Jones Conjecture  
         relative to $\calf_2$.
    \end{enumerate}   
  \end{theorem}

  \begin{proof}
    We start by recalling~\cite[Def.~1.5]{Bartels-Lueck(2012annals)} 
    where a metric space $Y$ 
    is called \emph{controlled $N$-dominated}
    if for every $\e > 0$ there is a finite $CW$-complex 
    $K$ of dimension at most $N$, maps $i \colon Y \to K$, 
    $p \colon K \to Y$ and a homotopy $H \colon Y \x [0,1] \to Y$
    between $p \circ i$ and $\id_Y$ such that the tracks
    $\{ H(y,t) \mid t \in [0,1] \}$ are of diameter at most $\e$.  
    Compact finite dimensional $\ANR$s have this 
    property; this follows for example 
    from~\cite[Thm.~10.1, p.122]{Borsuk(Theory-of-retracts)}.   
    \\[1ex]
    \ref{thm:cond-FJ:K} 
    In~\cite[Thm.~1.1]{Bartels-Lueck-Reich(2008hyper)} the 
    $K$-theoretic Farrell-Jones Conjecture relative to $\calf$ 
    is proven under very similar
    conditions.
    In this reference also an action of $G$ on a compact 
    metrizable space $\overline{X}$ is required. 
    Assumption~1.4 in this reference is what we defined as
    finitely $\calf$-amenability here.
    In this reference it is further assumed that $\overline{X}$ 
    contains a simplicial complex $X$ whose complement is 
    a $Z$-set in $\overline{X}$. 
    This further assumption is only used in the proof 
    of~\cite[Lem.~6.9]{Bartels-Lueck-Reich(2008hyper)}.
    It is not hard to see that this Lemma also holds for
    controlled $N$-dominated metric spaces,
    compare~\cite[Lem.~8.4]{Bartels-Lueck(2012annals)}. 
    Therefore, the $Z$-set assumption can be safely replaced with
    the assumption that $\overline{X}$ is a 
    compact finite-dimensional $\ANR$.
    Alternatively, \ref{thm:cond-FJ:K} can be deduced 
    from~\cite[Thm~1.1]{Wegner(2012CAT0)}. 
    The conditions given in this reference are more involved (using
    strong homotopy actions) and designed for more general situations,
    but can be checked to hold in our case.  
    \\[1ex]
    \ref{thm:cond-FJ:L}
    This follows from~\cite[Thm.~1.1(ii)]{Bartels-Lueck(2012annals)}.
    The assumption in this reference is that $G$ is
    \emph{transfer reducible over $\calf$}
        ~\cite[Def.~1.8]{Bartels-Lueck(2012annals)}.
    To check that this assumption is satisfied in our situation 
    we use the action of $G$ on $\overline{X}$.
    The covers appearing in~\cite[Def.~1.8]{Bartels-Lueck(2012annals)}
    exist because this action is finitely $\calf$-amenable.
  \end{proof}

  \begin{theorem}
     \label{thm:FJ-rel-hyp}
    Let $\calp$ be a family of subgroups of $G$ that contains
    all virtually cyclic subgroups.
    Suppose that $G$ is relatively hyperbolic to
    subgroups $P_1,\dots,P_n \in \calp$.
    Then
    \begin{enumerate}
      \item  \label{thm:FJ-rel-hyp:K}
         $G$ satisfies the $K$-theoretic Farrell-Jones Conjecture  
         relative to $\calp$;
      \item \label{thm:FJ-rel-hyp:L}
         $G$ satisfies the $L$-theoretic Farrell-Jones Conjecture  
         relative to $\calp_2$.
    \end{enumerate}  
  \end{theorem}
   
  \begin{proof}
    Let $\overline{P_{d,\Theta}}$ be the space from 
    Definition~\ref{def:compactification}.
    By Corollary~\ref{cor:wide-cover-relative-Rips} 
    this space carries a finitely $\calp$-amenable action.
    It is a finite-dimensional compact contractible $\ANR$ 
    for suitable $d,\Theta$ by 
    Lemma~\ref{lem:properties-of-overlineP_d,Theta}
    and Theorem~\ref{thm:overlineP-is-ANR}.  
    Now use Theorem~\ref{thm:conditions-for-FJ}.
  \end{proof}
    
  \begin{remark}
    In Theorem~\ref{thm:FJ-rel-hyp} it suffices to assume 
    that $\calp$ contains all virtually cyclic subgroups of type $I$, instead of
    all virtually cyclic subgroup.
    The reason for this is that no virtually cyclic subgroups of type $II$
    appeared in the previous Sections, see Remark~\ref{rem:type-I}.

    For $L$-theory this does not strengthen Theorem~\ref{thm:FJ-rel-hyp},
    since all virtually cyclic groups contain a virtually cyclic group of
    type $I$ as a subgroup of index at most $2$.   
       
    For $K$-theory this slightly strengthens Theorem~\ref{thm:FJ-rel-hyp}.
    However, it is known that such a strengthening is always possible:
    in the $K$-theoretic Farrell-Jones Conjecture only virtually cyclic
    subgroups of type $I$ are needed~\cite{Davis-Quinn-Reich(2011)}.
  \end{remark}

  \begin{corollary}
     \label{cor:FJ-rel-hyp}
    Let $G$ be relatively hyperbolic to subgroups $P_1,\dots,P_n$. Then
    \begin{enumerate}
      \item \label{cor:FJ-rel-hyp:K} 
         $G$ satisfies the $K$-theoretic Farrell-Jones Conjecture,
         provided that
         the $P_i$ satisfies the $K$-theoretic Farrell-Jones Conjecture;
      \item \label{cor:FJ-rel-hyp:L} 
         $G$ satisfies the $L$-theoretic Farrell-Jones Conjecture, 
         provided that all subgroups of $G$ that contain one of the 
         $P_i$ as a subgroup of index at 
         most $2$ satisfy the $L$-theoretic Farrell-Jones Conjecture. 
    \end{enumerate}
  \end{corollary}
    
  \begin{proof}
    This follows from Theorem~\ref{thm:FJ-rel-hyp} 
    and the transitivity principle reviewed 
    earlier in this section.
  \end{proof}

  \begin{remark}
    Many classes of groups that are known to satisfy the 
    Farrell-Jones Conjecture 
    are closed under finite index overgroups, but there is no 
    general result to  this effect.
    A good formalism to circumvent this difficulty is the 
    Farrell-Jones Conjecture with wreath products considered for example
    in~\cite{Bartels-Lueck-Reich-Rueping(GLnZ),Farrell-Roushon(2000)}.
 
    A group $G$ is said to satisfy the Farrell-Jones Conjecture 
    with wreath products relative to a family of subgroups $\calf$ 
    if for any finite group $F$ the 
    wreath product $G \wr F$ satisfies the Farrell-Jones Conjecture relative
    to $\calf$. 
    This version of the conjecture passes to finite 
    overgroups and to finite wreath 
    products~\cite[Rem.~6.2]{Bartels-Lueck-Reich-Rueping(GLnZ)}. 
    Moreover, the conditions 
    that we verified for relatively hyperbolic groups in the proof
    of Theorem~\ref{thm:FJ-rel-hyp} can also be used to obtain
    results for the Farrell-Jones Conjecture with wreath 
    products~\cite[Thm.~5.1]{Bartels-Lueck-Reich-Rueping(GLnZ)}.
    Combining this observation with the transitivity principle we obtain the
    following variant of Corollary~\ref{cor:FJ-rel-hyp}:
    \emph{ Suppose that $G$ is realtively hyperbolic to subgroups 
      $P_1,\dots,P_n$
      all of which satisfy the Farrell-Jones Conjecture with wreath products,
      then $G$ satisfies the Farrell-Jones Conjecture with wreath products.}

    Here it is no longer necessary to distinguish between $K$-and $L$-Theory;
    everything in this remark applies as stated to the $K$-theoretic
    and the $L$-theoretic version of the conjecture.
  \end{remark}

  \begin{remark}
    If $G$ is relatively hyperbolic to infinite subgroups $P_1,\dots,P_n$, then for the action of $G$ on $\Delta$ each $P_i$ fixes a unique point in $\Delta$.
    In particular, no overgroups of the $P_i$ have fixed points on the space $P_2(\Delta)$ of unordered pairs in $\Delta$ used  in~\cite[Sec.~9]{Bartels-Lueck(2012annals)}. 
    It seems plausible that this observation together with a careful analysis of the arguments in~\cite{Bartels-Lueck(2012annals)} can be used to show that the appearance of index $2$ overgroups  in Theorem~\ref{thm:FJ-rel-hyp}~\ref{thm:FJ-rel-hyp:L} and Corollary~\ref{cor:FJ-rel-hyp}~\ref{cor:FJ-rel-hyp:L} is not necessary.
  \end{remark}

  \noindent
  We conclude this section with some examples where Corollary~\ref{cor:FJ-rel-hyp}
  applies.

  \begin{example}
    Let $G$ be the fundamental group of a complete Riemannian manifold $M$ of pinched
    negative curvature and finite volume.
    Then $G$ is hyperbolic relative to virtually 
    finitely generated nilpotent groups~\cite{Bowditch(Rel-hyperbolic-groups),Farb(1998)}.
    Since virtually finitely generated nilpotent groups satisfy both the 
    $K$- and $L$-theoretic 
    Farrell-Jones Conjecture~\cite{Bartels-Farrell-Lueck(cocompact-lattices)},
    the $K$- and $L$-theoretic Farrell-Jones Conjecture for $G$ holds. 
  \end{example}

  \begin{example}
    Let $G$ be the fundamental group of a finite graph of groups with finite edge groups.
    Then the associated action of $G$ on the Bass-Serre tree reveals $G$ as relatively
    hyperbolic to the vertex groups.
    Thus if all vertex groups satisfy the $K$-theoretic Farrell-Jones Conjecture then 
    $G$ satisfies the $K$-theoretic Farrell-Jones Conjecture.
    If all overgroups of vertex groups of index at most $2$ satisfy the $L$-theoretic
    Farrell-Jones Conjecture then $G$ satisfies the $L$-theoretic Farrell-Jones Conjecture.
  \end{example}

  \begin{example}
    Suppose that $G$ acts cocompactly and properly discontinuously on a systolic
    complex with the Isolated Flats Property. 
    Then $G$ is relatively hyperbolic to virtually finitely generated abelian 
    subgroups~\cite{Elsner(Systolic-isolated-flats)}. 
    Since virtually finitely generated abelian groups satisfy  both the 
    $K$- and $L$-theoretic 
    Farrell-Jones Conjecture~\cite{Bartels-Farrell-Lueck(cocompact-lattices),
                Quinn(2012virtab)},
    the $K$- and $L$-theoretic Farrell-Jones Conjecture for $G$ holds. 
  \end{example}

  \appendix

  \section{The relative Rips complex and the boundary}
     \label{app:Z-set}

  In this appendix we prove that $\Delta$ is finite dimensional and embeds into a finite dimensional compact contractible $\ANR$ with a complement homeomorphic to a simplicial complex\footnote{The complement will be the relative Rips complex minus its vertices of infinite valence. Thus the complement is not a subcomplex of the relative Rips complex, but it is homeomorphic to a simplicial complex}.  
  For the boundary of hyperbolic groups both these facts are 
  well-known~\cite{Bestvina-Mess(1991)}.
  For relative hyperbolic groups closely related
  results have been obtained
  by Dahmani~\cite{Dahmani(2003class-spaces+boundaries-rel-hyp)} 
  and Mineyev-Yaman~\cite{Mineyev-Yaman(rel-hyperbolic-bounded-cohom)}.
  Our treatment is very similar to the one in these references,
  but we do not require any assumptions on the parabolic subgroups.

  Throughout this appendix we use again the notation from
  Section~\ref{sec:relative-hyperbolic-groups}.
  In particular, $G$ is a relatively hyperbolic group
  exhibited by a cocompact action on the
  fine and hyperbolic graph $\Gamma$ with finite edge stabilizers. 
  Throughout this appendix we will make the following additional
  assumption on $\Gamma$:
  no two vertices from $V_\infty$ are adjacent.
  This can be easily arranged for by replacing $\Gamma$ with
  its first barycentric subdivision.

  \subsection*{$\Delta_+$ is finite dimensional.}

  \begin{theorem}
    \label{thm:dim-delta-finite}
    The dimension of $\Delta_+$ is finite.
  \end{theorem}

  \noindent
  A very similar result by Dahmani 
  is~\cite[Lemma~3.7]{Dahmani(2003class-spaces+boundaries-rel-hyp)},
  whose proof we mostly copy.

  \begin{proof}
    For a vertex $v$ of finite valency 
    let $\dd_{v} \Gamma \subseteq \dd \Gamma$ consist of all $\xi \in \dd \Gamma$
    for which there exists a $\Theta^{(3)}$-small geodesic from $v$ to $\xi$.
    Let $U$ be the union of all $\dd_{v} \Gamma$, where we vary $v$ over
    all vertices of finite valency.
    If $\xi \in \dd \Gamma \setminus U$, then any geodesic from a vertex
    of finite valency to $\xi$ will have a $\Theta^{(3)}$-large angle at
    infinitely many vertices.  
    The finite union of finite dimensional spaces is again finite 
    dimensional, see~\cite[p.28]{Hurewicz-Wallman(DimensionTheory)}.
    Therefore it suffices to show that $V$, $U$ and $\dd \Gamma \setminus U$
    are finite dimensional subspace of $\Delta_+$.

    We recall again the countable sum 
    theorem~\cite[Thm.~2.5, p.125]{Pears(Dimension-theory)}: 
    the countable union of closed subsets of dimension $\leq n$ is
    of dimension $\leq n$.
    In particular, the countable subspace $V \subseteq \Delta_+$
    is of dimension $0$.
    The spaces $\dd_v \Gamma$ are finite dimensional by 
    Lemma~\ref{lem:dim-dd_v-Gamma} below with a uniform bound on 
    their dimensions.
    As a consequence of Addendum~\ref{add:compactness} 
    the $\dd_v \Gamma$ are closed in $\dd \Gamma$.
    Thus $U$ is finite dimensional by the countable sum theorem.
 
    It remains to show that  $\dd \Gamma \setminus U$ is finite dimensional.
    In fact, we will show that it is homeomorphic to a subset of the boundary
    of a tree $T$ and therefore $0$-dimensional. 
    Fix a vertex $v_0$ of finite valency. 
    The tree $T$ is a maximal subtree of $\Gamma$ and can be build inductively
    by choosing, for each vertex at distance $n$ from $v_0$, an edge of 
    $\Gamma$ that connects it to a vertex at distance $n-1$ from $v_0$.
    For any $\xi \in \dd \Gamma$ the tree $T$ will contain 
    a geodesic from $v_0$ to $\xi$.
    (To construct such a geodesic pick a sequence of vertices $v_n$ with 
    $v_n \to \xi$ and apply Lemma~\ref{lem:compactness}~\ref{lem:compactness:boundary} 
    to geodesics $c_n$ from $v_0$ to $v_n$ in $T$.)    
    The inclusion $T \to \Gamma$ induces a continuous surjective
    map $f \colon \Delta_+(T) \to \Delta_+$, where $\Delta_+(T)$
    is the union of the vertices of $T$ with $\dd T$ and
    is also equipped with the observer topology.
    For any $\xi \in \dd \Gamma \setminus U$ there is a unique
    geodesic in $T$ from $v_0$ to $\xi$.
    Indeed, any geodesic $c$ in $\Gamma$ from $v_0$ to $\xi$ 
    will have a $\Theta^{(3)}$-large angle at
    infinitely many vertices and any other geodesic $c'$ in $\Gamma$
    from $v_0$ to $\xi$ will, by Lemma~\ref{lem:large-angles},
    also pass through these vertices. 
    Therefore $c=c'$ if both are in $T$.
    It follows that the restriction of $f$ to the preimage of
    $\dd \Gamma \setminus U$ is bijective.
    We claim that, since $\Delta_+(T)$ is compact, 
    this restriction of $f$ is a homeomorphism. 
    To prove this claim 
    we need to show that the inverse of $f$ on $\dd \Gamma \setminus U$
    is continuous.
    Let $\xi_n \to \xi$ be a convergent sequence in $\dd \Gamma \setminus U$.
    Let $\xi'_n, \xi'$ be the unique preimages in $\Delta_+(T)$ 
    of the $\xi_n$, $\xi$.
    We need to show $\xi'_n \to \xi'$ in $\Delta_+(T)$.
    Assume this fails.
    Then, as $\Delta_+(T)$ is compact, there is a subsequence $I \subseteq \IN$ 
    with $\lim_{n \in I} \xi'_n = \xi'' \neq \xi'$ in $\Delta_+(T)$.
    By continuity of $f$, we have 
    $f(\xi'') = \lim_{n \in I} f(\xi'_n) = \lim_{n \in I} \xi_n = \xi$. 
    In particular, $\xi''$ belongs to the preimage of $\dd \Gamma \setminus U$
    under $f$.
    Now we use that the restriction of $f$ to this preimage is injective 
    to contradict $\xi'' \neq \xi'$.
    Thus $\xi'_n \to \xi'$ in $\Delta_+(T)$.
  \end{proof}

  \begin{lemma} \label{lem:dim-dd_v-Gamma}
    There is $N$ such that $\dim \dd_{v_0} \Gamma \leq N$ for all vertices $v_0$
    of finite valency.
  \end{lemma}

  \begin{proof}
    Let $V_{v_0}$ be the set of all vertices $v$ of $\Gamma$ for which there exists 
    a $\Theta^{(3)}$-small geodesic from $v_0$ to $v$.
    Hyperbolicity of $\Gamma$ implies that this is a quasi-convex subset of $\Gamma$:
    there is $r > 0$ such that
    for $v, v' \in V_{v_0}$ any geodesic between $v$ and $v'$ will be contained in the 
    $r$-neighborhood of the union of geodesics from $v_0$ to $v$ and $v'$.
    Consequently, $V_{v_0}$ is hyperbolic in the 
    metric $d_{v_0}$ inherited as a subspace of $\Gamma$.
    Let $\Gamma_{v_0}$ be the graph whose vertices are $V_{v_0}$ and 
    for which there is an edge between $v$ and $v'$ whenever $d_{v_0}(v,v') \leq 2r + 1$.
    Of course, the metric $d_{\Gamma_{v_0}}$ induced by $\Gamma_{v_0}$ on $V_{v_0}$ satisfies
    $d_{v_0} \leq (2r+1) d_{\Gamma_{v_0}}$.
    Conversely, since any $\Theta^{(3)}$-small geodesic in $\Gamma$ starting in $v_0$ defines also a
    path in $\Gamma_{v_0}$, we can as before apply hyperbolicity to any geodesic $c$ in $\Gamma$ between points $v, v' \in V_{v_0}$ and replace $c$ with a path in $\Gamma_{v_0}$ of length equal to the length of $c$.
    Thus $d_{\Gamma_{v_0}}$ and $d_{v_0}$ are Lipschitz equivalent. 
    In particular, $\Gamma_{v_0}$ is hyperbolic.
    
    Let $v, v' \in V_{v_0}$.
    Let $\Theta$ be a size for angles such that $3\Theta^{(3)} \subseteq \Theta$.
    Since $v_0$ is a vertex of finite valency we can, after enlarging $\Theta$, assume
    that all angles at the vertex $v_0$ are $\Theta$-small. 
    We claim that any geodesic $c$ in $\Gamma$ between $v$ and $v'$ is $\Theta$-small.
    If this fails, then for some internal vertex $w$ of $c$ the angle $\varangle_w c$
    is $\Theta$-large.
    Let $c_v$ and $c_{v'}$ be $\Theta^{(3)}$-small geodesics from $v_0$ to $v$ and $v'$.
    We apply Lemma~\ref{lem:large-angles-in-triangles-no-c} to the geodesic triangle  
    whose sides are $c$, $c_v$ and $c_{v'}$.
    Since $c_v$ and $c_v'$ are $\Theta^{(3)}$-small we conclude that $w = v_0$.
    But our choice of $\Theta$ guarantees that $\varangle_w c$ is $\Theta$-small.
    Therefore $c$ is $\Theta$-small.    

    For $v \in V_{v_0}$ the ball $B_r(v)$ of radius $r$ around $v$ in $V_{v_0}$
    with respect to $d_{v_0}$
    is contained in the set of vertices of $\Gamma$ that can be connected to $v$
    by a $\Theta$-small geodesic of length $r$. 
    If $v$ is of finite valency, then this set is finite by 
    Corollary~\ref{cor:finite-theta-balls}. 
    Moreover, as the action of $G$ on the vertices of $\Gamma$ is cofinite,
    the number of vertices in $B_r(v)$ is bounded by a number depending
    only on $r$.
    Any vertex $w \in V_{v_0}$ of infinite valency 
    is adjacent to a vertex $v \in V_{v_0}$ of finite valency.
    (Consider the $\Theta^{(3)}$-small geodesic between $w$ and $v_0$ and
    recall that we assumed that no two vertices of infinite valency are adjacent.)
    Since $B_r(w) \subseteq B_{r+1}(v)$ whenever $v,w \in V_{v_0}$
    are adjacent, we can now conclude that $V_{v_0}$ is 
    uniformly proper with respect to $d_{v_0}$ (or $d_{\Gamma_{v_0}}$):
    the number of elements in a ball is bounded by a number only depending on the
    radius of the ball.
    The dimension of the boundary of a hyperbolic graph  can be estimated in terms of
    the number of vertices in balls of a fixed radius (depending on the hyperbolicity constant);
    this is a standard fact see for example~\cite[Proof of Prop.9.3(ii)]{Bartels-Lueck-Reich(2008cover)}. 
    Consequently, the boundary of $\Gamma_{v_0}$ is finite dimensional.
    Since this boundary agrees with $\dd_{v_0} \Gamma$ it follows that
    $\dd_{v_0} \Gamma$ is finite dimensional.
    The action of $G$ on the set of vertices of finite valency is cocompact. 
    It follows that the maximum of the dimension of the $\dd_{v_0} \Gamma$ is finite.
  \end{proof}

  \subsection*{The relative Rips complex.}
 
  We will say that a geodesic is $(d,\Theta)$-small
  if it is $\Theta$-small and of length at most $d$. 

  \begin{definition} \label{def:relative-rips}
    Let $\Theta$ be a size for angles and  $d > 0$.
    The relative Rips complex $P_{d,\Theta}$ of $\Gamma$ has
    $V$ as the set of vertices.
    A finite set $\sigma$ of vertices spans a simplex for 
    $P_{d,\Theta}$ if and only if between any two vertices
    $v,v' \in \sigma$ there exists a $(d,\Theta)$-small geodesic. 
  \end{definition}

  \begin{lemma}
    \label{lem:P_d,Theta-finite-dim+loc-finite}
    The relative Rips complex $P_{d,\Theta}$ is
    finite dimensional.
    For $n \geq k \geq 1$ each $k$-simplex is contained in only finitely many 
    $n$-simplices.
  \end{lemma}

  \begin{proof}
    Let $v \in V$ and $e$ be an edge incident to $v$.
    Because of Lemma~\ref{lem:angles-locally-finite} 
    there are only finitely many vertices $v'$ for which there exists
    a $(d,\Theta)$-small geodesic from $v$ to $v'$
    whose first edge is $e$.
    Since $G$ acts cocompactly on the set of edges the number of
    such geodesics is uniformly bounded.

    Fix $v' \neq v$ such that there exists a $\Theta$-small geodesic $c$
    between $v$ and $v'$.
    Let $W(v,v')$ be the set of vertices $w$ for which
    there exist $(d,\Theta)$-small geodesics,
    $c_w$ between $w$ and $v$ and $c'_w$ between $w$ and $v'$.
    We need to show that the number of elements in $W(v,v')$ 
    is uniformly bounded.
    Let $e$ be the initial edge of $c$, starting at $v$.
    Let $f$ be the initial edge of $c_w$, starting at $v$.
    We claim that $(e,f)$ is $\Theta+2\Theta^{(3)}$ small.
    Lemma~\ref{lem:angles-locally-finite} implies then that
    the number of such $f$ is uniformly bounded and the first paragraph
    of this proof implies then that the cardinality of $W(v,v')$ is
    uniformly bounded.
    
    It remains to prove that $(e,f)$ is $\Theta+2\Theta^{(3)}$ small.
    If $c'_w$ does not meet $v$, then $(e,f)$ is $\Theta^{(3)}$ small.
    Otherwise $\varangle_v c'_w$ is $\Theta$-small and
    Lemma~\ref{lem:geodesic-2-gons} implies then that 
    $(e,f)$ is $\Theta + 2\Theta^{(3)}$-small.
  \end{proof}

  \noindent
  In the next statement the positive integer $\delta$ will be again 
  a hyperbolicity constant for $\Gamma$.

  \begin{proposition}
    \label{prop:contract-P_d,Theta}
    Assume that $d \geq 4 \delta$ and $7\Theta^{(3)} \subseteq \Theta$.  
    Then $P_{d,\Theta}$ is contractible.
    More precisely the following holds:

    Let $K$ be a finite subcomplex of $P_{d,\Theta}$.
    Let $L$ be the subcomplex of $P_{d,\Theta}$ spanned by all
    vertices $\tilde v$ for which there are vertices $v,v' \in K$
    such that $\tilde v$ belongs to a geodesic between $v$ and $v'$.
    Then $K$ is contractible within $L$.
  \end{proposition}

  \begin{proof}
    Let $V_K$ be the set of vertices of $K$.
    Fix $v_0 \in V_K$.
    Let $\alpha := \max_{v \in V_K} d_\Gamma(v_0,v)$.
    Let $a$ be the number of vertices in $V_K$ for which
    $d_\Gamma(v_0,v) = \alpha$.
    For $v \in V_K$ let $W(v)$ be the set of vertices
    $w \in V$ such that there exists a geodesic $c$ from
    $v_0$ to $v$ such that
    $c$ passes through $w$ such that $\varangle_w c$
    is $2\Theta^{(3)}$-large.
    Let $\beta := \max_{v \in V_K, w \in W(v)} d_\Gamma(v_0,w)$.
    Let $b$ be the number of vertices $v \in V_K$ for which
    there is $w \in W(v)$ with $d_\Gamma(v_0,w) = \beta$.
    Clearly, $\alpha \geq \beta$.
    If $\alpha = 0$, then $K$ consists of a single vertex and
    there is nothing to prove.

    For the general case we will use induction on 
    $(\alpha + \beta, a + b)$.
    We claim that there is a vertex $v \in V_K$ and a vertex $\tilde v$ 
    on a geodesic from $v_0$ to $v$ such that:
    \begin{numberlist}
    \item[\label{nl:v-to-tilde-v}] 
       there is a $(d,\Theta)$-small geodesic from $v$ to $\tilde v$;
    \item[\label{nl:v'-to-tilde-v}] 
       if there is a $(d,\Theta)$-small geodesic from $v$ to $v'$,
       with $v' \in V_K$, then there is also
       a $(d,\Theta)$-small geodesic from $\tilde v$ to $v'$;
    \item[\label{nl:smaller-numbers}] either $d_\Gamma(v_0,v) = \alpha$ and 
      $d_\Gamma(v_0,\tilde v) < \alpha$ or there is $w \in W(v)$
      with $d_\Gamma(v_0, w) = \beta$, and 
      $d_\Gamma(v_0,\tilde w) < \beta$ for all $\tilde w \in W(\tilde v)$.
      (In fact, in the second case we will use $\tilde v = w$.)  
    \end{numberlist}
    Given this claim there is a homotopy in $L$ that replaces 
    $v$ by $\tilde v$.
    The effect on $(\alpha + \beta, a + b)$ is then that we either
    reduce $\alpha + \beta$ or that we reduce $a+b$ but do not 
    increase $\alpha + \beta$. 
    This will complete the induction step modulo our claim.
     
    To prove the claim, we distinguish two cases.
    In the first case we assume $\alpha \geq \beta + d$.
    Then choose $v \in V_K$ such that $d_\Gamma(v_0,v) = \alpha$.
    As $\delta$ is an integer, we can pick a vertex $\tilde v$ 
    with $d_\Gamma(v,\tilde v) = 2 \delta$ that
    belongs to a geodesic $c$ between $v_0$ and $v$.
    Clearly~\eqref{nl:smaller-numbers} holds.
    The geodesic $c|_{[v,\tilde v]}$ is of length $2\delta \leq d$. 
    Let $w$ be an internal vertex of $c|_{[v,\tilde v]}$.
    If  $\varangle_w c$ is $\Theta$-large,
    then, since $2\Theta^{(3)} \subseteq \Theta$, $w \in W(v)$ and
    $\beta \geq d_\Gamma(v_0,w) > d_\Gamma(v_0,\tilde v) 
                         \geq \alpha - 2\delta \geq \alpha - d$.
    Since this contradicts $\alpha \geq \beta + d$, the angle 
    $\varangle_w c$ at $w$ is $\Theta$-small.
    Thus $c|_{[v,\tilde v]}$ is  $\Theta$-small 
    and~\eqref{nl:v-to-tilde-v} holds.
    In fact, this argument proves that 
    $c|_{[v,\tilde v]}$ is  $2\Theta^{(3)}$-small. 
    To prove~\eqref{nl:v'-to-tilde-v} consider $v' \in V_K$ such that there
    exists a $(d,\Theta)$-small geodesic $c'$  between
    $v$ and $v'$.
    Hyperbolicity of $\Gamma$ implies that there is $w \in V$ with
    $d_\Gamma(\tilde v, w) \leq \delta$ such that $w$ belongs to
    a geodesic between $v_0$ and $v'$ or to a geodesic between
    $v$ and $v'$.
    In the first case 
    $d_\Gamma(w, v') = d_\Gamma(v_0, v') - d_\Gamma(w,v_0)
     \leq d_\Gamma(v_0, v') - ( d_\Gamma (v_0,\tilde v) - d_\Gamma(w,\tilde v) )
     \leq \alpha - ( \alpha - 2\delta - \delta ) = 3 \delta \leq d - \delta$.
    In the second case
    $d_\Gamma(w,v') = d_\Gamma( v,v') - d_\Gamma( w, v)
      \leq d_\Gamma(v,v') - (d_\Gamma(v,\tilde v) - d_\Gamma(w,\tilde v))
      \leq d - (2\delta - \delta) = d - \delta$.
    Thus $d_\Gamma(\tilde v, v') \leq d$ in both cases.
    Let $\tilde c$ be a geodesic from $v'$ to $\tilde v$.
    It remains to show that $\tilde c$ is $\Theta$-small.
    Assume this fails.
    Then there is an internal vertex $w$ of $\tilde c$ such that 
    $\varangle_w \tilde c$ is $\Theta$-large.
    We can then apply Lemma~\ref{lem:large-angles-in-triangles-no-c}
    to the geodesics $c|_{[v,\tilde v]}$, $c'$ and $\tilde c$.
    Since $c|_{[v, \tilde v]}$ is $2\Theta^{(3)}$-small and 
    $2\Theta^{(3)} + 2\Theta^{(3)} \subseteq \Theta$, 
    Lemma~\ref{lem:large-angles-in-triangles-no-c} implies 
    $w \in c'$.
    In particular, $d_\Gamma(w,v) < d$ and therefore 
    $d_\Gamma(v_0,w) \geq d_\Gamma(v_0,v) - d_\Gamma(w,v) > \alpha - d \geq \beta$.
    Let $c''$ be a geodesic from $v'$ to $v_0$.
    Consider the geodesic triangle whose sides are
    $\tilde c$, $c''$ and $c|_{[v_0,\tilde v]}$.
    Since $\varangle_w \tilde c$ is $\Theta$-large, we can again
    apply Lemma~\ref{lem:large-angles-in-triangles-no-c}
    and deduce $w \in c'' \cup c|_{[v_0,\tilde v]}$.
    Moreover, since $2 \cdot 2\Theta^{(3)} + 3\Theta^{(3)} \subseteq \Theta$
    we can combining this Lemma with Lemma~\ref{lem:tripod} and 
    deduce that the angle of $c''$ or of $c|_{[v_0,\tilde v]}$ at
    $w$ will be $2\Theta^{(3)}$-large. 
    In the first case $w \in W(v')$ in the second $w \in W(v)$.
    Both cases imply $d_\Gamma(w,v_0) \leq \beta$.
    But this is a contradiction since we proved earlier 
    $d_\Gamma(v_0,w) > \beta$.  

    It remains to prove our claim under the assumption $\alpha < \beta + d$.
    If $\beta = 0$ then there are $(d,\Theta)$-small geodesics
    between $v_0$ and any $v \in V_K$.
    Thus we can set $\tilde v := v_0$ if $\beta = 0$.
    So assume $\beta > 0$ and
    pick $v \in V_K$, a geodesic $c$ between $v_0$ and $v$ 
    and an internal vertex $\tilde v$ of $c$ such that 
    $\varangle_{\tilde v} c$ is $2\Theta^{(3)}$-large and
    $d_\Gamma(v_0,\tilde v) = \beta$.
    Clearly~\eqref{nl:smaller-numbers} holds.  
    The geodesic $c|_{[v,\tilde v]}$ is of length $\leq \alpha - \beta \leq d$. 
    Let $w$ be an internal vertex of $c|_{[v,\tilde v]}$.
    If $\varangle_w c$ is $\Theta$-large,
    then, since $2\Theta^{(3)} \subseteq \Theta$, $w \in W(v)$ and
    we obtain the contradiction
    $\beta \geq d_\Gamma(v_0,w) > d_\Gamma(v_0,\tilde v) = \beta$.                   
    Thus $c|_{[v,\tilde v]}$ is $\Theta$-small 
    and~\eqref{nl:v-to-tilde-v} holds.
    In fact, this argument proves that 
    $c|_{[v,\tilde v]}$ is $2\Theta^{(3)}$-small. 
    To prove~\eqref{nl:v'-to-tilde-v} let $v' \in V_K$ such that there
    exists a $(d,\Theta)$-small geodesic $c'$ between
    $v$ and $v'$.
    Let $c''$ be a geodesic between $v_0$ and $v'$.
    Consider the geodesic triangle whose sides are
    $c$, $c'$ and $c''$.
    Since $\varangle_{\tilde v} c$ is $2\Theta^{(3)}$-large,
    Lemma~\ref{lem:large-angles-in-triangles-no-c} implies
    $\tilde v \in c' \cup c''$.
    If $\tilde v \in c'$, then  $c'|_{[v',\tilde v]}$ is 
    a $(d,\Theta)$-small geodesic  between $v'$ and
    $\tilde v$ and~\eqref{nl:v'-to-tilde-v} holds.
    If $\tilde v \in c''$, then we use the restriction $c''|_{[v',\tilde v]}$.
    This restriction is $\Theta$-small, since otherwise we could
    argue as in the proof of~\eqref{nl:v-to-tilde-v} and find
    $w \in W(v')$ with $d_\Gamma(v_0,w) > \beta$.
    Similarly, in this case 
    $d_\Gamma(\tilde v, v') = d_\Gamma(v',v_0) - d_\Gamma(\tilde v, v_0) 
          \leq \alpha - \beta \leq d$.  
  \end{proof}

  \noindent
  From now on we will assume that $d$ and $\Theta$ satisfy the assumptions
  of Proposition~\ref{prop:contract-P_d,Theta}.

  \begin{definition} \label{def:compactification}
    Let $\overline{P_{d,\Theta}} := P_{d,\Theta} \cup \dd \Gamma$.
    For an open subset $U$ of $\Delta_+$ we define
    $P_{d,\Theta}(U) \subseteq P_{d,\Theta}$ as the subcomplex
    spanned by the vertices of $\Gamma$ that belong to $U$.
    For $\xi \in \overline{P_{d,\Theta}}$ we 
    define collections of subsets $\caln_\xi$ of 
    $\overline{P_{d,\Theta}}$ as follows.
    \begin{enumerate}
    \item \label{def:comp:P-V_inf} 
      If $\xi \in P_{d,\Theta} \setminus V_\infty$, then
      $\caln_\xi$ consists of all open neighborhoods $W$ of $\xi$
      in $P_{d,\Theta} \setminus V_\infty$.
    \item \label{def:comp:dd} 
      If $\xi \in \dd \Gamma$, then $\caln_\xi$ consists of all
      sets of the form $P_{d,\Theta}(U) \cup U$, where 
      $U$ is a  neighborhood of $\xi$ in $\Delta_+$. 
    \item \label{def:comp:V_inf} 
       If $\xi \in V_\infty$, then $\caln_\xi$ consists of all
       sets of the form $P_{d,\Theta}(U) \cup U \cup W$, 
       where $U$ is a neighborhood of $\xi$ in $\Delta_+$,
       and $W$ is a neighborhood of $\xi$ in $P_{d,\Theta}$ 
       (i.e., the intersection of $W$ with any simplex of
         $P_{d,\Theta}$ is open in the simplex). 
    \end{enumerate}
    We will use the topology on $\overline{P_{d,\Theta}}$ which  for
    $U \subseteq \overline{P_{d,\Theta}}$ is open 
    if and only if for every $\xi \in U$ there is 
    $N \in \caln_\xi$ with $N \subseteq U$.
  \end{definition}
  
  \begin{lemma}
    \label{lem:is-neighborhood}
    Each $N \in \caln_\xi$ is a (not necessarily open)
    neighborhood of $\xi$ in $\overline{P_{d,\Theta}}$.
  \end{lemma}

  \begin{proof}
    If $\xi \in P_{d,\Theta} \setminus V_\infty$, then $\caln_\xi$ consists of open subsets of $\overline{P_{d,\Theta}}$. 
    Before we discuss the cases $\xi \in \dd \Gamma$ and $\xi \in V_\infty$ we point out a collection of open subsets of $\overline{P_{d,\Theta}}$.
    Consider $U' \subseteq \Delta_+$ open, $W' \subseteq P_{d,\Theta}$ open and assume that $P_{d,\Theta}(U') \subseteq W'$ and that $V_\infty \cap W' \subseteq U'$.
    It is not difficult to check that then $W' \cup U'$ is open in $\overline{P_{d,\Theta}}$.
        
    Let $\xi \in \dd \Gamma$ and $N = P_{d,\Theta}(U) \cup U \in \caln_\xi$.
    Lemma~\ref{lem:small-at-infinity} implies that there exists a neighborhood $U'$ of $\xi$ in $\Delta_+$ with the following property: 
    if $v$ is a vertex in $\Gamma$ for which there exists a geodesic of length $\leq d$ starting in $v$ and ending in $U'$, then $v \in U$. 
    It follows that $P_{d,\Theta}(U')$ is contained in the interior $P^\circ_{d,\Theta}(U)$ of $P_{d,\Theta}(U)$.
    Set now $W' := P^\circ_{d,\Theta}(U) \setminus (V_\infty \setminus U')$, i.e., we remove all vertices of infinite valency from $P_{d,\Theta}(U)$ that do not belong to $U'$.
    Then $W' \cup U'$ satisfies the conditions from the first paragraph of the proof and thus is open in $\overline{P_{d,\Theta}}$.     
    Since $\xi \in W' \cup U' \subseteq N$, $N$ is a neighborhood of $\xi$.
 
    Let $\xi \in V_\infty$ and $N = P_{d,\Theta}(U) \cup U \cup W \in \caln_\xi$.
    Addendum~\ref{add:small-at-infinty-V_infty} implies that there exists a neighborhood $U'$ of $\xi$ in $\Delta_+$  with the following property: 
    if $v$ is a vertex in $\Gamma$ for which there exists a $(d,\Theta)$-small geodesic starting in $v$ and ending in $U' \setminus \xi$, then $v \in U$. 
    It follows that $P_{d,\Theta}(U') \setminus \{ \xi \}$ is contained in $P^\circ_{d,\Theta}(U)$.
    Thus $P_{d,\Theta}(U') \subseteq P^\circ_{d,\Theta}(U) \cup W$. 
    Set $W' := (P^\circ_{d,\Theta}(U) \cup W) \setminus (V_\infty \setminus U')$.
    Again $W' \cup U'$ is open in $\overline{P_{d,\Theta}}$.
    Since $\xi \in W' \cup U' \subseteq N$, $N$ is a neighborhood of $\xi$. 
  \end{proof}

  Lemma~\ref{lem:is-neighborhood} implies that the inclusions $\Delta \to \overline{P_{d,\Theta}}$ and $P_{d,\Theta} \setminus V_\infty \to \overline{P_{d,\Theta}}$ are homeomorphisms onto their images.
  The canonical map $i \colon P_{d,\Theta} \to \overline{P_{d,\Theta}}$ is continuous, but $i$ is not a homeomorphism onto its image (unless $V_\infty = \emptyset$).
  For the continuity of $i$ it is important to include the sets $W$ in the definition of the neighborhood basis $\caln_\xi$ for $\xi \in V_\infty$, see Definition~\ref{def:compactification}~\ref{def:comp:V_inf}.
  
  \begin{lemma}
    \label{lem:properties-of-overlineP_d,Theta}
    \begin{enumerate}
    \item \label{lem:overlineP:countable} 
      The topology on $\overline{P_{d,\Theta}}$ is second countable;
    \item \label{lem:overlineP:compact} 
      $\overline{P_{d,\Theta}}$ is compact and metrizable;
    \item \label{lem:overlineP:dim} 
      $\overline{P_{d,\Theta}}$ is finite dimensional;
    \item \label{lem:overlineP:loc-contr}
      for any $\xi \in \overline{P_{d,\Theta}}$, and every neighborhood
      $U$ of $\xi$
      in $\overline{P_{d,\Theta}}$ there is a neighborhood $U'$ of $\xi$ in
      $\overline{P_{d,\Theta}}$ such that for any map 
      $f \colon S^k \to i^{-1}(U')$ there is a map
      $\hat f \colon D^{k+1} \to i^{-1}(U)$ such that
      $\hat f|_{S^k} = f$.
    \end{enumerate}
  \end{lemma}
  
  \begin{proof}
    \ref{lem:overlineP:countable} 
    Since $G$ is countable and acts cocompactly on $\Gamma$ there are only countably many simplices in $P_{d,\Theta}$.
    By Lemma~\ref{lem:P_d,Theta-finite-dim+loc-finite} every point in $P_{d,\Theta} \setminus V_\infty$ is contained in only finitely many simplices.
    Therefore, $P_{d,\Theta} \setminus V_\infty$ is second countable and we need for a basis of the topology on $\overline{P_{d,\Theta}}$ only countable many open subsets as in~\ref{def:comp:P-V_inf} from Definition~\ref{def:compactification}.
    Similarly, since $\Delta$ is second countable we need only countable many open subsets as in~\ref{def:comp:dd}.
    Finally, for~\ref{def:comp:V_inf}, since $V_\infty$ is countable, we need to argue that each $\xi \in V_\infty$ has a countable neighborhood basis. 
    It suffices to vary $U$ over a countable neighborhood basis of $\xi \in \Delta$.
    Addendum~\ref{add:small-at-infinty-V_infty} implies that
    for each such $U$ almost all vertices of $\Gamma$ that are incident to $\xi$ (in $P_{d,\Theta}$) are contained in $U$.
    Using Lemma~\ref{lem:P_d,Theta-finite-dim+loc-finite} it follows that all but finitely many of the simplices of $P_{d,\Theta}$ that contain $\xi$ are contained in $P_{d,\Theta}(U)$.
    Therefore, for each $U$ it suffices to vary $W$ over countable collection of neighborhoods of $\xi \in P_{d,\Theta}$.  
    \\[1ex] 
    \ref{lem:overlineP:compact}
    Let $\xi \in \Delta_+$, $v \in V_\infty$.
    Using  Addendum~\ref{add:small-at-infinty-V_infty} 
    we find open neighborhoods
    $U$ of $\xi$ in $\Delta_+$ and $W$ of $v$ in $P_{d,\Theta}$
    such that $P_{d,\Theta}(U) \cap W = \emptyset$.   
    Combining this observation with the
    fact that $\Delta_+$ and $P_{d,\Theta}$ are 
    Hausdorff we see that $\overline{P_{d,\Theta}}$ is also
    Hausdorff.

    Since the topology on $\overline{P_{d,\Theta}}$ has a countable basis
    it suffices to prove sequential compactness. 
    Let $(x_n)_{n \in \IN}$ be a sequence in $\overline{P_{d,\Theta}}$.
    We will produce a convergent subsequence.
    If $x_n \in \dd \Gamma$ for infinitely many $n$, then we can use
    the compactness of $\Delta$.
    Therefore we assume that $x_n \in P_{d,\Theta}$ for all $n$.
    Since $P_{d,\Theta}$ is finite dimensional we find $k$ and 
    vertices $v_{0,n}, \dots, v_{k,n}$ such that for each $n$
    the point $x_n$ belongs to the simplex spanned by $v_{0,n},\dots,v_{k,n}$.
    Since $\Delta_+$ is compact we can assume that
    $\xi_j := \lim_n v_{j,n}$ exists in $\Delta_+$ for each
    $j$.
    We now apply Lemma~\ref{lem:small-at-infinity}
    and Addendum~\ref{add:small-at-infinty-V_infty}
    to the $v_{0,n}, \dots, v_{k,n}$ as sequences in $n$.
    It follows that either the $\xi_j =: \xi$ all coincide
    or that we find a subsequence $I \subseteq \IN$
    such that for each $j$, $v_{j,n} =: w_j$ is constant in $n \in I$.
    In the first case $\lim_{n \in \IN} x_n = \xi$.
    In the second case we find a subsequence of $x_n$
    that converges to a point on the simplex spanned by the $w_j$.
    
    Since $\overline{P_{d,\Theta}}$ is compact and has a countable basis 
    for the topology it is metrizable.
    \\[1ex]
    \ref{lem:overlineP:dim} 
    This follows from 
    Theorem~\ref{thm:dim-delta-finite} and 
    Lemma~\ref{lem:P_d,Theta-finite-dim+loc-finite}
    and the fact that the finite union of finite dimensional spaces
    is finite dimensional~\cite[p.28]{Hurewicz-Wallman(DimensionTheory)}.
    \\[1ex]
    \ref{lem:overlineP:loc-contr}
    If $\xi \in \overline{P_{d,\Theta}} \setminus \Delta = 
          P_{d,\Theta} \setminus V_\infty$, then 
    $\xi$ has arbitrarily small contractible neighborhoods.
    If $\xi \in \Delta$, then we  may assume that
    $P_{d,\Theta}(U_+) \subseteq U$ for some neighborhood $U_+$ of
    $\xi \in \Delta_+$.
    Using Lemma~\ref{lem:U-V} we find a smaller neighborhood
    $U'_+$ of $\xi$ such that all vertices on geodesics between 
    points of $U'_+$ belong to $U_+$.
    Proposition~\ref{prop:contract-P_d,Theta}
    implies that any finite subcomplex of $P_{d,\Theta}(U'_+)$ is
    contractible within $P_{d,\Theta}(U_+)$.
    If $\xi \in \dd \Gamma$,
    we can use the interior of $P_{d,\Theta}(U'_+) \cup U'_+$ for $U'$.
    If $\xi \in V_\infty$, then we pick in addition an neighborhood
    $W'$ of $\xi$ in $P_{d,\Theta}$ such that
    $W' \subseteq U$.
    Moreover, since $P_{d,\Theta}(U'_+)$ is a subcomplex of $P_{d,\Theta}$,
    we can assume that
    $W' \cap P_{d,\Theta}(U'_+)$ is a deformation retract of $W'$.
    Now we can choose the interior of $P_{d,\Theta}(U'_+) \cup U'_+ \cup W'$
    for $U'$.
  \end{proof}

  Let $K$ be a simplicial complex and $\calu$ be a cover of a space $X$. 
  A map $f \colon K' \to X$ defined on a subcomplex $K'$ of $K$, 
  containing the $0$-skeleton $K^{(0)}$ of $K$,
  is said to be a \emph{partial $\calu$-realization} if 
  for every simplex $\sigma$ of $K$, there is a member of $\calu$
  that contains $f(K' \cap \sigma)$.
  If $K' = K$, then $f$ is called a \emph{full $\calu$-realization}. 
  
  \begin{lemma}
    \label{lem:extension}
    Let $\calu$ be an open cover of $\overline{P_{d,\Theta}}$ and $n \in \IN$.
    Then there exists an open cover $\calu'$ of 
    $\overline{P_{d,\Theta}}$ such that
    for any finite simplicial complex $K$ of dimension $\leq n$,
    any partial $i^{-1}\calu'$-realization of $K$ in $ P_{d,\Theta}$
    extends to a full $i^{-1}\calu$-realization $K \to P_{d,\Theta}$. 
  \end{lemma}

  \begin{proof}
    Using Lemma~\ref{lem:properties-of-overlineP_d,Theta}
       ~\ref{lem:overlineP:loc-contr} we find a sequence of
    successively smaller covers $\calu = \calu_n, \dots, \calu_0 = \calu'$
    such that for any $U' \in \calu_k$ there is $U \in \calu_{k+1}$
    such that any map $S^k \to i^{-1}(U')$ extends to $D^{k+1} \to i^{-1}(U)$.
    Inductively, any partial $i^{-1}\calu'$-realization  
    can then be extended to a $i^{-1}\calu$-realization $K \to P_{d,\Theta}$. 
  \end{proof}

  Let $\calu$ be a cover of a space $X$. 
  Maps $f, f' \colon Z \to X$ are said to be \emph{$\calu$-close} if 
  for any $x \in Z$ there is $U \in \calu$ containing  both $f(x)$ and $f'(x)$. 
  The \emph{mesh} of a cover of a metric space is the supremum of the
  diameter of its members.

  \begin{lemma}
    \label{lem:approximation}
    For any open cover $\calw$ of $\overline{P_{d,\Theta}}$ 
    and any map $f \colon K \to \overline{P_{d,\Theta}}$ defined on
    a finite complex there is a map
    $f' \colon K \to P_{d,\Theta}$ such that $f$ and $i \circ f'$ are
    $\calw$-close.
  \end{lemma}

  \begin{proof}
    Pick a metric $d_{\overline{P}}$ for $\overline{P_{d,\Theta}}$ and
    $\e > 0$ such that $3\e$ is a Lebesgue number for $\calw$.
    Let $\calu$ be a cover of mesh $\leq \e$.
    Pick $\calu'$ as in Lemma~\ref{lem:extension} with respect to
    $\calu$ and $n := \dim K$.
    Pick $\delta > 0$ such that $3\delta$ is a Lebesgue number for
    $\calu'$ and $\delta < \e$.
    Now subdivide $K$ until the diameter of
    the image of each simplex in $\overline{P_{d,\Theta}}$
    is at most $\delta$. 
    Since $i(P_{d,\Theta})$ is dense in $\overline{P_{d,\Theta}}$ we find
    $f'_0 \colon K^{(0)} \to P_{d,\Theta}$ 
    defined on the $0$-skeleton $K^{(0)}$ of $K$ such that 
    $d_{\overline{P}} (i \circ f'_0 (x),f(x)) \leq \delta$ for any
    $x \in K^{(0)}$.
    It follows that $f'_0$ is a partial $i^{-1}\calu'$ realization and 
    therefore extends to a full  
    $i^{-1}\calu$-realization $f' \colon K \to P_{d,\Theta}$.
    By construction for any $x \in K$, 
    $d_{\overline{P}}(f(x),f'(x)) \leq 2\delta +\e \leq 3\e$. 
    Thus $f$ and $i \circ f'$ are $\calw$-close.
  \end{proof}

  \begin{lemma}
    \label{lem:homotopy}
    Let $f \colon K \to  \overline{P_{d,\Theta}}$ be a map defined on
    a finite simplicial complex.
    Then there exists a homotopy 
    $H \colon K \x (0,1] \to P_{d,\Theta}$
    such that 
    \begin{equation*}
      (x,t) \mapsto \begin{cases} 
                        f(x) & t=0 \\ i( H(x,t) ) & t > 0 
                    \end{cases}  
    \end{equation*}
    is a continuous homotopy $K \x [0,1] \to \overline{P_{d,\Theta}}$.
  \end{lemma}

  \begin{proof}
    Let $d_{\overline{P}}$ be again a metric on $\overline{P_{d,\Theta}}$.
    Let $\calu_l$ be a sequence of open covers of
    $\overline{P_{d,\Theta}}$ with $\mesh \calu_l \to 0$.
    Let $\calu'_l$ be a cover as in Lemma~\ref{lem:extension}
    with respect to $\calu_l$ and $n := \dim K+1$.
    Let $\lambda_l$ be a Lebesgue number for $\calu_l$.
    Let $\e_l$ be a sequence such that $\e_l + \e_{l+1} \leq \lambda_l$
    for all $l$.
    Using Lemma~\ref{lem:approximation} we find maps 
    $f'_l \colon K \to P_{d,\Theta}$ such that $f$ and $i \circ f'_l$
    are $\e_l$-close, i.e.,
    $d_{\overline{P}}(f(x),i(f'_l(x))) < \e_l$ for all $x \in K$. 
    For $t = 1/l$ set $H(x,1/l) := f'_l(x)$.
    In order to define $H(x,t)$ for $1/l+1 < t < 1/l$, we can view
    $f'_{l+1} \coprod f'_l \colon K \x \{ 1/l+1,  1/l \} \to P_{d,\Theta}$
    as a partial $\calu'_l$-realization for an appropriate triangulation of
    $K \x [1/l+1,  1/l]$ and use Lemma~\ref{lem:extension}.
    Since $\mesh \calu_l \to 0$ the map $H$ defined this way
    has the desired continuity property.
  \end{proof}

  \begin{lemma}
    \label{lem:loc-contractible}
    The space $\overline{P_{d,\Theta}}$ is locally $n$-connected for each 
    $n$, i.e., for each $\xi \in \overline{P_{d,\Theta}}$ and each
    open neighborhood $U$ of $\xi$, there is a smaller neighborhood
    $U'$ of $\xi$ such that any map $f \colon S^k \to U'$ with 
    $k \leq n$ extends to a map $D^{k+1} \to U$.
  \end{lemma}

  \begin{proof}
    Pick $U'$ as in Lemma~\ref{lem:properties-of-overlineP_d,Theta}
      ~\ref{lem:overlineP:loc-contr}.
    Lemma~\ref{lem:homotopy} implies that any map
    $f \colon S^k \to U'$ is homotopic to a map
    of the form $i \circ f'$ with $f' \colon S^k \to i^{-1}(U')$.
    Now $f'$ extends to a map $D^{k+1} \to i^{-1}(U)$ by
    Lemma~\ref{lem:properties-of-overlineP_d,Theta}
      ~\ref{lem:overlineP:loc-contr} and this yields the desired 
    extension of $f$ as well.
  \end{proof}

  \noindent 
  Recall that we fixed $d$ and $\Theta$ such that  
  Proposition~\ref{prop:contract-P_d,Theta} applies.

  \begin{theorem}
    \label{thm:overlineP-is-ANR}
    The space $\overline{P_{d,\Theta}}$ is a contractible $\ANR$.
  \end{theorem}

  \begin{proof}
    {}From Lemmas~\ref{lem:properties-of-overlineP_d,Theta}
    and~\ref{lem:loc-contractible}
    we know that $\overline{P_{d,\Theta}}$ is compact, finite
    dimensional and locally $n$-connected for any $n$.
    Since this is a characterization of
    finite dimensional compact 
    $\ANR$s~\cite[Thm.~10.3, p.122]{Borsuk(Theory-of-retracts)},
    $\overline{P_{d,\Theta}}$ is an $\ANR$.
    Lemma~\ref{lem:homotopy} together with the contractibility 
    of $P_{d,\Theta}$ (Proposition~\ref{prop:contract-P_d,Theta})
    imply that $\overline{P_{d,\Theta}}$ is weakly contractible.
    Since finite dimensional $\ANR$s are retracts of 
    simplicial complexes~\cite[Thm.10.1, p.122]{Borsuk(Theory-of-retracts)} 
    the contractibility of $\overline{P_{d,\Theta}}$ follows.
  \end{proof}

  \begin{remark}
    Bestvina and Mess~\cite{Bestvina-Mess(1991)} proved 
    Theorem~\ref{thm:overlineP-is-ANR} for hyperbolic groups
    and the argument in this section is closely modeled on their
    argument. 
    For hyperbolic groups Bestvina and Mess showed moreover
    that the boundary is a $Z$-set in the analog of 
    $\overline{P_{d,\Theta}}$, but the methods of this section
    do not give this stronger statement.
    The reason for this is, that the contractions from
    Proposition~\ref{prop:contract-P_d,Theta} may not work
    in $\overline{P_{d,\Theta}} \setminus \Delta 
                 = P_{d,\Theta} \setminus V_\infty$.
    For example, if $\Gamma$ is a tree and if we use
    for $\Theta$ only the trivial angles $(e,e)$, 
    then $P_{d,\Theta} = \Gamma$.
    Nevertheless, it seems likely that in many
    cases $\Delta_+$ is a $Z$-set in $\overline{P_{d,\Theta}}$
    for suitable $d,\Theta$.
  \end{remark}

  \section{Extending open sets.}
     \label{app:extending}
   
  Here we review a convenient procedure to extend open subset to
  open subsets of an ambient space.
  Let $(X,d_X)$ be a metric space.
  Let $X_0 \subseteq X$ be a  subspace.
  For an open subset $U$ of $X_0$ we define
  $U^+ := \{ x \in X \mid d_X(x,U) < d_X(x,X_0 \setminus U) \}$.
  This construction has the following properties 
  (compare \cite[Lem.~4.14]{Bartels-Lueck(2012CAT(0)flow)}):
  $U^+$ is open; $X_0 \cap U^+ = U$;  
  $(U \cap V)^+ = U^+ \cap V^+$.
  We record the following two consequences.

  \begin{lemma}
    \label{lem:dd-U_+}
    For any open subset $U \subseteq X_0$, the open 
    subset $U^+ \subseteq X$ satisfies
    $U = X_0 \cap U^+$ and $\dd_0 U = X_0 \cap \dd U^+$ where
    $\dd_0$ is the boundary in $X_0$ and $\dd$ is the boundary 
    in $X$. 
  \end{lemma}

  \begin{proof}
    We already know that $U^+$ is open and satisfies $U = X_0 \cap U^+$.
    This implies $\dd_0 U \subseteq X_0 \cap \dd U  \subseteq X_0 \setminus U$.
    Let $V$ be the complement of the closure of $U$ in $X_0$.
    Then $V^+$ is open and $V^+ \cap U^+  = \emptyset$.
    Thus $V^+ \cap \dd U^+ = \emptyset$.
    Thus $X_0 \cap \dd U^+  \subseteq X_0 \setminus (U \cap V) = \dd_0 U$. 
  \end{proof}

  \noindent
  Suppose now that $G$ acts isometrically on $X$ and that $X_0 \subseteq X$
  is $G$-invariant.

  \begin{lemma}
    \label{lem:extend-cover}
    Let $\calf$ be a family of subgroups of $G$.
    Let $\calu$ be a $G$-invariant collection of $\calf$-subsets of $X_0$.
    Then $\calu^+ := \{ U^+ \mid U \in \calu \}$ is
    a $G$-invariant collection of $\calf$-subsets of $X$ extending
    the $U \in \calu$.
    The order of $\calu$ and $\calu^+$ agree. 
  \end{lemma}

  \begin{proof}
    We have $(gU)^+ = g(U^+)$ as the action is isometric.
    This implies that $\calu^+$ consists of
    $\calf$-subsets and that $\calu^+$ is $G$-invariant.
    The equation $(U \cap V)^+ = U^+ \cap V^+$ implies that the orders
    of $\calu$ and $\calu^+$ agree.
  \end{proof}

  \section{Ideal triangles are slim}
     \label{app:ideal-slim}
     
  \begin{lemma}
  	\label{lem:ideal-slim}
  	Let $\Gamma$ be a $\delta$-hyperbolic graph, i.e., all finite geodesic triangles are $\delta$-slim.
  	Then all geodesic triangles, including those with one or more corners on the boundary $\dd \Gamma$ are
  	$5\delta$-slim.  
  \end{lemma}
  
  \begin{proof}
  	Let $c$, $c'$ and $c''$ be the three sides of a geodesic triangles.
  	Write $\xi$, $\xi'$ and $\xi''$ for the three corners of the triangle, i.e.,
  	$c$ is geodesic between $\xi'$ and $\xi''$, $c'$ is a geodesic between $\xi$ and $\xi''$, and $c''$ is a geodesic between $\xi$ and $\xi'$.
  	We pick vertices $y \in c$, $y' \in c'$ and $y'' \in c''$,
  	We write $c_{\xi'}$ for the restriction of $c$ to a geodesic from $y$ to $\xi'$.
  	Similarly, we define $c_{\xi''}$, $c'_\xi$, $c'_{\xi''}$, $c''_{\xi}$ and $c''_{\xi'}$.
  	These six geodesics are asymptotic in three pairs.
  	Thus we find $A > 0$ such that the Hausdorff distances between $c'_{\xi}$ and $c''_{\xi}$, and between $c_{\xi'}$ and $c''_{\xi'}$, and between $c_{\xi''}$ and $c'_{\xi''}$ are all three $\leq A$.
  	 
  	Let now $v$ be a vertex on $c$.
  	We pick a geodesic $a$ of length at most $A$ that starts on $c'$ and ends on $c''$.
  	If $\xi \in \dd \Gamma$, then we can assume in addition that $d_\Gamma(v,a) \geq 100(A + \delta)$.
  	If $\xi$ is a vertex of $\Gamma$ then we assume that $a$ is the constant geodesic at $\xi$.
    We pick a vertex $x$ on $a$. 
  	Similarly, we pick geodesics $a'$ and $a''$ with vertices $x' \in a'$ and $x'' \in a''$.
  	We can assume that along $c$ the vertex $v$ is between $a' \cap c$ and $a'' \cap c$; if not then move $a'$ towards $\xi'$ and $a''$ towards $\xi''$.
  	We also pick geodesics $b$ from $x'$ to $x''$, $b'$ from $x$ to $x''$, and $b''$ from $x$ to $x'$.
  	
  	Consider now the geodesic $4$-gon whose four sides are $b$, and restrictions of $a'$, $c$ and $a''$.
  	Since $a'$ and $a''$ are either constant or far from $v$ we can apply hyperbolicity and find a vertex $w_1 \in b$ with $d_\Gamma(v,w_1) \leq 2 \delta$.
  	Next we use the geodesic triangle with sides $b$, $b'$ and $b''$.
  	By hyperbolicity there is $w_2 \in b' \cup b''$ with $d_\Gamma(w_1,w_2) \leq \delta$.
  	Without of loss of generality we assume $w_2 \in b'$.
  	Now we consider the geodesic $4$-gon whose four sides are $b'$ and restrictions of $a$, $c'$ and $a''$.   
  	Since $a$ and $a''$ are either constant or far from $v$ (and then also far from $w_2$) we can apply hyperbolicity and find a vertex $v' \in c'$ with $d_\Gamma(w_2,v') \leq 2\delta$.
  	Altogether, $d_\Gamma(v,v') \leq 5\delta$ and $v' \in c' \cup c''$.
  \end{proof}

\def\cprime{$'$} \def\polhk#1{\setbox0=\hbox{#1}{\ooalign{\hidewidth
  \lower1.5ex\hbox{`}\hidewidth\crcr\unhbox0}}}

\end{document}